\documentclass[10pt]{amsart}

\usepackage{hyperref}

\usepackage{amssymb}
\usepackage{amsmath}
\usepackage{xcolor}

\let\newpf\proof \let\proof\relax

\def\bm{\begin{matrix}}
\def\em{\end{matrix}}

\newcommand{\bt}{\begin{thm}}
\newcommand{\et}{\end{thm}}

\newcommand{\bl}{\begin{lemma}}
  \newcommand{\el}{\end{lemma}}

\newcommand{\beq}{\begin{eqnarray}}
\newcommand{\eeq}{\end{eqnarray}}

\def\be{\begin{equation}}
\def\ee{\end{equation}}

\def\ba{{\begin{align}}}
\def\ea{{\end{align}}}

\def\0{{\mathbf 0}}

\newtheorem{thm}{Theorem}[section]
\newtheorem{cor}[thm]{Corollary}

\newtheorem{lemma}[thm]{Lemma}

\newtheorem{conj}[thm]{Conjecture}
\theoremstyle{remark}
\newtheorem{rem}{Remark}[section]

\numberwithin{equation}{section}

\def \bn {\hfill \\ \smallskip\noindent}

\theoremstyle{definition}

\newtheorem{definition}{Definition}[section]
\def\proof{\bn {\bf Proof.} }

\def\note#1
{\marginpar

{\tiny $\leftarrow$
\par
\hfuzz=20pt \hbadness=9000 \hyphenpenalty=-100 \exhyphenpenalty=-100
\pretolerance=-1 \tolerance=9999 \doublehyphendemerits=-100000
\finalhyphendemerits=-100000 \baselineskip=6pt
#1}\hfuzz=1pt}

\newcommand{\dist}{\operatorname{dist}}

\newcommand{\Q}{{\mathbb Q}}
\newcommand{\R}{{\mathbb R}}
\newcommand{\T}{{\mathbb T}}

\newcommand{\Z}{{\mathbb Z}}

\def\B0{{\bold{0}}}

\newcommand{\lf}{\lfloor}
\newcommand{\rf}{\rfloor}

\catcode`\@=12

\def\Empty{}
\newcommand\oplabel[1]{
  \def\OpArg{#1} \ifx \OpArg\Empty {} \else
  	\label{#1}
  \fi}

\newcommand{\comm}[1]{}
\newcommand{\comment}[1]{}

\begin{document}

\title{Anti-resonances and sharp analysis of Maryland localization for all parameters}

\author{Rui Han, Svetlana Jitomirskaya and Fan Yang}



\thanks{}

\begin{abstract}
We develop the technique to prove localization through the analysis of
eigenfunctions in presence of both exponential frequency resonances
and exponential phase
barriers (anti-resonances) and use it to prove localization for the Maryland model for all parameters.

\end{abstract}

\maketitle

\section{Introduction}
As once noted by Ya. G. Sinai, localization is a game of
resonances. When restrictions to boxes that are not too far away from
each other have eigenvalues
that are too close, small denominators are created, making proofs of
localization always challenging and proofs of delocalization
occasionally possible (e.g. \cite{g,as,js,aw,jk,jl2}). Two types of
resonances have played a special role in the spectral theory of
quasiperiodic operators. Frequency resonances were first exploited in
\cite{as} based on \cite{g} to prove absence of eigenvalues (and
therefore singular continuous spetrum in the hyperbolic regime) for
quasiperiodic operators with Liouville frequencies. Phase resonances
were discovered in \cite{js} and used to prove absence of
eigenvalues for quasiperiodic operators with even potentials and
(arithmetically defined) generic frequencies. It was conjectured in
\cite{jcongr1} that for the almost Mathieu family- the
prototypical quasiperiodic operator - the two above types of
resonances are the only ones that appear and the competition between
the Lyapunov growth and combined exponential resonance strength resolves in
a sharp way. This was so far proved for single-type-resonances only:
for pure frequency resonances (that is for so-called
$\alpha$-Diophantine phases for which there are no exponential phase resonances) in \cite{ayz,jl1} and for pure phase
resonances (that is for Diophantine frequencies for which there are no
exponential frequency resonances) in \cite{jl2}. Papers \cite{jl1,jl2}
required developing sharp techniques for dealing with correspondingly
phase and frequency resonances. Both can survive, on the localization
side, adding weak resonances of the other type, but so far without the
desired sharpness. In this paper we succeed for the first
time  in
dealing in a sharp way with two different types of resonances. 

Our operators come not from the almost Mathieu, but from another popular
family, the Maryland model, a family of quasi-periodic Schr\"odinger operators
\begin{equation}\label{marylandmodel}
(H_{\lambda,\alpha,\theta}u)_n=u_{n+1}+u_{n-1}+\lambda\tan(\pi(\theta+n\alpha)u_n),
\end{equation}
where $\lambda>0$ is the coupling constant, irrational $\alpha\in \T=[0,1]$ is the frequency and $\theta\in \T$ is the phase.
We assume $\theta \notin \Theta=\{\frac{1}{2}+\alpha \Z+\Z\}$.

Maryland model was originally proposed by Grempel,
Fishman and Prange \cite{fgp} as a linear version of the quantum
kicked rotor and has attracted continuing interest from the physics community, see e.g. \cite{berry, fishman2010, GKDS2014}, since it serves as 
an exactly solvable example of the family of incommensurate models.

Frequency resonances are ubiquitous for all quasiperiodic potentials,
while phase resonances discussed above exist only for even sampling
functions, and thus not for the Maryland model. Indeed, as a result, for
Diophantine (i.e. non-resonant) frequencies it has localization for
{\it all} phases \cite{simm, JYMaryland}. However, it does have barriers, when
trajectory of a given phase approaches the singularity too
early. Barriers can compensate for the resonances, and therefore serve
as what we call here {\it anti-resonances}. They are precisely the
reason why for the Maryland model there are phases with localization
even for the most Liouville frequencies \cite{maryland}.  Thus
Maryland model features a combination of frequency resonances and
phase anti-resonances. Our main achievement here is the development of
a precise
understanding of barriers as phase anti-resonances, of how the competition
between them and  frequency resonances unfolds, as well as sharp analysis
of eigenfunction decay in presence of this combination.

B. Simon called the Maryland model a useful laboratory \cite{simm} because it is
exactly solvable in some sense, and thus can serve as a source of both general conjectures and counterexamples. It has explicit expression for the
Lyapunov exponent, integrated density of states, and even (somewhat implicit)  for
the eigenvalues and eigenfunctions. Thus certain features admit a more
direct analysis. Indeed it admits a very beautiful
trick noticed originally in \cite{fgp}: under the combination of Cayley and Fourier transforms it
leads to an explicit cohomological equation, and the analysis
becomes similar to that of Sarnak \cite{s}. \footnote{ \cite{s} was the paper where
the importance of the arithmetics in this type of spectral questions was
proposed even before \cite{as}}. Utilizing the Cayley transform, the
spectral decomposition for the Maryland model was determined fully,
for all $\alpha,\theta,$ in \cite{maryland}, making it the first - and
so far the only - model with spectral transitions
where this could be claimed. Previously, localization up to the sharp
threshold was proved for a.e. $\theta$ in \cite{simm}. The extension
of the analysis from a.e $\theta$ in \cite{simm} to {\it all} $\theta$ in
\cite{maryland} required  accounting the effect of the barriers, and Cayley
transform allowed to do it albeit in a highly implicit way.

Namely, let $p_n/q_n$ be the continued fraction approximants of
$\alpha$. The index $\beta(\alpha)$ that measures exponential strength
of the frequency resonances is defined as follows:
\begin{align}\label{def:beta}
\beta(\alpha)=\limsup_{n\to\infty}\frac{\ln q_{n+1}}{q_n}= \limsup_{n\to\infty}\frac{-\ln \|n\alpha\|}{n},
\end{align}
where $\|x\|_{\T}:=\text{dist} (x, \Z)$.

A new index, $\delta(\alpha,\theta)$ was introduced in \cite{maryland} as
\begin{align}\label{def:delta}
\delta(\alpha,\theta):=\limsup_{n\to \infty}\frac{\ln q_{n+1}+\ln
  \|q_n(\theta-\frac{1}{2})\|_{\T}}{q_n}.
\end{align}

With the Lyapunov exponent $L_\lambda(E)$ explicitly defined by
\eqref{LE} (and not dependent on $\alpha,\theta$), it was proved in \cite{maryland} that

\begin{thm}
$H_{\lambda,\alpha,\theta}$ has purely singular continuous spectrum on $\{E: L_{\lambda}(E)<\delta(\alpha,\theta)\}$, and pure point spectrum on $\{E: L_{\lambda}(E)>\delta(\alpha,\theta)\}$. 
\end{thm}

Thus Maryland model has sharp spectral transition defined by the
interplay between the Lyapunov exponent and $\delta(\alpha,\theta)$.

The index $\delta(\alpha,\theta)$ appeared naturally in the context of the cohomological
equation arising as a result of the Cayley and Fourier transforms. It is clear that $\delta(\alpha,\theta)\leq \beta(\alpha)$. 
Indeed $\delta(\alpha,\theta)=\beta(\alpha)$ holds for a.e.  (however
not all) $\theta$. 

Here we show another representation for $\delta(\alpha,\theta), $ see
Corollary \ref{cor:alternate_delta},
\begin{align}\label{def:another}
\delta(\alpha,\theta)=\limsup_{n\to \infty}\frac{\max(0,\ln q_{n+1}+\displaystyle\min_{k=0,\ldots,q_n-1}\ln
  \|\theta-\frac{1}{2}+k\alpha\|_{\T})}{q_n}.
\end{align}

Therefore, $\delta(\alpha,\theta)$ can be  interpreted  as the
exponential strength of
frequency resonances, $\beta(\alpha)$, tamed by   the
phase anti-resonances, defined as the positions of exponential near-zeros of the $\cos(\pi(\theta+k\alpha)).$ 

As mentioned, the proof of \cite{maryland} as well as all the ones prior to it, including
the original physics paper \cite{fgp}, have been based on a Cayley
transform and therefore indirect. Moreover, the eigenfunctions of the
Maryland model are, as a result of indirect analysis, known exactly,
yet the formulas don’t allow for easy conclusions about their
behavior, which is expected to be quite interesting, with transfer matrices
satisfying certain exact renormalization \cite{fs}. Also, Maryland eigenfunctions
are expected, through numerics, to have hierarchical structure driven
by the continued fraction expansion of the frequency. In \cite{JYMaryland} two
of the authors
developed a Green's function based approach to localization  for this
model, and used it to obtain Anderson localization for all $\theta$ and
Diophantine $\alpha$ or, in other words, for the case
$\beta(\alpha)=0$, that is in absence of frequency resonances. In this
paper we fully handle the difficult resonant case, where the anti-resonances
also start playing a crucial role. Our main conclusion is
\begin{thm}
\label{main0}
For any $\alpha\in \R\setminus \Q$ and any $\theta,$  
the spectrum on $\{E:L_{\lambda}(E)\ge\delta(\alpha, \theta)\}$  is
pure point and for any eigenvalue
$E\in \{L_{\lambda}(E)>\delta(\alpha, \theta)\}$ and any $\epsilon>0,$
the corresponding
eigenfunction $\phi_E$ satisfies $|\phi_E(k)|< e^{-
  (L_{\lambda}(E)-\delta(\alpha, \theta)-\epsilon)|k|}$ for
sufficiently large $|k|.$
\end{thm}

\begin{rem} We make a few remarks:
  \begin{enumerate}
    \item It is known that $\sigma_{pp}(H_{\lambda,\alpha,\theta})=
      \{E:L_{\lambda}(E)\ge\delta(\alpha, \theta)\}$ \cite{maryland}.
    We include the statement on pure point
     spectrum on  $ \{E:L_{\lambda}(E)\ge\delta(\alpha, \theta)\}$
     only to emphasize that it also 
         follows independently from our approach.
    \item In fact, we prove a much more precise local statement at
      each scale, see
      Lemma \ref{lem:main}.
      \item We therefore obtain sharp bounds on the decay of all
        eigenfunctions except for at most two values of $E$ where
        $L(E)= \delta(\alpha, \theta),$ which may or may not be eigenvalues \cite{maryland}.
       
       \end{enumerate}
     \end{rem}
     
Our main achievement however is the {\it approach} we develop here to treat
the ``resonance tamed by an anti-resonance'' situation. This paper is the first one in a
series of at least two as it paves the way to study full asymptotics
of the eigenfunctions. The result of Theorem \ref{main0} provides the
sharp upper envelope, but does not otherwise give insight into the
fine behavior of the eigenfunctions. However, we develop here the key
tools for such study, and the latter will be presented in the follow-up
work \cite{HJY}. In fact, some of our technical statements are
more detailed than needed for the purpose of our main results, because
we want to create the foundation for what will follow
in \cite{HJY}. Moreover, we expect this to lead to universal hierarchical
structure in the behavior of the eigenfunctions, identical to the one
discovered in \cite{jl1} for the almost Mathieu operator in case of 
absence of the anti-resonances (i.e. $\delta(\alpha, \theta)=\beta(\alpha)$), but a lot
more rich and complex in presence of the anti-resonances. We also
comment that Fourier transforms of the eigenfunctions represent
functions with natural boundaries on both circles bounding the annulus
of analiticity \cite{simm}, and our analysis promises to provide
various insights on their boundary behavior which is expected to be universal.

Moreover, while certain arguments we present depend on some specific
aspects of the Maryland model, the crucial part of the proof:
sharp analysis of the effect of  anti-resonances, is actually quite robust, and we
expect it to be useful in the study of other one-frequency
quasiperiodic models with 
unbounded potentials that have attracted attention recently
\cite{ilya1,ilya2,ilya3,gps} . Additionally, in the models that lead to
singular Jacobi matrices (that is where
the off-diagonal terms can approach zero) positions of off-diagonal
exponential near-zeros also compensate for resonant small divisors,
thus creating
effective anti-resonances. Such models have appeared lately in the
study of various graphene-type structures (e.g. \cite{graphene, AA}),
and in presence of certain anisotropy they also lead to models with
hyperbolicity where
frequency resonances coexist with phase anti-resonances. We expect
many parts of our
method 
to be applicable to all those scenarios.

In particular, a popular model, also originating in physics,  is the
extended Harper's model (EHM), introduced by Thouless \cite{thou}. It is a
family indexed by five parameters, given by $$
(H^{\mathrm{EHM}}_{\lambda, \alpha,\theta}u)_n=c_{\lambda}(\theta+n\alpha)u_{n+1}+\overline{c_{\lambda}(\theta+(n-1)\alpha)}u_{n-1}+2\cos{2\pi(\theta+n\alpha)}u_{n},
$$
where $\lambda =(\lambda_1,\lambda_2,\lambda-3),\;c_{\lambda}(\theta)=\lambda_1 e^{-2\pi i(\theta+\frac{\alpha}{2})}+\lambda_2+\lambda_3 e^{2\pi i(\theta+\frac{\alpha}{2})}$.
See \cite{JMreview} for a 2017 review  and \cite{dryten,AJM,nopoint,HJ,HYZ,Xin} for more recent results.

Extended Harper's model has a range of parameters where the
Lyapunov exponent is positive on the spectrum, while the corresponding
Jacobi matrix is singular, thus again becoming a fertile ground for
resonance/anti-resonance analysis. However, in this case there is an
additional feature: phase resonances, thus the situation is even more complicated. 
In fact, our analysis prompts us to formulate the following conjecture. 

Let $L(\lambda)$ be the Lyapunov exponent of the EHM on the spectrum (it
is known exactly and only depends on $\lambda$, see \cite{Jm}).
Let the exponent $\gamma(\alpha,\theta)$ be defined by 
$$\gamma(\alpha,\theta):=\limsup -\frac{\ln \|2\theta+n\alpha\|}{|n|}.$$
It characterizes the exponential strength of phase resonances.
Let $\mathcal{R}_1:=\{\lambda: 0<\max(\lambda_1+\lambda_3, \lambda_2)<1\}$ be the positive Lyapunov exponent regime of the EHM, and 
$\mathcal{R}_s:=\{\lambda: \lambda_1=\lambda_3\geq
\frac{\lambda_2}{2}, \text{ or } \lambda_1+\lambda_3=\lambda_2\}$ be
the singular regime: that is where $c_{\lambda}(\theta)=0$ for some $\theta\in \T$. 
\begin{conj}
For $\lambda \in \mathcal{R}_1\cap \mathcal{R}_s$, $H^{\mathrm{EHM}}_{\lambda, \alpha, \theta}$ has purely singular continuous spectrum if $L(\lambda)<\tilde{\delta}(\alpha,\theta)+\gamma(\alpha,\theta)$, and pure point spectrum if $L(\lambda)>\tilde{\delta}(\alpha,\theta)+\gamma(\alpha,\theta)$, where 
\begin{align}\label{def:tildedet}
\tilde{\delta}(\alpha,\theta)=\limsup_{n\to \infty} \frac{\ln
  q_{n+1}+\sum_{\theta ':c_{\lambda}(\theta ')=0}\ln \|q_n(\theta-\theta
  ')\|}{q_n},
\end{align}
where zeros are counted with multiplicities.
We note that $c_{\lambda}$ has either one or two (possibly coinciding)
zeros in the
indicated regime.
.
\end{conj}
\begin{rem}
The index $\tilde{\delta}$ index was introduced to account for
anti-resonances in the proofs of {\it singular continuous spectrum}
in \cite{JYsingular, HJ} for, correspondingly, general
operators with unbounded potentials/singular Jacobi matrices (for
unbounded potentials the sum is instead over the singularities). 
In particular purely singular continuous spectrum was proved for the extended Harper's model whenever $L(\lambda)<\tilde{\delta}(\alpha,\theta)$ \cite{HJ}.
\end{rem}

\begin{rem}
As we were finalizing this paper we learned of a preprint \cite{liu} by Liu, where he proved localization for
  the almost Mathieu operator with
  completely resonant phases (when $2\theta\in \Z\alpha$, hence
  $\beta(\alpha)=\gamma(\alpha,\theta)$) up to the conjectured
  threshold $\{\ln \lambda>2\beta(\alpha)\}$, which improves on earlier results \cite{liuyuan2,liu2}. It is another remarkable
  case of sharp analysis in a situation of two coexisting types of
  resonances: in that case, phase and frequency.
  It is interesting to see whether the techniques of \cite{liu} can be
  combined with ours to prove the corresponding result for completely
  resonant phases for the extended Harper's
  model, thus  localization for
  $L>\tilde{\delta}(\alpha,\theta)+\beta(\alpha)$. Here, however, the analysis
  would require studying the interaction of {\it three} coexisting types of
  resonances. 
\end{rem}

Finally, our proof is local, thus potentially allowing also for the
analysis of the behavior of generalized eigenfunctions corresponding
to the singular continuous spectrum regime.

We now briefly comment on our argument. Proofs of arithmetic
localization in the spirit of \cite{j0, j, jks, AJ1,liuyuan2,  liuyuan, jl1,jl2}\footnote{See
  \cite{pcmi,icm} for  recent reviews}  have to deal with the competition between the
hyperbolicity of the transfer matrices and exponential strength of the
resonances. A sharp way to resolve this competition for pure frequency
resonances has been developed in \cite{jl1}.  We start with following its main
framework combined with the strategy of \cite{JYMaryland} and dealing
with technical complications arising from the unboundedness and  lack of
continuity. However, this alone only brings us to the same conclusion as in
\cite{jl1}, that is localization in the regime
$\{L_{\lambda}(E)>\beta(\alpha)\}$. 
The region that needs completely new ideas is
$\{\delta(\alpha,\theta)<L(E)\leq \beta(\alpha)\}$. That's what we
develop here, exploiting the unbounded nature of the potentials rather
than circumventing it, by showing how a properly understood  anti-resonance creates
additional decay of the Green's function. This helps us to
lower the threshold down all the way to the sharp
$\delta(\alpha,\theta)$.

This paper is organized as follows. in Sec. \ref{details} we present
more detail on the general strategy and difficulties of the proof. In
Sec. \ref{Sec:pre} we collect some preliminary results; in
Sec. \ref{Sec:delta} we locate the  minimum values of (the absolute
value of) $q_n$ consequent cosines at $\{m_n+\ell q_n\}_{\ell}$, and
give a characterization of $\delta(\alpha,\theta)$ using these minimum
values; in Sec. \ref{Sec:proof} we discuss the proof of our main
Theorem \ref{main0} and reduce it to the main Lemma \ref{lem:main}. We
present some standard uniformity results in Sec. \ref{Sec:uni}; our
key estimate for the numerators is presented in Sec. \ref{Sec:key}; the proof of Lemma \ref{lem:main} is presented in Sec. \ref{Sec:loc} with preparations in Sections \ref{Sec:C1} and \ref{Sec:C2} addressing non-resonant $m_n$ and resonant $m_n$ respectively.

\section{Strategy and difficulties}\label{details}

We first  introduce some notations and recall the key framework,
slightly modified from the one
developed in \cite{jl1}, also with adaptions from \cite{JYMaryland}.

Let $\tau>0$ be a small constant. For large $n$, let $b_n=[\tau q_n]$ and $R_{\ell}:=[\ell q_n-b_n, \ell q_n+b_n]$ be the resonant regimes and also $r_{\ell}:=\sup_{y\in R_{\ell}} |\phi(y)|$.
Let us also write $L_{\lambda}(E)$ as $L$ and $L-\ln 2$ as $\tilde{L}$.
We want to prove the generalized eigenfunction $\phi$ decays exponentially (with a positive decay rate independent of $n$) on $[q_n/3, q_{n+1}/3]$.
To do so, first we show that (roughly) at each non-resonance $|\phi|$
can be dominated by $|\phi|$ at its two nearby resonant regimes.
This allows us to only focus on the relations between $r_{\ell}$'s. 
For each $y\in r_{\ell}$, we want to expand $\phi(y)$ using the Green's formula \eqref{Green_tildeP}
\begin{align*}
|\phi(y)|\leq &\frac{|\tilde{P}_{x_2-y}(\theta+(y+1)\alpha)|}{|\tilde{P}_{2q_n-1}(\theta+x_1\alpha)|}
\prod_{j=x_1}^{y}|\cos(\pi(\theta+j\alpha))|\cdot |\phi(x_1-1)|\\
&+\frac{|\tilde{P}_{y-x_1}(\theta+x_1\alpha)|}{|\tilde{P}_{2q_n-1}(\theta+x_1\alpha)|}\prod_{j=y}^{x_2}|\cos(\pi(\theta+j\alpha))|\cdot |\phi(x_2+1)|,
\end{align*} 
with a nicely placed interval $D=[x_1, x_2] \ni y$ satisfying
$|D|=2q_n-1$ and such that 
$$\frac{1}{4} |D|<y-x_1< \frac{3}{4} |D|.$$ 
Let $I_{\ell}=[\ell q_n-3q_n/2, \ell q_n-q_n/2]$ be the collection of potential left-end points $x_1$ in the Green's formula. 
The goal is to obtain a good lower bound of $|\tilde{P}_{2q_n-1}(\theta+x_1\alpha)|$ for certain $x_1\in I_{\ell}$. 
This is done by showing: 

1. $\ln |\tilde{P}_{2q_n-1}(\theta)|$ has an average lower bound $\tilde{L}=L-\ln 2$,

2. $\tilde{P}_{2q_n-1}(\theta)$ is a polynomial in $\tan(\theta)$ of degree $2q_n-1$

Hence the Lagrange interpolation formula tells us $|\tilde{P}_{2q_n-1}(\theta)|$ can not be simultaneously small at $2q_n$ well distributed point $\theta$'s.
One can show that $\{\theta+j\alpha\}_{j\in I_0\cup I_{\ell}}$ are $\eta/(2q_n)$-uniform (in the sense of \eqref{defuniform}) with $\eta=\ln(q_{n+1}/|\ell|)$.
This implies there exists a certain $x_1\in I_{\ell}$ ($x_1$ can not be in $I_0$ because that leads to a contradiction) such that 
$$|\tilde{P}_{2q_n-1}(\theta_{x_1})|\gtrsim \frac{|\ell|}{q_{n+1}} e^{2q_n L}.$$
Plugging this lower bound into the Green's formula,
using standard control of the $|\tilde{P}_k|\lesssim e^{\tilde{L}|k|}$ by the Lyapunov exponent $\tilde{L}$ in the numerators, 
and adapting the estimates for the (possibly) non-resonant $|\phi(x_1-1)|$ and $|\phi(x_2+1)|$, 
one can show that
\begin{align}\label{eq:intro:L>beta}
r_{\ell}\lesssim e^{-(L-\beta_n) q_n} \max(r_{\ell-1}, r_{\ell+1}),
\end{align}
where $\beta_n:=(\ln q_{n+1})/q_n$.
When $L>\beta$, combining the inequality above with our argument in Section \ref{Sec:loc} already yields exponential decay.
However this does not work for $\delta<L<\beta$, so
one has to look for an additional decay to break the $\beta$ barrier.

It turns out this additional decay in some simple cases comes in
handy, directly from the product (indeed the minimum) of cosines in the Green's formula. 
The minimum values of (the absolute values of) cosines can be located
at points of the form $m_n+\ell q_n$ for certain $m_n,$  and the minimum values are roughly of the size $\exp((\delta-\beta_n)q_n)$.
In those simple cases, the product of cosines contain (at least) one of the minimum values thus bringing 
the decay in \eqref{eq:intro:L>beta} down to
$$r_{\ell}\lesssim e^{-(L-\delta) q_n} \max(r_{\ell-1}, r_{\ell+1}),$$
which is just enough.
However, there are difficult cases when the product of cosines does
not contain such a minimum value, see e.g. Case 1 of Section \ref{Sec:C2}. 
Then the question is: where does the additional decay hide in these cases?
Tackling this question is the main breakthrough of this paper that has
made the full analysis possible.  the key ideas are presented in Section \ref{Sec:key}.

\section{Preliminaries}\label{Sec:pre}
We will from now use $\|\theta\|$ for $\|\theta\|_{\T}$ for simplicity.
For $x\in \R$, let $[x]$ be the largest integer that is less than or equal to $x$.
For a fixed $\theta$, let $\theta_k:=\theta+k\alpha$.
We adopt the convention that various large (or small) constants (e.g. $N(\varepsilon)$) may change
their exact values even within the same inequality.

\subsection{Continued fractions}
Let $[a_1,a_2,...]=\alpha$ be the continued fraction expansion of $\alpha$.
For $k\geq 1$, let $p_k/q_k:=[a_1,a_2,...,a_k]$ be the continued fraction approximants to $\alpha$. The following properties hold
\begin{align}\label{eq:cont1}
\|q_{k-1}\alpha\|=\min_{1\leq n<q_k} \|n\alpha\|,
\end{align}
\begin{align}\label{eq:cont2}
\frac{1}{2q_{k+1}}\leq \|q_k\alpha\|\leq \frac{1}{q_{k+1}},
\end{align}
\begin{align}\label{eq:cont3}
q_{k+1}=a_{k+1}q_{k}+q_{k-1},
\end{align}
and
\begin{align}\label{eq:cont4}
\|q_{k-1}\alpha\|=a_{k+1}\|q_k\alpha\|+\|q_{k+1}\alpha\|.
\end{align}

A key technical lemma is the following.
\begin{lemma} \label{lana}\cite{AJ1}
Let $\alpha\in \R\setminus \Q $,\ $\theta\in\R$ and $0\leq j_0 \leq q_{n}-1$ be such that 
$$\mid \cos \pi(\theta+j_{0}\alpha)\mid = \inf_{0\leq j \leq q_{n}-1} \mid \cos \pi(\theta+j\alpha)\mid ,$$
then for some absolute constant $C>0$,
$$-C\ln q_{n} \leq \sum_{j=0,j\neq j_0}^{q_{n}-1} \ln \mid \cos \pi (\theta+j\alpha) \mid+(q_{n}-1)\ln2 \leq C\ln q_n$$
\end{lemma}

\subsection{Solution and Green's function}
Let $G_{[x_1,x_2]}(x,y)=(H_{[x_1, x_2]}-E)^{-1}(x,y)$ be the Green's function, where $H_{[x_1, x_2]}$ is the operator $H_{\lambda,\alpha,\theta}$ restricted to the interval $[x_1,x_2]$.

Let $\phi$ be a solution to $H\phi=E\phi$, let $[x_1, x_2]$ be an interval containing $y$, then we have
\begin{align}\label{Green}
\phi(y)=G_{[x_1, x_2]}(x_1, y)\phi(x_1-1)+G_{[x_1, x_2]}(x_2, y)\phi(x_2+1).
\end{align}

\subsection{Cocycles}
Consider the equation $H\phi=E\phi$. Let
\begin{align}\label{transferA}
A(\theta, E)=
\left(
\begin{matrix}
E-\lambda\tan{\pi \theta}\ &-1\\
1                          &0
\end{matrix}
\right).
\end{align}
Then any solution can be reconstructed via the following relation
\begin{align*}
\left(
\begin{matrix}
\phi(k+1)\\
\phi(k)
\end{matrix}
\right)
=
A(\theta+k\alpha, E)
\left(
\begin{matrix}
\phi(k)\\
\phi(k-1)
\end{matrix}
\right).
\end{align*}
If we iterate this process, we get
\begin{align*}
\left(
\begin{matrix}
\phi(k)\\
\phi(k-1)
\end{matrix}
\right)
=
A_k(\theta, E)
\left(
\begin{matrix}
\phi(0)\\
\phi(-1)
\end{matrix}
\right),
\end{align*}
where
\begin{align*}
\left\lbrace
\begin{matrix}
A_k(\theta, E)=A(\theta+(k-1)\alpha, E)\cdots A(\theta+\alpha, E)A(\theta, E)\ \mathrm{for}\ k\geq 1,\\
A_0(\theta, E)=\mathrm{Id},\\
A_k(\theta, E)=(A_{-k}(\theta+k\alpha, E))^{-1}\ \mathrm{for}\ k\leq -1.
\end{matrix}
\right.
\end{align*}
Note that the cocycle $A(\theta, E)$ is actually singular because it contains $\tan{\pi\theta}$. Sometimes it is more convenient for us to work with non-singular cocycles. Let us denote 
\begin{align}\label{defnonsingular}
F(\theta, E)=\cos{\pi \theta}\cdot A(\theta, E)=
\left(
\begin{matrix}
E\cos{\pi \theta}-\lambda\sin{\pi \theta}\ &-\cos{\pi \theta}\\
\cos{\pi \theta}                          &0
\end{matrix}
\right).
\end{align}

\subsection{Lyapunov exponent}
Let $L(\alpha, A(\theta, E))$ be the Lyapunov exponent of the Maryland model, it is defined as follows
\begin{align}\label{defLE}
L(\alpha, A(\theta, E))=\lim_{k\rightarrow \infty}\frac{1}{k}\int_{\T}\ln{\|A_k(\theta, E)\|} \mathrm{d}\theta.
\end{align}
It was shown in \cite{fgp} that $L(\alpha, A(\theta, E))$ depends only on $\lambda$ and $E$ (hence we denote it by $L_{\lambda}(E)$) and is uniquely determined by the following equation
\begin{align}\label{LE}
e^{L_{\lambda}(E)}+e^{-L_{\lambda}(E)}=\frac{\sqrt{(2+E)^2+\lambda^2}+\sqrt{(2-E)^2+\lambda^2}}{2}.
\end{align}
Let us also denote $\tilde{L}(E)=L(\alpha, F(\theta, E))$, then by (\ref{defnonsingular}) we have
\begin{align}\label{LEtildeLE}
\tilde{L}_{\lambda}(E)=L_{\lambda}(E)-\ln{2}.
\end{align}

From this point on, we shall write $L_{\lambda}(E)$ as $L(E)$ or $L$, and $\tilde{L}_{\lambda}(E)$ as $\tilde{L}(E)$ or $\tilde{L}$ for simplicity.

\subsection{General upper bounds of transfer matrices}
\begin{lemma}[e.g. \cite{Furman}]\label{upperbounds}
Let $(a, D)$ be a continuous cocycle, then for any $\epsilon>0$, there exists a constant $C(\alpha, D, \varepsilon)$ such that for any $k\in \Z$,
\begin{align}
\|D_k(\theta)\|\leq C(\alpha, D, \varepsilon) e^{|k|(L(\alpha, D)+\epsilon)}\ \mathrm{for}\ \mathrm{any}\ \theta\in\T.
\end{align}
\end{lemma}

As a corollary we have the following lemma which will be used many times throughout the paper.
\begin{cor}\label{cor:prod_cos}
Let $\ell_2\geq \ell_1$. We have
\begin{align*}
\prod_{\ell=\ell_1}^{\ell_2}|\cos(\pi(\theta+\ell\alpha))|\leq C(\varepsilon) e^{(\ell_2-\ell_1)(-\ln{2}+\varepsilon)} \inf_{j=\ell_1}^{\ell_2}|\cos(\pi(\theta+j\alpha))|,
\end{align*}
where $C(\varepsilon)$ is a constant that depends only on $\varepsilon$.
\end{cor}

\subsection{A closer look at the transfer matrix}
If we consider the Schr\"odinger cocycle $(\alpha, A(\theta, E))$, it turns out $A_k(\theta, E)$ has the following expression
\begin{equation}\label{PinA}
A_k(\theta,E)=
\left(
\begin{array}{cc}
P_k(\theta,E) & -P_{k-1}(\theta+\alpha,E) \\
P_{k-1}(\theta,E) & -P_{k-2}(\theta+\alpha,E)
\end{array}
\right),
\end{equation}
where
\begin{align}
P_k(\theta,E)=&\det{[(E-H_{\theta})|_{[0,k-1]}]}    \notag \\
=&
\det{ 
\left[\begin{array}{cccccc}
   E-\lambda\tan{\pi\theta}                         &                 -1                                           &             & &    \\
                    -1                                          & E-\lambda\tan\pi(\theta+\alpha)             & -1                    \\
                                                                 &-1                                                             & \cdots   \\
                                                                  &                                                               &             & \cdots  &  -1 \\
                                                                  &                                                               &             & -1        & E-\lambda\tan\pi(\theta+(k-1)\alpha)
\end{array}\right]_{k\times k}
}
\end{align}
Let $\tilde{P}_k(\theta, E)=\prod_{j=0}^{k-1} \cos \pi (\theta+j \alpha)\cdot P_k(\theta, E)$.  Then clearly
\begin{align}\label{tildePinF}
F_k(\theta, E)=
\left(
\begin{array}{cc}
\tilde{P}_k(\theta,E) & -\tilde{P}_{k-1}(\theta+\alpha,E)\cos{\pi\theta} \\
\tilde{P}_{k-1}(\theta,E)\cos{\pi (\theta+(k-1)\alpha)} & -\tilde{P}_{k-2}(\theta+\alpha,E)\cos{\pi\theta}\cos{\pi(\theta+(k-1)\alpha)}
\end{array}
\right).
\end{align}
By the fact that $F$ is continous and (\ref{upperbounds}), (\ref{tildePinF}) we have the following control of $\tilde{P}_k$.
\begin{lemma}\label{lem:upperbddtildeP}
For any $\varepsilon>0$ there exists constant $C(\alpha, E, \lambda, \varepsilon)>0$ such that for any $k\in \Z$,
\begin{align}
|\tilde{P}_k(\theta, E)|\leq C(\alpha, E, \lambda, \varepsilon) e^{(\tilde{L}(E)+\varepsilon)|k|}\ \mathrm{for}\ \mathrm{any}\ \theta\in \T.
\end{align}
\end{lemma}

There is the following connection between the determinants $P_k$ and Green's functions:
\begin{align}\label{PkG}
|G_{[x_1, x_2]}(x_1, y)|=&\frac{|P_{x_2-y}(\theta_{y+1})|}{|P_{x_2-x_1+1}(\theta_{x_1})|}
=\frac{|\tilde{P}_{x_2-y}(\theta_{y+1})|}{|\tilde{P}_{x_2-x_1+1}(\theta_{x_1})|} \prod_{k=x_1}^y |\cos(\pi\theta_k)|\\
|G_{[x_1, x_2]}(x_2, y)|=&\frac{|P_{y-x_1}(\theta_{x_1})|}{|P_{x_2-x_1+1}(\theta_{x_1})|}
=\frac{|\tilde{P}_{y-x_1}(\theta_{x_1})|}{|\tilde{P}_{x_2-x_1+1}(\theta_{x_1})|}\prod_{k=y}^{x_2} |\cos(\pi\theta_k)|\notag
\end{align}
As a consequence, from \eqref{Green} we can deduce
\begin{align}\label{Green_tildeP}
|\phi(y)|\leq \frac{|\tilde{P}_{x_2-y}(\theta_{y+1})|}{|\tilde{P}_{x_2-x_1+1}(\theta_{x_1})|} \prod_{k=x_1}^y |\cos(\pi\theta_k)| 
\cdot |\phi(x_1-1)|+ \frac{|\tilde{P}_{y-x_1}(\theta_{x_1})|}{|\tilde{P}_{x_2-x_1+1}(\theta_{x_1})|}\prod_{k=y}^{x_2} |\cos(\pi\theta_k)|\cdot |\phi(x_2+1)|.
\end{align}

\subsection{Writing $\tilde{P}_k$ as a polynomial of $\tan{\pi\theta}$, see \cite{JYMaryland}}\

$\tilde{P}_k(\theta)/{\cos^k{\pi\theta}}$ 
can be expressed as a polynomial of degree $k$ in $\tan\pi\theta$, namely,
\begin{equation}\label{Phpolyg}
\frac{\tilde{P}_k(\theta)}{(\cos{\pi\theta})^k}  =: g_k(\tan\pi\theta).
\end{equation}
By Lagrange interpolation formula,
\begin{equation*}
g_k(\tan\pi \theta)=\sum_{j=0}^{k}g_k(\tan \pi \theta_j) \frac{\prod_{\ell \neq j}\tan \pi \theta-\tan\pi\theta_{\ell}}{\prod_{\ell\neq j}\tan\pi\theta_j-\tan\pi\theta_{\ell}}.
\end{equation*}
Thus 
\begin{align}
\tilde{P}_k(\theta)=(\cos\pi \theta)^k g_k(\tan\pi \theta)
&=\sum_{j=0}^{k}\tilde{P}_k(\theta_j)\frac{\prod_{\ell \neq j}\tan \pi \theta-\tan\pi\theta_{\ell}}{\prod_{\ell\neq j}\tan\pi\theta_j-\tan\pi\theta_{\ell}}\cdot  \frac{\cos^k {\pi \theta}}{\cos^k{\pi\theta_j}} \notag\\
&=\sum_{j=0}^{k}\tilde{P}_k(\theta_j)\prod_{\ell\neq j} \frac{\sin\pi(\theta-\theta_{\ell})}{\sin\pi(\theta_j-\theta_{\ell})}. \label{tildePlagrange}
\end{align}

\subsection{Average lower bound of $\tilde{P}_k$}
The following is Lemma 3.1 of \cite{JYMaryland}.
\begin{lemma}\label{subharmonic}
By Herman's subharmonic trick, one has
\begin{align}\label{averagelower}
\frac{1}{k}\int_{0}^{1}\ln |\tilde{P}_k(\theta)| \mathrm{d}\theta =\frac{1}{k} \int_{0}^{1} \ln |\tilde{P}_k(2\theta)| \mathrm{d}\theta \geq L-\ln2=\tilde{L}
\end{align}
\end{lemma}

\subsection{Uniformity}
\begin{definition}
We say that the set $\{\theta_1,..., \theta_{k+1}\}$ are $\gamma$-uniform if
\begin{align}\label{defuniform}
\max_{\theta\in [0, 1]}\ \max_{j=1,..., k+1} \prod_{\ell\neq j}\ \frac{|\sin\pi(\theta-\theta_{\ell})|}{|\sin\pi(\theta_j-\theta_{\ell})|}<e^{\gamma k}.
\end{align}
\end{definition}

\section{Resonance and non-resonances}\label{Sec:uni}
Choose a value (from multiple possible values) of $\tau_n$ such that
$$
\tau_n\in \left(\frac{\varepsilon}{2\max(L, 1)}, \frac{\varepsilon}{\max(L, 1)} \right],
$$ 
and $\tau_n q_n\in \Z$.
Define $b_n=\tau_n q_n$. 
For any $y\in \Z$ we call $y$ {\it resonant} (at the scale of $q_n$) if $\dist(y, q_n\Z)\leq b_n$, otherwise we call $y$ {\it non-resonant}.

\subsection{Non-resonances: uniformity}
For a non-resonant $y$, let $n_0$ be the least positive integer so that 
\begin{align*}
2q_{n-n_0}\leq \dist(y, q_n\Z).
\end{align*}
Once $n_0$ is chosen, we can fix $s$ be the greatest positive integer such that 
\begin{equation}
2sq_{n-n_0}\leq \dist(y, q_n\Z).
\end{equation} Clearly, 
Let 
\begin{align}\label{def:Iy}
\tilde{I}_0&=[-[sq_{n-n_0}/2]-sq_{n-n_0},\, -[sq_{n-n_0}/2]-1 ]\cap \Z, \notag\\
\tilde{I}_y&=[y-[sq_{n-n_0}/2]-sq_{n-n_0},\, y-[sq_{n-n_0}/2]-1] \cap \Z.
\end{align}
Clearly $\tilde{I}_0\cup \tilde{I}_y$ contains $2sq_{n-n_0}$ distinct numbers.
Let us also note that by our choice of $n_0$, we have
\begin{align}\label{eq:bn<yn-n0+1}
b_n< \dist(y, q_n \Z)<2q_{n-n_0+1}
\end{align}
and also 
\begin{align}\label{sq-n0<qn-n0+1}
sq_{n-n_0}<q_{n-n_0+1}.
\end{align}

The following lemma is the consequence of a variant of Lemma 9.10 of \cite{AJ1}. We will include its proof in the appendix for completeness.
\begin{lemma}\label{lem:nonres_uni}
For a non-resonant $y$, for $n>N(\varepsilon)$ large enough, we have $\{\theta_{\ell}\}_{\ell\in \tilde{I}_0\cup \tilde{I}_y}$ are $\varepsilon$-uniform.
\end{lemma}
Combining this with a standard argument in the literature, one can show
\begin{lemma}\label{lem:nonres_I2_large}
For $n>N(\varepsilon)$ large enough, there exists $x_1\in \tilde{I}_y$ so that 
$$|\tilde{P}_{2sq_{n-n_0}-1}(\theta_{x_1})|\geq e^{(\tilde{L}-2\varepsilon)(2sq_{n-n_0}-1)}.$$
\end{lemma}
We will also include its proof in the appendix.

\subsection{Resonances: uniformity}
For $\ell\in \Z$, let $I_{\ell}$ be defined below
\begin{align}\label{def:1_Ia}
I_\ell:=&[(\ell-1)q_n -\lf q_n/2\rf, \ell q_n-\lf q_n/2\rf-1]\cap \Z.
\end{align}

\begin{lemma}\label{lem:1_res_uni}
For $\ell $ such that $0<|\ell |\leq 2q_{n+1}/(3q_n)$, 
$\{\theta_j\}_{j\in I_0\cup I_{\ell}}$ are $\frac{\ln{(q_{n+1}/|\ell |)}}{2q_n-1}+\epsilon$-uniform.
\end{lemma}
This is a variant of Theorem B.5. of \cite{jl1}. We include the proof in the appendix.

\begin{cor}\label{cor:1_res_uni}
For $\ell $ such that $0<|\ell |\leq 2q_{n+1}/(3q_n)$, there exists $x_1\in I_0\cup I_{\ell}$ such that 
$$|\tilde{P}_{2q_n-1}(\theta_{x_1})|\geq \frac{|\ell|}{q_{n+1}} e^{(\tilde{L}-2\varepsilon)(2q_n-1)}.$$
\end{cor}
\begin{proof}
Suppose otherwise, we have for any $x_1\in I_0\cup I_{\ell}$, 
$$|\tilde{P}_{2q_n-1}(\theta_{x_1})|< \frac{|\ell|}{q_{n+1}} e^{(\tilde{L}-2\varepsilon)(2q_n-1)}.$$
By \eqref{tildePlagrange}
\begin{align*}
\tilde{P}_{2q_n-1}(\theta)=\sum_{x_1\in I_0\cup I_{\ell}}\tilde{P}_{2q_n-1}(\theta_{x_1})\prod_{\substack{j\in I_0\cup I_{\ell}\\ j\neq x_1}} \frac{\sin\pi(\theta-\theta_j)}{\sin\pi(\theta_{x_1}-\theta_j)}.
\end{align*}
Combining this with Lemma \ref{lem:1_res_uni} yields, uniformly in $\theta$,
\begin{align*}
|\tilde{P}_{2q_n-1}(\theta)|\leq 2q_n e^{(\tilde{L}-\varepsilon)(2q_n-1)}<e^{(\tilde{L}-\frac{\varepsilon}{2})(2q_n-1)}.
\end{align*}
Hence contradiction with \eqref{averagelower}. \qed
\end{proof}

\section{Characterization of the index $\delta$}\label{Sec:delta}
From this point on, we shall write $\delta(\alpha,\theta)$ as $\delta$ and $\beta(\alpha)$ as $\beta$ for simplicity.

Fix a small $\varepsilon>0$ such that $L>\delta+700\varepsilon$.

It is evident from the definition of $\delta$ that it is never greater than $\beta$. 
The following lemma shows $\delta$ is always non-negative.
\begin{lemma}\label{lem:delta_geq_0}
We have $0\leq \delta\leq \beta$ for all $\alpha, \theta$.
\end{lemma}
\begin{proof}
Recall the definition 
$$\delta=\limsup_{n\to \infty}\frac{\ln{q_{n+1}}+\ln{\|q_n(\theta-\frac{1}{2})\|}}{q_n}=
\limsup_{n\to \infty}\frac{\ln{q_{n+1}}+\ln{\|q_n|\theta-\frac{1}{2}|\|}}{q_n}.$$
Suppose $\delta<0$, then there exists $c>0$, such that for $n$ large enough, we have
\begin{align*}
\frac{\ln{q_{n+1}}+\ln{\|q_n|\theta-\frac{1}{2}|\|}}{q_n}<-c.
\end{align*}
This implies 
\begin{align}\label{eq:qn_theta-1/2}
\|q_n|\theta-\frac{1}{2}|\|<\frac{e^{-c q_n}}{q_{n+1}}.
\end{align}
Let $\{\tilde{p}_k/\tilde{q}_k\}$ be the continued fraction approximants to $|\theta-\frac{1}{2}|$.
Take $k$ large, and $n$ such that
\begin{align}\label{eq:qn_qk+1_qn+1}
q_n<\tilde{q}_{k+1}\leq q_{n+1}.
\end{align}
By \eqref{eq:qn_theta-1/2}, \eqref{eq:cont1} and \eqref{eq:cont2}, we have
$$\frac{1}{2\tilde{q}_{k+1}}\leq \|\tilde{q}_k |\theta-\frac{1}{2}|\|\leq \|q_n |\theta-\frac{1}{2}|\|<\frac{e^{-c q_n}}{q_{n+1}}<\frac{1}{2q_{n+1}}.$$
This implies 
$$\tilde{q}_{k+1}>q_{n+1},$$
which contradicts with \eqref{eq:qn_qk+1_qn+1}.
\qed
\end{proof}

Define 
\begin{align}\label{def:betan}
\beta_n:=\frac{\ln q_{n+1}}{q_n},
\end{align}
and
\begin{align}\label{def:deltan}
\delta_n:=\frac{\ln \|q_n(\theta-\frac{1}{2})\|-\ln \|q_n\alpha\|}{q_n}.
\end{align}
Since $\limsup_{n\to \infty}\delta_n=\delta$, we have that for $n>N(\varepsilon)$ large enough, 
\begin{align}\label{eq:L>deltan+100}
L>\delta_n+650\varepsilon.
\end{align}
It is also clear by \eqref{eq:cont2} that
\begin{align}\label{eq:deltan<betan}
\delta_n\leq \frac{\ln(1/2)-\ln \|q_n\alpha\|}{q_n}<\beta_n.
\end{align}
By Lemma \ref{lem:delta_geq_0}, we also have
\begin{align}\label{eq:lim_max_0_deltan}
\delta=\limsup \max(0, \delta_n).
\end{align}

First we will characterize the minimal values of $\{|\cos\pi(\theta+j\alpha)|\}_j$ on scale $q_n$.
\begin{definition}\label{def:minimal}
We say $(m,\ell)\in \Z^2$ is {\it $\theta$-minimal on scale $q_n$} if the following holds
\begin{enumerate}
\item $m\in [-q_n/2, q_n/2)$
\item $|\ell|\leq \frac{1}{q_n}(e^{\delta_n q_n}+q_n+\frac{1}{2})$,
\item $\|\theta-\frac{1}{2}+(m+\ell q_n)\|<(\frac{1}{2}+\frac{1}{2q_n})\|q_n\alpha\|$,
\item (i). If $a_{n+1}\geq 4$, we have 
$$\|\theta-\frac{1}{2}+(m+jq_n)\alpha\|\leq 20 \min_{|k|<q_n} \|\theta-\frac{1}{2}+(m+j q_n+k)\alpha\|,$$ 
holds for any $|j|\leq a_{n+1}/6$.

(ii). If $a_{n+1}\leq 3$, we have
$$\|\theta-\frac{1}{2}+m\alpha\|\leq 20\min_{-q_n/2\leq k<q_n/2}\|\theta-\frac{1}{2}+k\alpha\|.$$
\end{enumerate}
\end{definition}

Next we show that the existence of $\theta$-minimal $(m,\ell)$.
\begin{lemma}\label{lem:mn}
For any $q_n$ sufficiently large, there exists $\theta$-minimal $(m_n,\ell_n)$ at scale $q_n$.
\end{lemma}
\begin{rem}
For any given $\theta$, following the procedure below, one can construct $(m_n,\ell_n)$ explicitly.
\end{rem}
\begin{proof}
By the definition of $\delta_n$, we have
\begin{align*}
\|q_n(\theta-\frac{1}{2})\|=e^{\delta_n q_n}\|q_n\alpha\|.
\end{align*}
Using the fact that $\|q_n\alpha\|>1/(q_{n+1}+q_n)$ and $\|q_n(\theta-1/2)\|\leq 1/2$, we have
\begin{align}\label{eq:delta_n0}
e^{\delta_n q_n}\leq \frac{q_{n+1}+q_n-1}{2}.
\end{align}
One can choose $j_0\in \Z$ such that
\begin{align}\label{eq:delta_n1}
\|q_n(\theta-\frac{1}{2})+j_0 q_n\alpha\|\leq \frac{1}{2}\|q_n\alpha\|,
\end{align}
and $j_0$ satisfies
\begin{align}\label{eq:delta_n2}
|j_0|\leq 
e^{\delta_n q_n}+\frac{1}{2}.
\end{align}
Note that \eqref{eq:delta_n1} yields
\begin{align}
\|q_n(\theta-\frac{1}{2}+j_0\alpha)\|\leq \frac{1}{2}\|q_n\alpha\|.
\end{align}
Hence there exists an integer $p$ such that
\begin{align}\label{eq:delta_n3}
|(\theta-\frac{1}{2}+j_0\alpha)-\frac{p}{q_n}|\leq \frac{1}{2q_n}\|q_n\alpha\|.
\end{align}
Since $p_n, q_n$ are coprime, there exists 
\begin{align}\label{eq:delta_n3'}
j_1\in [-\frac{q_n}{2}, \frac{q_n}{2})
\end{align} 
such that
\begin{align*}
j_1p_n\equiv p \text{ (mod } q_n\text{)}.
\end{align*}
This implies
\begin{align}\label{eq:delta_n4}
j_1\alpha-\frac{p}{q_n}=j_1\frac{p_n}{q_n}-\frac{p}{q_n}+j_1(\alpha-\frac{p_n}{q_n})=k+j_1(\alpha-\frac{p_n}{q_n}),
\end{align}
where $k\in \Z$ and 
\begin{align}\label{eq:delta_n5}
|j_1(\alpha-\frac{p_n}{q_n})|=\frac{j_1}{q_n}\|q_n\alpha\|\leq \frac{1}{2}\|q_n\alpha\|.
\end{align}
Combining \eqref{eq:delta_n3}, \eqref{eq:delta_n4} and \eqref{eq:delta_n5}, we have
\begin{align}\label{eq:delta_n6}
\|\theta-\frac{1}{2}+(j_0+j_1)\alpha\|\leq (\frac{1}{2}+\frac{1}{2q_n})\|q_n\alpha\|.
\end{align}
Define $k_n:=j_0+j_1$, then clearly by \eqref{eq:delta_n2} and \eqref{eq:delta_n3'} we have
\begin{align}\label{eq:delta_n7}
|k_n|\leq e^{\delta_n q_n}+\frac{q_n+1}{2}.
\end{align}

Define $m_n^{(1)}\in [-q_n/2,q_n/2)$ and $\ell_n^{(1)}\in \Z$ be such that
\begin{align}\label{def:mn_elln}
m_n^{(1)}+\ell_n^{(1)} q_n=k_n.
\end{align}
By \eqref{eq:delta_n7} and \eqref{eq:delta_n0}, we have
\begin{align}\label{eq:delta_n9'}
|\ell_n^{(1)}|\leq \frac{1}{q_n}(e^{\delta_n q_n}+q_n+\frac{1}{2})\leq \frac{q_{n+1}+3q_n}{2q_n}\leq \left[\frac{a_{n+1}+3}{2}\right].
\end{align}
By \eqref{eq:delta_n6} and \eqref{def:mn_elln}, we have
\begin{align}\label{eq:delta_n6'}
\|\theta-\frac{1}{2}+(m_n^{(1)}+\ell_n^{(1)} q_n)\|\leq (\frac{1}{2}+\frac{1}{2q_n})\|q_n\alpha\|.
\end{align}
We also have, by \eqref{eq:delta_n9'}, that for $\ell\neq \ell_n^{(1)}$ such that $|\ell|\leq [a_{n+1}/6]$ that
\begin{align}\label{eq:delta_n10}
|\ell_n^{(1)}-\ell| \cdot \|q_n\alpha\|\leq \left(\left[\frac{a_{n+1}+3}{2}\right]+\left[\frac{a_{n+1}}{6}\right] \right) \|q_n\alpha\|\leq \frac{2}{3}\|q_{n-1}\alpha\|+\frac{3}{2}\|q_n\alpha\|<\frac{1}{2}.
\end{align}
Hence
\begin{align*}
\|(\ell_n^{(1)}-\ell)q_n\alpha\|=|\ell_n^{(1)}-\ell|\cdot \|q_n\alpha\|.
\end{align*}
This implies
\begin{align}\label{eq:delta_n11}
(|\ell_n^{(1)}-\ell|-\frac{1}{2}-\frac{1}{2q_n})\|q_n\alpha\|\leq \|\theta-\frac{1}{2}+(m_n^{(1)}+\ell q_n)\alpha\|
\leq &(|\ell_n^{(1)}-\ell|+\frac{1}{2}+\frac{1}{2q_n}) \|q_n\alpha\|.
\end{align}
Combining the right hand side of \eqref{eq:delta_n11} with \eqref{eq:delta_n10}, we have
\begin{align*}
\|\theta-\frac{1}{2}+(m_n^{(1)}+\ell q_n)\alpha\|
\leq &\left(\left[\frac{a_{n+1}+3}{2}\right]+\left[\frac{a_{n+1}}{6}\right]+\frac{1}{2}+\frac{1}{2q_n}\right)\|q_n\alpha\|\\
\leq &\begin{cases}
(\frac{2}{3}a_{n+1}+2+\frac{1}{2q_n})\|q_n\alpha\|, \text{ if } a_{n+1}\geq 7\\
(a_{n+1}-\frac{1}{2}+\frac{1}{2q_n})\|q_n\alpha\|, \text{ if } 4\leq a_{n+1}\leq 6.
\end{cases}
\end{align*}
Hence for integer $|k|<q_n$ and $a_{n+1}\geq 7$, we have
\begin{align}\label{eq:delta_n12}
\|\theta-\frac{1}{2}+(m_n^{(1)}+\ell q_n+k)\alpha\|
\geq &\|k\alpha\|-\|\theta-\frac{1}{2}+(m_n^{(1)}+\ell q_n)\alpha\| \notag\\
\geq &\|q_{n-1}\alpha\|-\|\theta-\frac{1}{2}+(m_n^{(1)}+\ell q_n)\alpha\| \notag\\
\geq & a_{n+1}\|q_n\alpha\|-\|\theta-\frac{1}{2}+(m_n^{(1)}+\ell q_n)\alpha\| \notag\\
\geq &(\frac{1}{3}a_{n+1}-2-\frac{1}{2q_n}) \|q_n\alpha\| \notag\\
\geq &\frac{1}{20}\|\theta-\frac{1}{2}+(m_n^{(1)}+\ell q_n)\alpha\|,
\end{align}
and similarly for $4\leq a_{n+1}\leq 6$, 
\begin{align}\label{eq:delta_n13}
\|\theta-\frac{1}{2}+(m_n^{(1)}+\ell q_n+k)\alpha\|
\geq &a_{n+1}\|q_n\alpha\|-(a_{n+1}-\frac{1}{2}+\frac{1}{2q_n})\|q_n\alpha\| \notag\\
\geq &\frac{1}{20}\|\theta-\frac{1}{2}+(m_n^{(1)}+\ell q_n)\alpha\|.
\end{align}
Thus combining \eqref{eq:delta_n9'}, \eqref{eq:delta_n6'}, \eqref{eq:delta_n12}, \eqref{eq:delta_n13}, we have proved that $(m_n, \ell_n)=(m_n^{(1)}, \ell_n^{(1)})$ is $\theta$-minimal at scale $q_n$ if $a_{n+1}\geq 4$. 

Next we consider the case $a_{n+1}\in [1, 3]$. 

If $\ell_n^{(1)}=0$, we have by \eqref{eq:delta_n6'} that for any $|k|<q_n$,
\begin{align*}
\|\theta-\frac{1}{2}+(m_n^{(1)}+k)\alpha\|
\geq &\|q_{n-1}\alpha\|-\|\theta-\frac{1}{2}+m_n^{(1)}\alpha\|\\
\geq &(\frac{1}{2}-\frac{1}{2q_n})\|q_n\alpha\|\\
\geq &\frac{1}{2}\|\theta-\frac{1}{2}+m_n^{(1)}\alpha\|.
\end{align*}
This verifies (4)(ii) in Definition \ref{def:minimal}. Hence $(m_n,\ell_n)=(m_n^{(1)},\ell_n^{(1)})$ is $\theta$-minimal at scale $q_n$.

If $\ell_n^{(1)}\neq 0$, the minimal may not occur at $m_n$. We define $\tilde{m}_n\in [-q_n/2, q_n/2)$ such that
\begin{align}\label{def:tildemn}
\|\theta-\frac{1}{2}+\tilde{m}_n\alpha\|=\inf_{k\in [-q_n/2, q_n/2)} \|\theta-\frac{1}{2}+k\alpha\|.
\end{align}
Next we divide into two different cases depending on $\|\theta-\frac{1}{2}+\tilde{m}_n\alpha\|<\frac{1}{3}\|q_n\alpha\|$ or not.

Case 1. If $\|\theta-\frac{1}{2}+\tilde{m}_n\alpha\|\geq \frac{1}{3}\|q_n\alpha\|$.
We have
\begin{align*}
\|\theta-\frac{1}{2}+m_n^{(1)}\alpha\|\leq (|\ell_n^{(1)}|+\frac{1}{2}+\frac{1}{2q_n})\|q_n\alpha\|\leq 4\|q_n\alpha\|\leq 12\|\theta-\frac{1}{2}+\tilde{m}_n\alpha\|,
\end{align*}
which verifies (4)(ii) in Definition \ref{def:minimal}. Hence $(m_n,\ell_n)=(m_n^{(1)}, \ell_n^{(1)})$ is $\theta$-minimal at scale $q_n$.

Case 2. If 
\begin{align}\label{eq:tildemn<13}
\|\theta-\frac{1}{2}+\tilde{m}_n\alpha\|< \frac{1}{3}\|q_n\alpha\|.
\end{align}
Such $\tilde{m}_n$ may not exist in some cases, but if it exists, $(m_n,\ell_n)=(\tilde{m}_n, 0)$ will be $\theta$-minimal at scale $q_n$.
Indeed, (4)(ii) obviously hold for $m=\tilde{m}_n$ due to the definition of $\tilde{m}_n$. 
(3) of Definition \ref{def:minimal} holds for $(m,\ell)=(\tilde{m}_n,0)$ due to \eqref{eq:tildemn<13}.
Hence $(\tilde{m}_n, 0)$ is $\theta$-minimal at scale $q_n$. \qed
\end{proof}

Define 
\begin{align}\label{def:cnell}
c_{n,\ell}:=|\cos(\pi \theta_{m_n+\ell q_n})|.
\end{align}
As a corollary of the construction in Lemma \ref{lem:mn} we have.
\begin{cor}\label{cor:mn}
If $a_{n+1}\geq 4$, we have
\begin{align}\label{eq:cnelln}
c_{n,\ell_n}\leq \frac{2}{3}\|q_n\alpha\|,
\end{align} 
and for $|\ell|\leq q_{n+1}/(6q_n)$
\begin{align}\label{eq:delta_n_min_cos}
c_{n,\ell}\leq 21 \min_{|k|<q_n} |\cos(\pi\theta_{m_n+\ell_n q_n+k})|,
\end{align}
and further if $\ell\neq \ell_n$,
\begin{align}\label{eq:cnell_low_up}
c|\ell-\ell_n|\cdot \|q_n\alpha\|\leq c_{n,\ell}\leq C|\ell-\ell_n|\cdot \|q_n\alpha\|,
\end{align}
where $c,C\in (1/3,2)$ are two absolute constants.

If $a_{n+1}\leq 3$, we have that 
\begin{align}\label{eq:cnell_min_an+1<4}
c_{n,0}\leq 21 \min_{k\in [-q_n/2,q_n/2)} |\cos(\pi\theta_k)|.
\end{align}
\end{cor}
\begin{proof}
\eqref{eq:delta_n_min_cos} and \eqref{eq:cnell_min_an+1<4} follow from (4) of Definition \ref{def:minimal},
and \eqref{eq:cnell_low_up} follows from \eqref{eq:delta_n11}.
\qed
\end{proof}

We also have the following corollary of the construction of Lemma \ref{lem:mn}
\begin{cor}\label{cor:elln>0}
If $e^{\delta_n q_n}>3q_n$, we have
\begin{align*}
\left| |\ell_n|-\frac{1}{q_n} e^{\delta_n q_n}\right|\leq \frac{1}{2q_n}+1.
\end{align*}
\end{cor}
\begin{proof}
This follows from an inspection of \eqref{eq:delta_n2}, \eqref{eq:delta_n3'}, \eqref{eq:delta_n7} and \eqref{eq:delta_n9'}.
\end{proof}
Combining \eqref{eq:lim_max_0_deltan} with Corollaries \ref{cor:mn} and \ref{cor:elln>0}, we have
\begin{cor}\label{cor:alternate_delta}
\begin{align*}
\delta(\alpha,\theta)=\limsup_{n\to \infty}\max(0,\delta_n)=\limsup_{n\to \infty} \max(0, \frac{\ln q_{n+1}+\ln |c_{n,0}|}{q_n}).
\end{align*}
\end{cor}
\begin{proof}
Let 
$$\delta_n':=\frac{\ln q_{n+1}+\ln |c_{n,0}|}{q_n}.$$
If $a_{n+1}\geq 4$, and $e^{\delta_n q_n}\leq 3q_n$, we have by \eqref{eq:delta_n9'} and \eqref{eq:cnell_low_up} that
\begin{align*}
\delta_n'\leq \frac{\ln q_{n+1}+\ln (C|\ell_n| \|q_n\alpha\|)}{q_n}\leq \frac{\ln (5C)}{q_n}.
\end{align*}
This implies
\begin{align}\label{eq:d-d'_1}
|\max(0, \delta_n)-\max(0, \delta_n')|\leq \frac{\ln(3q_n)}{q_n}.
\end{align}
If $e^{\delta_n q_n}>3q_n$, we have by Corollary \ref{cor:elln>0} that
\begin{align*}
\delta_n'\leq \frac{\ln(e^{\delta_n q_n}+\frac{1}{2}+q_n)}{q_n}-\frac{\ln q_n}{q_n}<\delta_n+\frac{\ln(2/q_n)}{q_n},
\end{align*}
and
\begin{align*}
\delta_n'\geq \frac{\ln(e^{\delta_n q_n}-\frac{1}{2}-q_n)}{q_n}-\frac{\ln q_n}{q_n}<\delta_n-\frac{\ln(2q_n)}{q_n}.
\end{align*}
Hence 
\begin{align}\label{eq:d-d'_2}
|\delta_n-\delta_n'|\leq \frac{\ln(2q_n)}{q_n}.
\end{align}
If $1\leq a_{n+1}\leq 3$, we have by \eqref{eq:deltan<betan} that
\begin{align}\label{eq:d-d'_3}
\delta_n<\beta_n=\frac{\ln q_{n+1}}{q_n}<\frac{\ln (4q_n)}{q_n},
\end{align}
and by (3) of Definition \ref{def:minimal} we have
\begin{align}\label{eq:d-d'_4}
\delta_n'=\frac{\ln q_{n+1}+\ln |\sin(\pi(\theta-\frac{1}{2}+m_n\alpha))|}{q_n}\leq \frac{\ln q_{n+1}+\ln (\pi \|q_n\alpha\|)}{q_n}<\frac{\ln \pi}{q_n}.
\end{align}
Combining \eqref{eq:d-d'_3} with \eqref{eq:d-d'_4}, we have
\begin{align}\label{eq:d-d'_5}
|\max(0,\delta_n)-\max(0,\delta_n')|\leq \frac{\ln (4q_n)}{q_n}.
\end{align}
Corollary \ref{cor:alternate_delta} follows from combining \eqref{eq:d-d'_1}, \eqref{eq:d-d'_2} and \eqref{eq:d-d'_5}.
\end{proof}

Let us also note that if $\beta_n\geq \delta_n+200\varepsilon$, we use the following estimate, obtained from \eqref{eq:delta_n9'},
\begin{align}\label{eq:delta_n9}
|\ell_n|\leq \frac{1}{q_n}(e^{\delta_n q_n}+\frac{2q_n+1}{2})\leq 2\max(e^{\delta_n q_n},1),
\end{align}
which implies, after combining with \eqref{eq:cnell_low_up}, that for $|\ell|\leq q_{n+1}/(6q_n)$ and some absolute constant $0<C<8$,
\begin{align}\label{eq:cnell_final}
c_{n,\ell}\leq C\max(|\ell|, e^{\delta_n q_n}, 1) e^{-\beta_nq_n}.
\end{align}

As a consequence of \eqref{eq:delta_n_min_cos} and Lemma \ref{lem:upperbddtildeP} and Corollary \ref{cor:prod_cos}, we have the following corollaries. Since the proofs are very standard, we leave all of them to the appendix.
\begin{cor}\label{cor:cos_prod_lower}
Let $I=[\ell_1, \ell_2]\subset \Z$ be such that there exists $j\in \Z$, $|j|<q_{n+1}/(6q_n)$, that satisfies 
$$I\subset [m_n+j q_n+1, m_n+(j+1)q_n-1].$$
Then for $n>N(\varepsilon)$ large enough, we have
\begin{align*}
\prod_{\ell\in I}|\cos(\pi\theta_{\ell})|\geq e^{-\varepsilon(2q_n-|I|)} e^{-(\ln 2)|I|}
\end{align*}
\end{cor}

Combining the above Corollary with Lemma \ref{upperbounds}, we have
\begin{cor}\label{cor:A_upper}
Let $I=[\ell_1, \ell_2]\subset \Z$ be such that there exists $j\in \Z$, $|j|<q_{n+1}/(6q_n)$, that satisfies 
$$I\subset [m_n+j q_n+1, m_n+(j+1)q_n-1].$$
Then for $n>N(\varepsilon)$ large enough, we have
\begin{align*}
\|A_{|I|}(\theta_{\ell_1})\|\leq e^{3\varepsilon q_n} e^{L |I|}.
\end{align*}
\end{cor}

And further, we have
\begin{cor}\label{cor:A_upper_mn}
Let $I=[\ell_1, \ell_2]\subset \Z$ be such that $\ell_1\in [(j-1)q_n+m_n+1, jq_n+m_n-1]$ and $\ell_2\in [jq_n+m_n+1, (j+1)q_n+m_n-1]$, for some $j\in \Z$, $|j|<q_{n+1}/(6q_n)$. For $n>N(\varepsilon)$ large enough we have
\begin{align*}
\|A_{|I|}(\theta_{\ell_1})\|\leq e^{7\varepsilon q_n} \frac{1}{c_{n,j}} e^{L|I|}.
\end{align*}
\end{cor}

Next lemma allows us to locate $\tilde{m}_n$.
\begin{lemma}
Let $a_{n+1}\in \{1,2,3\}$.
Suppose there exists $\tilde{m}_n\in [-q_n/2, q_n/2)$ satisfying \eqref{def:tildemn} and \eqref{eq:tildemn<13}, we have $\tilde{m}_n\in \mathcal{A}$, where
\begin{align*}
\mathcal{A}:=
\begin{cases}
\emptyset, \text{ if } a_{n+1}\in \{2, 3\} \text{ and } |\ell_n|\leq a_{n+1}-1\\
\{m_n-b q_{n-1}\}, \text{ if } a_{n+1}\in \{2, 3\} \text{ and } \ell_n=b a_{n+1}, \text{ where } b=\pm 1\\
\{m_n-b q_{n-1}\}, \text{ if } a_{n+1}=1, \ell_n=b, \text{ where } b=\pm 1\\
\{m_n+b(q_n-q_{n-1}), \text{ if } a_{n+1}=1, \ell_n=2b, a_{n+2}\geq 2, \text{ where } b=\pm 1\\
\{m_n-b q_{n-1}, m_n+b (q_n-q_{n-1})\}, \text{ if } a_{n+1}=a_{n+2}=1, \ell_n=2b, \text{ where } b=\pm 1
\end{cases}.
\end{align*}
\end{lemma}
We leave the proof of this lemma in the appendix.
\begin{proof}
We have by \eqref{eq:delta_n6'} and \eqref{eq:tildemn<13} that
\begin{align}\label{eq:delta_n14}
\|(m_n+\ell_n q_n-\tilde{m}_n)\alpha\|\leq \|\theta-\frac{1}{2}+(m_n+\ell_n q_n)\alpha\|+\|\theta-\frac{1}{2}+\tilde{m}_n\alpha\|<\|q_n\alpha\|.
\end{align}
We further divide into $a_{n+1}=1$ and $a_{n+1}=2,3$.

Case 1. If $a_{n+1}=2$ or $3$, we have by \eqref{eq:delta_n9'} that $|\ell_n|\leq a_{n+1}$. 
Hence
\begin{align}\label{eq:delta_n15}
0<|m_n+\ell_n q_n-\tilde{m}_n|<a_{n+1}q_n+q_n<q_{n+1}+q_n.
\end{align}
If further $|\ell_n|\leq a_{n+1}-1$, we have
\begin{align*}
0<|m_n+\ell_n q_n-\tilde{m}_n|<a_{n+1}q_n<q_{n+1},
\end{align*}
which implies 
\begin{align*}
\|(m_n+\ell_n q_n-\tilde{m}_n)\alpha\|\geq \|q_n\alpha\|,
\end{align*}
contradicting \eqref{eq:delta_n14}.
Hence $\tilde{m}_n$ can not exist if $0<|\ell_n|\leq a_{n+1}-1$.

If $|\ell_n|=a_{n+1}$, combining \eqref{eq:delta_n14} and \eqref{eq:delta_n15}, we must have
\begin{align*}
|m_n+\ell_n q_n-\tilde{m}_n|=q_{n+1},
\end{align*}
yielding
\begin{align*}
\tilde{m}_n=
\begin{cases}
m_n-q_{n-1}, \text{ if } \ell_n=a_{n+1}\\
m_n+q_{n-1}, \text{ if } \ell_n=-a_{n+1}
\end{cases}.
\end{align*}

Case 2. If $a_{n+1}=1$, we have by \eqref{eq:delta_n9'} that $|\ell_n|\leq 2$.

If $|\ell_n|=1$, we have
\begin{align*}
0<|m_n+\ell_n q_n-\tilde{m}_n|<2q_n<q_{n+2}.
\end{align*}
Combining this with \eqref{eq:delta_n14}, we must have
\begin{align*}
|m_n+\ell_n q_n-\tilde{m}_n|=q_{n+1}=q_n+q_{n-1},
\end{align*}
which implies
\begin{align*}
\tilde{m}_n=
\begin{cases}
m_n-q_{n-1}, \text{ if } \ell_n=1\\
m_n+q_{n-1}, \text{ if } \ell_n=-1
\end{cases}.
\end{align*} 

If $|\ell_n|=2$ and $a_{n+2}\geq 2$, we have
\begin{align*}
0<|m_n+\ell_n q_n-\tilde{m}_n|<3q_n<q_{n+2},
\end{align*}
Combining this with \eqref{eq:delta_n14}, we must have
\begin{align*}
|m_n+\ell_n q_n-\tilde{m}_n|=q_{n+1}=q_n+q_{n-1},
\end{align*}
which implies
\begin{align*}
\tilde{m}_n=
\begin{cases}
m_n+q_n-q_{n-1}, \text{ if } \ell_n=2\\
m_n-q_n+q_{n-1}, \text{ if } \ell_n=-2
\end{cases}.
\end{align*}

If $|\ell_n|=2$ and $a_{n+2}=1$, we have
\begin{align*}
0<|m_n+\ell_n q_n-\tilde{m}_n|<3q_n<q_{n+3}.
\end{align*}
Combining this with \eqref{eq:delta_n14}, we must have
\begin{align*}
|m_n+\ell_n q_n-\tilde{m}_n|\in \{q_{n+1}, q_{n+2}\}.
\end{align*}
Hence 
\begin{align*}
\tilde{m}_n=
\begin{cases}
m_n-q_{n-1} \text{ or } m_n+q_n-q_{n-1}, \text{ if } \ell_n=2\\
m_n+q_{n-1} \text{ or } m_n-q_n+q_{n-1}, \text{ if } \ell_n=-2
\end{cases}.
\end{align*}
\qed
\end{proof}


\section{Proof of Theorem \ref{main0}}\label{Sec:proof}

We first normalize the generalized eigenfunction $\phi$ such that
$|\phi(0)|^2+|\phi(-1)|^2=2$ (we normalize the right-hand-side to $2$ for simplicity).

For convenience we will assume 
\begin{align}\label{assume:phi0=1}
|\phi(0)|\geq 1
\end{align}
and 
\begin{align}\label{Shnol}
|\phi(k)|\leq C_0 |k|.
\end{align}
If $|\phi(0)|<1$ and $|\phi(-1)|>1$, we simply need to replace the $|\phi(0)|$ on the left-hand-sides of \eqref{eq:Green_at_0}, \eqref{phi0_2}, \eqref{phi0_3} and \eqref{eq:P2sq_I1_2} with $|\phi(-1)|$ and slightly adjust the right-hand-sides accordingly.

Theorem \ref{main0} is a consequence of the following lemma.
\begin{lemma}\label{lem:main}
Let $\phi$ be an eigenfunction satisfying $|\phi(0)|\geq 1$ and \eqref{Shnol}. Then for $n>N(\alpha, E, \lambda, \varepsilon, C_0)$ large enough and $\frac{1}{12}q_n\leq |k|<\frac{1}{12}q_{n+1}$, we have
\begin{align*}
|\phi(k)|\leq e^{-(L-\delta_n-650\varepsilon)|k|}.
\end{align*}
\end{lemma}
The remaining of the paper will be devoted to the proof of Lemma \ref{lem:main}, dividing into the following three cases.

{\it Case 1. $0\leq \delta_n\leq \beta_n\leq 300\varepsilon$.}
This case is essentially the Diophantine case that is handled in \cite{JYMaryland}. 
We include a brief proof following \cite{JYMaryland} in the appendix.

{\it Case 2. $300\varepsilon\leq \beta_n\leq \delta_n+200\varepsilon$.}
This is an intermediate case. In this case we have $L>\beta_n+200\varepsilon$. One can combine the strategy in \cite{jl1} (which proves Anderson localization for the almost Mathieu operator in the case of $L>\beta$) with that in \cite{JYMaryland} to handle this case. Compared to the Case 3 below, Case 2 has a lot of simplifications, and in particular does not require our key estimates in Section \ref{Sec:key} at all. 

{\it Case 3. $\beta_n>\max(\delta_n+200\varepsilon, 300\varepsilon)$.}
This is our main achievement of the paper.

From this point on, we shall omit the dependence of the parameters on $\alpha, E, \lambda$ and only emphasize on the dependence on $\varepsilon$.
We shall also assume $\beta_n>300\varepsilon$, since the $\beta_n\leq 300\varepsilon$ case is only proved in the appendix.


\section{Key technical lemmas}\label{Sec:key}
The following lemmas on $\tilde{P}_{q_n-1}$ are the key to prove Anderson localization in the sharp regime $L>\delta$.
These lemmas reveal that large potential values $|\tan(\pi\theta_{m_n+\ell q_n})|$ combined with Shnol's theorem \eqref{Shnol} yield improved upper bounds, roughly speaking with an additional $e^{(\delta_n-\beta_n)q_n}$ decay, for $|\tilde{P}_{q_n-1}(\theta_{m_n+\ell q_n+1})|$.

\begin{lemma}\label{lem:Pqn_mn}
For $n>N(\varepsilon)$ large enough we have
\begin{align*}
|\tilde{P}_{q_n-1}(\theta+(m_n-q_n+1)\alpha)| \leq C_0 e^{q_n (\tilde{L}+\delta_n-\beta_n+3\varepsilon)}.
\end{align*}
\end{lemma}

\begin{proof}
Let $I:=[m_n-q_n+1,m_n-1]$.
By Green's formula \eqref{Green_tildeP} and assumption \eqref{assume:phi0=1}, we have
\begin{align}\label{eq:Green_at_0}
1\leq |\phi(0)|
\leq &\frac{|\tilde{P}_{m_n-1}(\theta_1)|}{|\tilde{P}_{q_n-1}(\theta_{m_n-q_n+1})|} \prod_{j=m_n-q_n+1}^0|\cos(\pi \theta_j)| 
\cdot |\phi(m_n-q_n)|\\
&+\frac{|\tilde{P}_{-m_n+q_n-1}(\theta_{m_n-q_n+1})|}{|\tilde{P}_{q_n-1}(\theta_{m_n-q_n+1})|}\prod_{j=0}^{m_n-1}|\cos(\pi\theta_j)| \cdot |\phi(m_n)| \notag
\end{align}
Note that by Corollary \ref{cor:prod_cos} and Lemma \ref{lem:upperbddtildeP}, equation \eqref{eq:Green_at_0} yields
\begin{align}\label{eq:Green_at_0_2}
|\tilde{P}_{q_n-1}(\theta_{m_n-q_n+1})| \leq C(\varepsilon) e^{m_n (\tilde{L}+\varepsilon)} e^{(q_n-m_n)(-\ln 2+\varepsilon)} |\phi(m_n-q_n)|
+C(\varepsilon) e^{(q_n-m_n)(\tilde{L}+\varepsilon)} e^{m_n (-\ln 2+\varepsilon)} |\phi(m_n)|.
\end{align}

Next, let us consider the eigenvalue equation:
\begin{align}\label{eq:ev_at_mn}
\phi(m_n+1)+\phi(m_n-1)+\lambda \tan(\pi \theta_{m_n})\phi(m_n)=E\phi(m_n)
\end{align}
By \eqref{Shnol}, we have
\begin{align*}
|\phi(k)|\leq C_0 |k|
\end{align*} 
for any $|k| \geq 1$.
Hence \eqref{eq:ev_at_mn} and \eqref{def:cnell} imply that for $n$ large enough
\begin{align*}
\frac{\lambda}{2 c_{n,0}} |\phi(m_n)| \leq |\lambda \tan(\pi \theta_{m_n})-E|\cdot |\phi(m_n)| &\leq 2 \max(|\phi(m_n+1)|,|\phi(m_n-1)|)\leq 2C_0 q_n.
\end{align*}
This implies
\begin{align}\label{est:mn}
|\phi(m_n)| \leq \frac{4}{\lambda} C_0 c_{n,0} q_n.
\end{align}
Similarly one has
\begin{align}\label{est:mn-qn}
|\phi(m_n-q_n)|\leq \frac{4}{\lambda} C_0 c_{n,-1} q_n.
\end{align}
Plugging \eqref{est:mn} and \eqref{est:mn-qn} into \eqref{eq:Green_at_0_2}, using \eqref{LEtildeLE} we have
\begin{align*}
|\tilde{P}_{q_n-1}(\theta_{m_n-q_n+1})| 
\leq &C_0 e^{q_n (-\ln 2+2\varepsilon)} \max(e^{m_n L}c_{n,0}, e^{(q_n-m_n)L}c_{n,-1})\\
\leq &C_0 e^{q_n (\tilde{L}+2\varepsilon)} \max(c_{n,0}, c_{n,-1}).
\end{align*}
By \eqref{eq:cnell_final}, we have
\begin{align}
|\tilde{P}_{q_n-1}(\theta_{m_n-q_n+1})| 
\leq C_0 e^{q_n (\tilde{L}-\beta_n+3\varepsilon)} \max(e^{\delta_n q_n}, 1).
\end{align}
\qed
\end{proof}

Furthermore, we have for any $|\ell |\leq 2q_{n+1}/(3q_n)$
\begin{lemma}\label{lem:Pqn_mn+ell}
We have
\begin{align*}
|\tilde{P}_{q_n-1}(\theta_{m_n+\ell q_n+1})| \leq e^{q_n (\tilde{L}-\beta_n+4\varepsilon)}\max(e^{\delta_n q_n}, |\ell|, 1).
\end{align*} 
\end{lemma}
\begin{proof}
By telescoping argument, we have
\begin{align*}
&|\tilde{P}_{q_n-1}(\theta_{m_n+\ell q_n+1})-\tilde{P}_{q_n-1}(\theta_{m_n-q_n+1})|\\
\leq &\|F_{q_n-1}(\theta_{m_n+\ell q_n+1})-F_{q_n-1}(\theta_{m_n-q_n+1})\|\\
\leq &\sum_{j=0}^{q_n-2} \| F_{q_n-j-2}(\theta_{m_n+\ell q_n+2+j})\|\cdot \| F(\theta_{m_n-q_n+1+j})-F(\theta_{m_n+\ell q_n+1+j})\|\cdot  \| F_{j}(\theta_{m_n-q_n+1}) \|,
\end{align*}
where $F_0:=\mathrm{Id}$.
By \eqref{tildePinF} and Lemma \ref{lem:upperbddtildeP}, we have that for $\ell\neq -1$,
\begin{align}\label{eq:P-P_tele}
&|\tilde{P}_{q_n-1}(\theta_{m_n+\ell q_n+1})-\tilde{P}_{q_n-1}(\theta_{m_n-q_n+1})| \notag\\
\leq &\sum_{j=0}^{q_n-2} C(\varepsilon) e^{(q_n-2)(\tilde{L}+\varepsilon)} \|\theta_{m_n+\ell q_n+1+j}-\theta_{m_n-q_n+1+j}\|\notag\\
\leq &\sum_{j=0}^{q_n-2} C(\varepsilon) e^{(q_n-2)(\tilde{L}+\varepsilon)}\|(\ell+1) q_n \alpha\|\notag\\
\leq &|\ell+1| e^{q_n(\tilde{L}-\beta_n+2\varepsilon)}.
\end{align}
Combining \eqref{eq:P-P_tele} with Lemma \ref{lem:Pqn_mn}, we have
\begin{align}\label{eq:C0<eqn}
|\tilde{P}_{q_n-1}(\theta_{m_n+\ell q_n+1})|
\leq &C_0 e^{q_n (\tilde{L}-\beta_n+3\varepsilon)} \max(e^{\delta_n q_n},1)+ |\ell+1| e^{q_n(\tilde{L}-\beta_n+2\varepsilon)} \notag\\
\leq &e^{q_n (\tilde{L}-\beta_n+4\varepsilon)}\max(e^{\delta_n q_n}, |\ell|, 1),
\end{align}
where we require $C_0<e^{q_n\varepsilon}$ in \eqref{eq:C0<eqn}. \qed
\end{proof}

Lemma \ref{lem:Pqn_mn+ell} implies the following lemma.
\begin{lemma}\label{lem:cor_Pqn}
For $|\ell |<2q_{n+1}/(3q_n)$, assume $k<2q_n$ and
$$y\leq \ell q_n+m_n+1,\ \ \text{and}\ \ y+k-1\geq (\ell +1)q_n+m_n-1.$$
We then have
\begin{align*}
|\tilde{P}_k(\theta_y)|\leq \max(e^{\delta_n q_n}, |\ell |, 1) e^{-(\beta_n-6\varepsilon) q_n} e^{k\tilde{L}}.
\end{align*}
\end{lemma}
\begin{proof}
Note that when $y=\ell q_n+m_n+1$ and $y+k-1=(\ell+1)q_n+m_n-1$, this is the previous lemma.
Hence let us consider the case when $(y, y+k-1)\neq (\ell q_n+m_n+1, (\ell+1)q_n+m_n-1)$.
We shall divide into three cases: 

1. $y=\ell q_n+m_n+1$, $y+k-1>(\ell+1)q_n+m_n-1$; 

2. $y<\ell q_n+m_n+1$, $y+k-1=(\ell+1)q_n+m_n-1$; 

3. $y<\ell q_n+m_n+1$, $y+k-1>(\ell+1)q_n+m_n-1$.

These three cases are very similar to each other, hence we will only present the proof for Case 1 in details.

{\it Case 1.} Let us consider the normalized transfer matrix
\begin{align*}
F_k(\theta_y)=F_{k-q_n}(\theta_{(\ell+1)q_n+m_n+1})F(\theta_{(\ell+1)q_n+m_n})F_{q_n-1}(\theta_{\ell q_n+m_n+1})
\end{align*}
Recall 
\begin{align*}
F(\theta_{(\ell+1)q_n+m_n})=
\left(
\begin{matrix}
\cos(\pi\theta_{(\ell+1)q_n+m_n})E-\lambda \sin(\pi\theta_{(\ell+1)q_n+m_n})\ \ &-\cos(\pi\theta_{(\ell+1)q_n+m_n})\\
\cos(\pi\theta_{(\ell+1)q_n+m_n}) &0
\end{matrix}
\right)
\end{align*}
and
\begin{align*}
&F_{q_n-1}(\theta_{\ell q_n+m_n+1})\\=&
\left(
\begin{matrix}
\tilde{P}_{q_n-1}(\theta_{\ell q_n+m_n+1})\ \ &-\tilde{P}_{q_n-2}(\theta_{\ell q_n+m_n+2})\cos(\pi\theta_{\ell q_n+m_n+1})\\
\tilde{P}_{q_n-2}(\theta_{\ell q_n+m_n+1})\cos(\pi\theta_{(\ell+1)q_n+m_n-1})\ \ &-\tilde{P}_{q_n-3}(\theta_{\ell q_n+m_n+2})\cos(\pi\theta_{\ell q_n+m_n+1})\cos(\pi\theta_{(\ell+1)q_n+m_n-1})
\end{matrix}
\right)
\end{align*}
By Lemmas \ref{lem:upperbddtildeP}, \ref{lem:Pqn_mn+ell} and inequality \eqref{eq:cnell_final}, we have
\begin{align*}
F(\theta_{(\ell+1)q_n+m_n})F_{q_n-1}(\theta_{\ell q_n+m_n+1})=:
\left(
\begin{matrix}
b_1 &b_2\\
b_3 & b_4
\end{matrix}
\right),
\end{align*}
satisfies
\begin{align*}
\begin{cases}
|b_1|\leq C \max(e^{\delta_n q_n}, |\ell |, 1) e^{(\tilde{L}-\beta_n+4\varepsilon) q_n}\\
|b_3|\leq \max(e^{\delta_n q_n}, |\ell |, 1) c_{n,\ell+1} e^{(\tilde{L}-\beta_n+4\varepsilon) q_n} \\
|b_4|\leq  c_{n,\ell+1}e^{(\tilde{L}+\varepsilon) q_n}\leq C \max(e^{\delta_n q_n}, |\ell |, 1) e^{(\tilde{L}-\beta_n+\varepsilon) q_n}
\end{cases}
\end{align*}
For $F_{k-q_n}(\theta_{(a+1)q_n+m_n+1})$ we use the estimate from Lemma \ref{lem:upperbddtildeP},
resulting in the  desired estimate for the upper left corner of $F_k(\theta_y)$, which is $\tilde{P}_k(\theta_y)$, as follows
\begin{align*}
|\tilde{P}_k(\theta_y)|\leq C(\varepsilon) \max(e^{\delta_n q_n}, |\ell |, 1) e^{(\tilde{L}-\beta_n+4\varepsilon) q_n} e^{(k-q_n)(\tilde{L}+\varepsilon)}\leq \max(e^{\delta_n q_n}, |\ell |, 1) e^{-(\beta_n-6\varepsilon) q_n} e^{k\tilde{L}}.
\end{align*}
This proves Case 1. \qed
\end{proof}

In order to unify the estimates for $\tilde{P}$ in both the $\beta_n\geq \delta_n+200\varepsilon$ and $\beta_n<\delta_n+200\varepsilon$ cases, we define
\begin{align}\label{def:g}
g_{k,\ell}:=
\begin{cases}
\max(e^{\delta_n q_n}, |\ell|, 1) e^{-(\beta_n-6\varepsilon)q_n}\ &\text{ if } \beta_n\geq \delta_n+200\varepsilon\\
e^{2\varepsilon k} &\text{ if } \beta_n<\delta_n+200\varepsilon
\end{cases}
\end{align}
By Lemmas \ref{lem:upperbddtildeP} and \ref{lem:cor_Pqn}, we have
\begin{cor}\label{cor:Pk_key}
For $|\ell |<2q_{n+1}/(3q_n)$, assume $k<2q_n$ and
$$y\leq \ell q_n+m_n+1,\ \ \text{and}\ \ y+k-1\geq (\ell +1)q_n+m_n-1.$$
We then have
\begin{align*}
|\tilde{P}_k(\theta_y)|\leq g_{k,\ell} e^{k\tilde{L}}.
\end{align*}
\end{cor}


\section{The case of non-resonant singularity: $\dist(m_n, q_n \Z)>b_n$}\label{Sec:C1}
 
We will first prove non-resonant $y$'s can be dominated by resonances, and then study the relation between adjacent resonant regions.
\subsection{Non-resonance}\

Assume $\ell q_n+b_n \leq y \leq (\ell+1)q_n-b_n$.
We introduce some notations:
\begin{align*}
\begin{cases}
I^-:=[\ell q_n+b_n, \ell q_n+m_n-1],\\
I^+:=[\ell q_n+m_n+1, (\ell+1)q_n-b_n],\\
|\phi(x_0^-)|:=\max_{y\in I^-} |\phi(y)|,\\
|\phi(x_0^+)|:=\max_{y\in I^+} |\phi(y)|,\\
R_{\ell}:=[\ell q_n-b_n, \ell q_n+b_n],\\
r_{\ell}:=\max_{k\in R_{\ell}} |\phi(k)|.
\end{cases}
\end{align*}

\begin{lemma}\label{lem:C2_n-r}
We have, 
for $y=\ell q_n+m_n$,
\begin{align}\label{eq:I-capI+}
|\phi(y)| \leq e^{15\varepsilon q_n}  c_{n,\ell} \max(e^{-(y-\ell q_n)L} r_{\ell}, e^{-((\ell+1)q_n-y)L} r_{\ell+1})
\end{align}
For any $y \in I^-$,
\begin{align}\label{eq:I-}
|\phi(y)| \leq e^{15\varepsilon q_n} \max(e^{-(y-\ell q_n)L} r_{\ell}, c_{n,\ell} e^{-((\ell+1)q_n-y)L} r_{\ell+1}).
\end{align}
For any $y \in I^+$,
\begin{align}\label{eq:I+}
|\phi(y)| \leq e^{15\varepsilon q_n} \max(c_{n,\ell} e^{-(y-\ell q_n)L} r_{\ell}, e^{-((\ell+1)q_n-y)L} r_{\ell+1}).
\end{align}
\end{lemma}
We leave the proof in the appendix.

\subsection{Resonance}

The main lemma of this section is the following.
\begin{lemma}\label{lem:21_ra<final}
For any $\ell\neq 0$, $|\ell|\leq q_{n+1}/(6q_n)$,
\begin{align*}
r_{\ell} \leq e^{40\varepsilon q_n} \frac{e^{-q_nL}}{\max(|\ell|,1)} \max(r_{\ell-1}, r_{\ell+1})\times
\begin{cases}
\max(|\ell|, e^{\delta_n q_n}), &\text{ if } \beta_n\geq \delta_n+200\varepsilon\\
e^{\beta_n q_n}, &\text{ if } \beta_n<\delta_n+200\varepsilon
\end{cases}.
\end{align*}
\end{lemma}
\begin{proof}
This lemma is built on the following lemma.
\begin{lemma}\label{lem:21_assume_Ia}
Assume that there exists $x_1\in I_{\ell}$ such that 
\begin{align}\label{Case1_assume}
|\tilde{P}_{2q_n-1}(\theta_{x_1})|\geq \max(|\ell |, 1) e^{-\beta_n q_n} e^{(\tilde{L}-2\varepsilon)(2q_n-1)}.
\end{align}
Then we have
\begin{align*}
r_{\ell} \leq e^{39\varepsilon q_n} \frac{e^{\beta_n q_n}}{\max(|\ell|,1)} e^{-q_n L}\,  \max(c_{n,\ell-1} r_{\ell-1}, c_{n,\ell} r_{\ell+1}).
\end{align*}
\end{lemma}
We postpone the proof of this lemma till the end of the section.

As a corollary of Lemma \ref{lem:21_assume_Ia}, we have the following.
\begin{lemma}\label{lem:21_assume_la_beta><delta}
Under the assumption of Lemma \ref{lem:21_assume_Ia}, we have
\begin{align*}
r_{\ell} \leq e^{40\varepsilon q_n} \frac{e^{-q_nL}}{\max(|\ell|,1)} \max(r_{\ell-1}, r_{\ell+1})\times
\begin{cases}
\max(|\ell|, e^{\delta_n q_n}, 1), &\text{ if } \beta_n\geq \delta_n+200\varepsilon\\
e^{\beta_n q_n}, &\text{ if } \beta_n<\delta_n+200\varepsilon
\end{cases}.
\end{align*}
\end{lemma}
\begin{proof}
If $\beta_n\geq \delta_n+200\varepsilon$, bound the $c_{n,j}$'s by \eqref{eq:cnell_final}. Otherwise trivially bound the $c_{n,j}$'s by $1$. \qed
\end{proof}

Next we show the assumption in Lemma \ref{lem:21_assume_la_beta><delta} can not hold for $\ell=0$.
\begin{lemma}\label{lem:21_I0_small}
For any $x\in I_0$, $|\tilde{P}_{2q_n-1}(\theta_x)|<e^{-\beta_n q_n} e^{(\tilde{L}-2\varepsilon)(2q_n-1)}$
\end{lemma}
\begin{proof}
We prove by contradiction. Assume that there exists $x_1\in I_0$ such that
$$|\tilde{P}_{2q_n-1}(\theta_x)|\geq e^{-\beta_n q_n} e^{(\tilde{L}-2\varepsilon)(2q_n-1)}.$$
Then Lemma \ref{lem:21_assume_la_beta><delta} implies that, using \eqref{assume:phi0=1} and \eqref{Shnol},
\begin{align}\label{phi0_2}
1\leq |\phi(0)|\leq &r_0\leq e^{-(L-40\varepsilon) q_n} \max(r_{-1}, r_1) \times 
\begin{cases} 
e^{\delta_n q_n}, \text{ if } \beta_n\geq \delta_n+200\varepsilon\\
e^{\beta_n q_n}, \text{ if } \beta_n<\delta_n+200\varepsilon
\end{cases}\\
\leq &C_0q_n \begin{cases} 
e^{(\delta_n-L+40\varepsilon) q_n} e^{\delta_n q_n}, \text{ if } \beta_n\geq \delta_n+200\varepsilon \notag\\
e^{(\delta_n-L+40\varepsilon) q_n} e^{\beta_n q_n}, \text{ if } \beta_n<\delta_n+200\varepsilon
\end{cases} \notag\\
<&1. \notag
\end{align}
Contradiction. \qed
\end{proof}

Combining Corollary \ref{cor:1_res_uni} with Lemma \ref{lem:21_I0_small}, we obtain the following 
\begin{lemma}\label{lem:21_P2q_n_large}
For any $\ell \neq 0$, $|\ell|\leq q_{n+1}/(6q_n)$, we have
$$|\tilde{P}_{2q_n-1}(\theta_{x_1})|\geq \max(|\ell|, 1) e^{-\beta_n q_n} e^{(\tilde{L}-2\varepsilon)(2q_n-1)}.$$
\end{lemma}
Finally, Lemma \ref{lem:21_ra<final} follows from combining Lemma \ref{lem:21_assume_la_beta><delta} with Lemma \ref{lem:21_P2q_n_large}. \qed
\end{proof}

Now let us prove Lemma \ref{lem:21_assume_Ia}.
\subsection*{Proof of Lemma \ref{lem:21_assume_Ia}}\
Recalling that $r_{\ell}=\max_{k\in R_{\ell}} |\phi(k)|$, let $k\in R_{\ell}$.

Expanding $\phi(k)$ in the interval $I=[x_1, x_2]$ using \eqref{PkG}, where $x_2=x_1+2q_n-2$, we have,
\begin{align}\label{eq:hhh1}
|\phi(k)|
\leq &\frac{|\tilde{P}_{x_2-k}(\theta_{k+1})|}{|\tilde{P}_{2q_n-1}(\theta_{x_1})|} \prod_{j=x_1}^k |\cos(\pi \theta_j)|\, |\phi(x_1-1)|+
\frac{|\tilde{P}_{k-x_1}(\theta_{x_1})|}{|\tilde{P}_{2q_n-1}(\theta_{x_1})|} \prod_{j=k}^{x_2} |\cos(\pi \theta_j)|\, |\phi(x_2+1)|
\end{align}

\subsection*{Case 1. If $x_1\in [(\ell-1)q_n+m_n+1, \ell q_n-\lf q_n/2\rf -1]$.}\

Note that since
$$k+1\leq \ell q_n+b_n+1\leq  \ell q_n+m_n+1,$$
and
$$x_2=x_1+2q_n-2\geq (\ell+1)q_n+m_n-1,$$
Corollary \ref{cor:Pk_key} implies 
hence
$$|\tilde{P}_{x_2-k}(\theta_{k+1})|\leq g_{|x_2-k|, \ell}  e^{|x_2-k|\tilde{L}}.$$

By Corollary \ref{cor:prod_cos}  and \eqref{eq:delta_n_min_cos}, we have that
\begin{align*}
&\prod_{j=k}^{x_2}|\cos(\pi \theta_j))|\leq C(\varepsilon) c_{n,\ell} c_{n,\ell+1} e^{(-\ln{2}+\varepsilon)|x_2-k|},\ \ \text{ and }\\
&\prod_{j=x_1}^k |\cos(\pi \theta_j))|\leq C(\varepsilon) e^{(-\ln 2+\varepsilon)|k-x_1|}.
\end{align*}
By Lemma \ref{lem:upperbddtildeP} we have
\begin{align*}
|\tilde{P}_{k-x_1}(\theta_{x_1})|\leq C(\varepsilon) e^{(\tilde{L}+\varepsilon)|k-x_1|}.
\end{align*}
Plugging these upper bounds together with the lower bound \eqref{Case1_assume} into \eqref{eq:hhh1}, we obtain
\begin{align}\label{eq:111}
|\phi(k)| \leq &C(\varepsilon) \frac{g_{|x_2-k|, \ell} e^{\beta_n q_n}}{\max(|\ell |,1)} e^{5\varepsilon q_n} e^{-L |k-x_1|} |\phi(x_1-1)|+C(\varepsilon) \frac{c_{n,\ell}c_{n,\ell+1}e^{\beta_n q_n}}{\max(|\ell |,1)} e^{5\varepsilon q_n} e^{-L|x_2-k|} |\phi(x_2+1)|.
\end{align}

Equation \eqref{eq:I+} of Lemma \ref{lem:C2_n-r} implies
\begin{align*}
|\phi(x_1-1)|\leq e^{15\varepsilon q_n}  \max\{c_{n,\ell-1} e^{-(x_1-(\ell-1)q_n)L} r_{\ell-1},  e^{-(\ell q_n-x_1)L} r_{\ell}\}
\end{align*}
and 
\begin{align*}
|\phi(x_2+1)|\leq e^{15\varepsilon q_n} \max\{c_{n,\ell+1} e^{-(x_2+1-(\ell+1)q_n)L} r_{\ell+1}, e^{-((\ell+2)q_n-x_2-1)L} r_{\ell+2}\}
\end{align*}
Plugging the above estimates into \eqref{eq:111}, we have
\begin{align}\label{eq:21_ra_leq_1}
|\phi(k)| \leq e^{24\varepsilon q_n} \frac{e^{\beta_n q_n}}{\max(|\ell|,1)} 
\max
\begin{cases} 
c_{n,\ell-1} g_{|x_2-k|, \ell} e^{-q_n L} r_{\ell-1}\\
 g_{|x_2-k|, \ell}  e^{-q_n L} r_{\ell}\\
c_{n,\ell} (c_{n,\ell+1})^2 e^{-q_n L} r_{\ell+1}\\
c_{n,\ell} c_{n,\ell+1} e^{-2q_n L} r_{\ell+2}
\end{cases}
\end{align}

\subsection*{Case 2. If $x_1\in [(\ell-1)q_n-\lf q_n/2\rf, (\ell-1)q_n+m_n]$}\

We have $$x_2\in [(\ell+1)q_n-\lf q_n/2\rf-2, (\ell+1)q_n+m_n-2].$$

By Corollary \ref{cor:prod_cos} and \eqref{def:cnell} we have that
\begin{align*}
&\prod_{j=k}^{x_2}|\cos(\pi \theta_j))|\leq C(\varepsilon) c_{n,\ell} e^{(-\ln{2}+\varepsilon)|x_2-k|},\ \ \text{ and }\\
&\prod_{j=x_1}^k |\cos(\pi \theta_j))|\leq C(\varepsilon) c_{n,\ell-1} e^{(-\ln 2+\varepsilon)|k-x_1|}.
\end{align*}
By Lemma \ref{lem:upperbddtildeP} we have
\begin{align*}
|\tilde{P}_{k-x_1}(\theta_{x_1})|\leq C(\varepsilon) e^{(\tilde{L}+\varepsilon)|k-x_1|} \text{ and }
|\tilde{P}_{x_2-k}(\theta_{k+1})|\leq C(\varepsilon) e^{(\tilde{L}+\varepsilon)|x_2-k|}.
\end{align*}
Plugging these upper bounds together with the lower bound \eqref{Case1_assume} into \eqref{eq:hhh1}, we have
\begin{align}\label{eq:hhh3}
|\phi(k)|
\leq &C(\varepsilon) e^{6\varepsilon q_n}  \frac{e^{\beta_n q_n}}{\max(|\ell|, 1)} 
\left(e^{-|k-x_1|L} c_{n,\ell-1} |\phi(x_1-1)|+e^{-|x_2-k|L} c_{n,\ell} |\phi(x_2+1)|\right).
\end{align}
Lemma \ref{lem:C2_n-r} implies
\begin{align*}
&|\phi(x_1-1)|\leq e^{15\varepsilon q_n} \max\{
             c_{n,\ell-2} e^{-(x_1-(\ell-2)q_n)L} r_{\ell-2}, 
             e^{-|(\ell-1)q_n-x_1|L} r_{\ell-1},
             c_{n,\ell-1} e^{-(\ell q_n-x_1)L} r_{\ell}\}, \ \ \text{ and }\\
&|\phi(x_2+1)|\leq e^{15\varepsilon q_n} \max\{
             c_{n,\ell} e^{-(x_2-\ell q_n)L} r_{\ell},
             e^{-|(\ell+1)q_n-x_2|L} r_{\ell+1},
             c_{n,\ell+1} e^{-((\ell+2)q_n-x_2)L} r_{\ell+2}\}.
\end{align*}
Plugging these estimates in \eqref{eq:hhh3}, we have
\begin{align}\label{eq:21_ra_leq_3}
|\phi(k)|\leq e^{24\varepsilon q_n} \frac{e^{\beta_n q_n}}{\max(|\ell|,1)} \max
\begin{cases}
c_{n,\ell-2}c_{n,\ell-1} e^{-2 q_n L} r_{\ell-2}\\
e^{-q_n L} c_{n,\ell-1} r_{\ell-1}\\
\max((c_{n,\ell-1})^2, (c_{n,\ell})^2) e^{-q_n L} r_{\ell}\\
c_{n,\ell} e^{-q_n L} r_{\ell+1}\\
c_{n,\ell} c_{n,\ell+1} e^{-2 q_n L} r_{\ell+2}
\end{cases}
\end{align}

Putting estimates in both cases \eqref{eq:21_ra_leq_1} and \eqref{eq:21_ra_leq_3} together, we obtain, after setting $|\phi(k)|=r_{\ell}$,
\begin{align}\label{eq:21_ra_leq_4}
r_{\ell} \leq e^{24\varepsilon q_n} \frac{e^{\beta_n q_n}}{\max(|\ell|,1)}  \max
\begin{cases}
c_{n,\ell-2}c_{n,\ell-1} e^{-2 q_n L} r_{\ell-2}\\
c_{n,\ell-1} e^{-q_n L} r_{\ell-1} \max_{k\in R_{\ell}} g_{|x_2-k|, \ell} \\
e^{-q_n L} c_{n,\ell-1} r_{\ell-1}\\
e^{-q_n L} r_{\ell} \cdot \max_{k\in R_{\ell}} g_{|x_2-k|, \ell}  \\
\max((c_{n,\ell-1})^2, (c_{n,\ell})^2) e^{-q_n L} r_{\ell}\\
c_{n,\ell} e^{-q_n L} r_{\ell+1}\\
c_{n,\ell} c_{n,\ell+1} e^{-2 q_n L} r_{\ell+2}
\end{cases}
\end{align}
We have by Corollary \ref{cor:A_upper_mn} that
\begin{align*}
r_{\ell-2} \leq \frac{1}{c_{n,\ell-2}}e^{7\varepsilon q_n} e^{L(q_n+2b_n)}<\frac{1}{c_{n,\ell-2}}e^{9\varepsilon q_n} e^{Lq_n} r_{\ell-1}, 
\end{align*}
and similarly
\begin{align*}
r_{\ell+2} \leq \frac{1}{c_{n,\ell+1}} e^{9\varepsilon q_n} e^{q_n L} r_{\ell+1}.
\end{align*}
Plugging these estimates into \eqref{eq:21_ra_leq_4} yields
\begin{align}\label{eq:21_ra_leq_5}
r_{\ell} \leq e^{33\varepsilon q_n} \frac{e^{\beta_n q_n}}{\max(|\ell|,1)} \max
\begin{cases}
c_{n,\ell-1} e^{-q_n L} r_{\ell-1}\\
c_{n,\ell-1} e^{-q_n L} r_{\ell-1} \cdot \max_{k\in R_{\ell}}  g_{|x_2-k|, \ell}\\
e^{-q_n L} r_{\ell} \cdot \max_{k\in R_{\ell}}  g_{|x_2-k|, \ell}\\
\max((c_{n,\ell-1})^2, (c_{n,\ell})^2) e^{-q_n L} r_{\ell}\\
c_{n,\ell} e^{-q_n L} r_{\ell+1}
\end{cases}
\end{align}
Next we further bound this, dividing into two cases.

Case (i). If $\beta_n\geq \delta_n+200\varepsilon$, using \eqref{def:g} to bound 
\begin{align*}
\max_{k\in R_{\ell}} g_{|x_2-k|,\ell} 
\leq &\max(e^{\delta_n q_n}, |\ell|) e^{-(\beta_n-6\varepsilon)q_n}\\
\leq &e^{6\varepsilon q_n} \max(e^{-200\varepsilon q_n}, \frac{1}{6q_n})<e^{6\varepsilon q_n},
\end{align*}
and using \eqref{eq:cnell_final} to bound
\begin{align*}
\max((c_{n,\ell-1})^2, (c_{n,\ell})^2)\leq \max(c_{n,\ell-1}, c_{n,\ell}) \leq C \max(|\ell|, e^{\delta_n q_n}, 1) e^{-\beta_n q_n},
\end{align*}
we arrive at
\begin{align}\label{eq:21_ra_leq_6}
r_{\ell} \leq e^{39\varepsilon q_n} \frac{e^{\beta_n q_n}}{\max(|\ell|,1)} e^{-q_n L} \max
\begin{cases}
c_{n,\ell-1} r_{\ell-1}\\
\max(|\ell|, e^{\delta_n q_n}, 1) e^{-\beta_n q_n}  r_{\ell}\\
c_{n,\ell} r_{\ell+1}
\end{cases}.
\end{align}
Note that the coefficient of $r_{\ell}$ can be bounded by, using \eqref{eq:L>deltan+100},
\begin{align}\label{eq:rell_drop}
e^{39\varepsilon q_n} \frac{e^{\beta_n q_n}}{\max(|\ell|,1)} e^{-q_n L} \max(|\ell|, e^{\delta_n q_n}, 1) e^{-\beta_n q_n}\leq e^{-(L-\delta_n-39\varepsilon)q_n}<1.
\end{align}
Hence \eqref{eq:21_ra_leq_6} implies
\begin{align}\label{eq:21_ra_leq_7}
r_{\ell} \leq e^{39\varepsilon q_n} \frac{e^{\beta_n q_n}}{\max(|\ell|,1)} e^{-q_n L} \max(c_{n,\ell-1} r_{\ell-1}, c_{n,\ell} r_{\ell+1}).
\end{align}

Case (ii). If $\beta_n<\delta_n+200\varepsilon$, using \eqref{def:g} to bound 
\begin{align*}
\max_{k\in R_{\ell}} g_{|x_2-k|,\ell} 
\leq e^{2\varepsilon |x_2-k|}\leq e^{3\varepsilon q_n},
\end{align*}
and trivially bounding 
\begin{align*}
\max((c_{n,\ell-1})^2, (c_{n,\ell})^2)\leq 1,
\end{align*} 
we obtain from \eqref{eq:21_ra_leq_5} that
\begin{align}\label{eq:21_ra_leq_8}
r_{\ell} \leq e^{36\varepsilon q_n} \frac{e^{\beta_n q_n}}{\max(|\ell|,1)} e^{-q_n L} \max(c_{n,\ell-1} r_{\ell-1}, r_{\ell}, c_{n,\ell} r_{\ell+1}).
\end{align}
Note that the coefficient of $r_{\ell}$ can be bounded by, using \eqref{eq:L>deltan+100},
\begin{align}\label{eq:rell_drop_2}
e^{36\varepsilon q_n} \frac{e^{\beta_n q_n}}{\max(|\ell|,1)} e^{-q_n L}\leq e^{-(L-\delta_n-236\varepsilon)q_n}<1.
\end{align}
Hence \eqref{eq:21_ra_leq_8} implies
\begin{align}\label{eq:21_ra_leq_9}
r_{\ell} \leq e^{36\varepsilon q_n} \frac{e^{\beta_n q_n}}{\max(|\ell|,1)} e^{-q_n L} \max(c_{n,\ell-1} r_{\ell-1}, c_{n,\ell} r_{\ell+1}).
\end{align}
Lemma \ref{lem:21_assume_Ia} follows from combining \eqref{eq:21_ra_leq_7} and \eqref{eq:21_ra_leq_9}.
\qed



\section{The case of the resonant singularity: $\dist(m_n, q_n\Z)\leq b_n$}\label{Sec:C2}
Let us introduce some notations:
\begin{align*}
R_{\ell}^+:=[\ell q_n+m_n+1, \ell q_n+b_n] \text{ and } R_{\ell}^-:=[\ell q_n-b_n, \ell q_n+m_n-1],
\end{align*}
and
\begin{align}\label{def:ra+-}
r_{\ell}^+:=\max_{y\in R_{\ell}^+} |\phi(y)| \text{ and } r_{\ell}^-:=\max_{y\in R_{\ell}^-} |\phi(y)|
\end{align}

\subsection{Non-resonance}
\begin{lemma}\label{lem:22_non-to-res}
If $\ell q_n+b_n<y<(\ell+1)q_n-b_n$, for some $|\ell |\leq q_{n+1}/(6q_n)$. Then
\begin{align*}
|\phi(y)|\leq e^{30\varepsilon q_n} \max(e^{-(y-\ell q_n)L}r_{\ell}^+, e^{-((\ell+1)q_n-y)L}r_{\ell+1}^-).
\end{align*}
\end{lemma}
We leave the proof in the appendix.



\subsection{Resonance}\

Assume without loss of generality that $0<m_n\leq b_n$. 

The main lemma of this section is the following.
\begin{lemma}\label{lem:22_res_final}
For any $\ell\neq 0$ such that $|\ell|<q_{n+1}/(6q_n)$, we have
\begin{align*}
r_{\ell}\leq &e^{50\varepsilon q_n}\frac{e^{-q_nL}}{\max(|\ell|, 1)} \max(r_{\ell-1}, r_{\ell+1})\times
\begin{cases}
\max(|\ell|, e^{\delta_n q_n}), \text{ if } \beta_n\geq \delta_n+200\varepsilon,\\
e^{\beta_n q_n}, \text{ if }\beta_n< \delta_n+200\varepsilon
\end{cases}
\end{align*}
\end{lemma}
\begin{proof}
This lemma is mainly built on the following lemma.
\begin{lemma}\label{lem:22_res}
Assume that there exists $x_1\in I_{\ell}$, for some $|\ell | <q_{n+1}/(6q_n)$, such that
\begin{align}\label{ieq:22_res}
|\tilde{P}_{2q_n-1}(\theta_{x_1})|\geq  \max(|\ell|, 1) e^{-\beta_n q_n}e^{(\tilde{L}-2\varepsilon)(2q_n-1)}.
\end{align}

We then have
\begin{align*}
r_{\ell}^-\leq 
e^{49\varepsilon q_n} \frac{e^{\beta_n q_n}}{\max(|\ell|, 1)} e^{-q_n L} \max(
&c_{n,\ell-1} r_{\ell-1}^-,
c_{n,\ell-1} r_{\ell-1}^+,  
\gamma\, r_{\ell-1}^+, c_{n,\ell} r_{\ell}^+,
c_{n,\ell} r_{\ell+1}^-,
c_{n,\ell} c_{n,\ell+1} r_{\ell+1}^+),
\end{align*}
and
\begin{align*}
r_{\ell}^+\leq 
e^{49\varepsilon q_n} \frac{e^{\beta_n q_n}}{\max(|\ell|, 1)} e^{-q_n L} \max(
&c_{n,\ell} c_{n,\ell-1} r_{\ell-1}^-, 
c_{n,\ell} r_{\ell-1}^+,
c_{n,\ell} r_{\ell}^-, \gamma\, r_{\ell+1}^-,
c_{n,\ell+1} r_{\ell+1}^-,
c_{n,\ell+1} r_{\ell+1}^+),
\end{align*}
where 
\begin{align*}
\gamma:=\begin{cases}
\max(e^{\delta_n q_n}, |\ell|, 1) e^{-\beta_n q_n}, \text{ if } \beta_n\geq \delta_n+200\varepsilon\\
1, \text{ otherwise}
\end{cases}
\end{align*}
\end{lemma}
We will postpone the proof of this lemma till the end of the section. 

As a corollary of Lemma \ref{lem:22_res}, we have the following.
\begin{lemma}\label{lem:22_res_beta><delta+200}
Under the assumption of Lemma \ref{lem:22_res}.
We have
\begin{align*}
r_{\ell}^-\leq e^{50\varepsilon q_n} \frac{e^{-q_n L}}{\max(|\ell|, 1)} \max
(r_{\ell-1}^-,
r_{\ell-1}^+,
r_{\ell}^+,
r_{\ell+1}^-,
c_{n,\ell+1}r_{\ell+1}^+) \times
\begin{cases} 
\max(|\ell|, e^{\delta_n q_n}, 1), &\text{ if } \beta_n\geq \delta_n+200\varepsilon\\
e^{\beta_n q_n}, &\text{ if } \beta_n< \delta_n+200\varepsilon
\end{cases},
\end{align*}
and 
\begin{align*}
r_{\ell}^+\leq &e^{50\varepsilon q_n} \frac{e^{-q_n L}}{\max(|\ell|, 1)} \max
(c_{n,\ell-1}r_{\ell-1}^-,
r_{\ell-1}^+,
r_{\ell}^+,
r_{\ell+1}^-,
r_{\ell+1}^+) \times
\begin{cases}\max(|\ell|, e^{\delta_n q_n}, 1), &\text{ if } \beta_n\geq \delta_n+200\varepsilon\\
e^{\beta_n q_n}, &\text{ if } \beta_n< \delta_n+200\varepsilon
\end{cases}.
\end{align*}
\end{lemma}
\begin{proof}
If $\beta_n\geq \delta_n+200\varepsilon$, bound the $c_{n,j}$'s by \eqref{eq:cnell_final}. Otherwise trivially bound the $c_{n,j}$'s by $1$.
\qed
\end{proof}

Next we show the assumption \eqref{ieq:22_res} can not hold for $\ell=0$. 
\begin{lemma}\label{lem:22_I0large}
For any $x_1\in I_0$, we have
$$|\tilde{P}_{2q_n-1}(\theta_{x_1})|<e^{-\beta_n q_n}e^{(\tilde{L}-2\varepsilon)(2q_n-1)}.$$
\end{lemma}
\begin{proof}
Suppose otherwise, there exists $x_1\in I_0$ such that 
$$|\tilde{P}_{2q_n-1}(\theta_{x_1})|\geq  e^{-\beta_n q_n}e^{(\tilde{L}-2\varepsilon)(2q_n-1)}.$$
By Lemma \ref{lem:22_res_beta><delta+200} and equations \eqref{assume:phi0=1} and \eqref{Shnol} we have
\begin{align}\label{phi0_3}
1\leq |\phi(0)|\leq r_0^-
\leq &e^{(-L+50\varepsilon)q_n} \max(r_{-1}^-, r_{-1}^+, r_0^+, r_1^-, r_1^+) \times 
\begin{cases} 
\max(e^{\delta_n q_n}, 1), \text{ if } \beta_n\geq \delta_n+200\varepsilon\\
e^{\beta_n q_n}, \text{ if } \beta_n< \delta_n+200\varepsilon
\end{cases}\\
\leq &C_0 q_n e^{(-L+50\varepsilon)q_n}\times 
\begin{cases} 
\max(e^{\delta_n q_n}, 1), \text{ if } \beta_n\geq \delta_n+200\varepsilon \notag\\
e^{\beta_n q_n}, \text{ if } \beta_n< \delta_n+200\varepsilon
\end{cases} \notag\\
<&1. \notag
\end{align}
Contradiction! \qed
\end{proof}

Combining Corollary \ref{cor:1_res_uni} with Lemma \ref{lem:22_I0large}, we obtain the following.
\begin{lemma}
For any $\ell\neq 0$ such that $|\ell|\leq q_{n+1}/(6q_n)$, we have
$$|\tilde{P}_{2q_n-1}(\theta_{x_1})|\geq  \max(|\ell|, 1) e^{-\beta_n q_n}e^{(\tilde{L}-2\varepsilon)(2q_n-1)}.$$
\end{lemma}
Thus Lemma \ref{lem:22_res_beta><delta+200} holds for any $\ell\neq 0$ such that $|\ell|\leq q_{n+1}/(6q_n)$.
It particular it implies for the same $\ell$, the following hold
\begin{align*}
r_{\ell}\leq &e^{50\varepsilon q_n}\frac{e^{-q_nL}}{\max(|\ell|, 1)} \max(r_{\ell-1}, r_{\ell}, r_{\ell+1})\times
\begin{cases}
\max(|\ell|, e^{\delta_n q_n}), \text{ if } \beta_n\geq \delta_n+200\varepsilon,\\
e^{\beta_n q_n}, \text{ if }\beta_n< \delta_n+200\varepsilon
\end{cases}.
\end{align*}
By arguments similar to \eqref{eq:rell_drop} and \eqref{eq:rell_drop_2}, the $r_{\ell}$ terms on the right-hand-side of the equation above can be dropped. This proves Lemma \ref{lem:22_res_final}.\qed
\end{proof}

\subsection*{Proof of Lemma \ref{lem:22_res}}
We are going to give a detailed proof for $r_{\ell}^-$ when $0<m_n\leq b_n$, the other cases are similar.

Let $x_2=x_1+2q_n-2$.

By the Green's formula \eqref{Green_tildeP} we have,
\begin{align}\label{eq:res_formula_1}
|\phi(k)|
\leq &\frac{|\tilde{P}_{x_2-k}(\theta_{k+1})|}{|\tilde{P}_{2q_n-1}(\theta_{x_1})|} \prod_{j=x_1}^k |\cos(\pi \theta_j)|\, |\phi(x_1-1)|+
\frac{|\tilde{P}_{k-x_1}(\theta_{x_1})|}{|\tilde{P}_{2q_n-1}(\theta_{x_1})|} \prod_{j=k}^{x_2} |\cos(\pi \theta_j)|\, |\phi(x_2+1)|
\end{align}

\subsection*{Estimates for $r_{\ell}^-$}
\subsection*{Case 1. If $x_1\in [(\ell-1)q_n+m_n+2, \ell q_n-\lf q_n/2\rf-1]$}

By Lemma \ref{lem:upperbddtildeP} we have
\begin{align}\label{eq:res_1}
|\tilde{P}_{k-x_1}(\theta_{x_1})|\leq C(\varepsilon) e^{(\tilde{L}+\varepsilon)|k-x_1|}.
\end{align}
If $k\in R_{\ell}^-$, we have
$$k+1\in [\ell q_n-b_n+1, \ell q_n+m_n].$$
Also, 
$$x_2\in [(\ell+1)q_n+m_n, (\ell+2)q_n-\lf q_n/2\rf -3],$$
we have by Corollary \ref{cor:Pk_key} that
\begin{align}\label{eq:res_2}
|\tilde{P}_{x_2-k}(\theta_{k+1})|\leq g_{|x_2-k|, \ell} e^{\tilde{L} |x_2-k|}.
\end{align}
By Corollary \ref{cor:prod_cos}, we have
\begin{align}\label{eq:res_3}
\prod_{j=x_1}^k |\cos(\pi \theta_j)|\leq C(\varepsilon) e^{-(\ln 2-\varepsilon)|x_1-k|}, \text{ and }
\prod_{j=k}^{x_2} |\cos(\pi \theta_j)|\leq C(\varepsilon) e^{-(\ln 2-\varepsilon)|x_2-k|} c_{n,\ell} c_{n,\ell+1}.
\end{align}

Plugging estimates \eqref{eq:res_1}, \eqref{eq:res_2} and \eqref{eq:res_3} into \eqref{eq:res_formula_1}, 
we have
\begin{align}\label{eq:res_5}
|\phi(k)|\leq C(\varepsilon) e^{5\varepsilon q_n} \frac{g_{|x_2-k|,\ell} e^{\beta_n q_n}}{\max(|\ell|, 1)}e^{-L|x_1-k|} |\phi(x_1-1)|
+C(\varepsilon) e^{5\varepsilon q_n} \frac{c_{n,\ell} c_{n,\ell+1} e^{\beta_n q_n}}{\max(|\ell|, 1)}  e^{-L|x_2-k|} |\phi(x_2+1)|.
\end{align}

Lemma \ref{lem:22_non-to-res} implies
\begin{align}\label{eq:res_6}
\begin{cases}
|\phi(x_1-1)|\leq e^{30\varepsilon q_n}\max\{e^{-(x_1-(\ell-1)q_n)L} r_{\ell-1}^+, e^{-(\ell q_n-x_1)L} r_{\ell}^-\}\\
|\phi(x_2+1)|\leq e^{30\varepsilon q_n}\max\{e^{-(x_2-(\ell+1)q_n)L} r_{\ell+1}^+, e^{-((\ell+2)q_n-x_2)L} r_{\ell+2}^-\}
\end{cases}
\end{align}

Plugging \eqref{eq:res_6} into \eqref{eq:res_5}, we have
\begin{align}\label{eq:res_7}
|\phi(k)|\leq e^{43\varepsilon q_n}\frac{e^{\beta_n q_n}}{\max(|\ell|, 1)}e^{-q_n L} 
\max(
g_{|x_2-k|, \ell} r_{\ell-1}^+,
g_{|x_2-k|, \ell} r_{\ell}^-,
c_{n,\ell} c_{n,\ell+1} r_{\ell+1}^+),
\end{align}
where we controlled $r_{\ell+2}^-$ by $r_{\ell+1}^+$ using Corollary \ref{cor:A_upper} in the following way
$$r_{\ell+2}^-\leq e^{6\varepsilon q_n} e^{q_n L} r_{\ell+1}^+.$$

\subsection*{Case 2. If $x_1=(\ell-1)q_n+m_n+1$ and $x_2=(\ell+1)q_n+m_n-1$}
We use the following estimates by Lemma \ref{lem:upperbddtildeP},
\begin{align}\label{eq:res_1''}
|\tilde{P}_{k-x_1}(\theta_{x_1})|\leq C(\varepsilon)e^{(\tilde{L}+\varepsilon)|k-x_1|}, \text{ and }
|\tilde{P}_{x_2-k}(\theta_{k+1})|\leq C(\varepsilon)e^{(\tilde{L}+\varepsilon)|x_2-k|}.
\end{align}
By Corollary \ref{cor:prod_cos}, we have
\begin{align}\label{eq:res_3''}
\prod_{j=x_1}^k |\cos(\pi \theta_j)|\leq C(\varepsilon) e^{-(\ln 2-\varepsilon)|x_1-k|}, \text{ and }
\prod_{j=k}^{x_2} |\cos(\pi \theta_j)|\leq C(\varepsilon) e^{-(\ln 2-\varepsilon)|x_2-k|} c_{n,\ell}.
\end{align}

Plugging estimates \eqref{eq:res_1''} and \eqref{eq:res_3''} into \eqref{eq:res_formula_1}, 
we have
\begin{align}\label{eq:res_5''}
|\phi(k)|\leq C(\varepsilon) e^{5\varepsilon q_n} \frac{e^{\beta_n q_n}}{\max(|\ell |, 1)} \max(e^{-q_n L} |\phi((\ell-1)q_n+m_n)|,  c_{n,\ell} e^{-q_n L} |\phi((\ell+1)q_n+m_n)|).
\end{align}
By the eigenvalue equation
\begin{align*}
\phi(j q_n+m_n+1)+\phi(j q_n+m_n-1)=(E-\lambda\tan(\pi\theta_{j q_n+m_n}))\phi(j q_n+m_n).
\end{align*}
We have that for any $j\in \Z$,
\begin{align}\label{eq:res_6''}
\phi(j q_n+m_n)|\leq Cc_{n,j} \max(|\phi(j q_n+m_n+1)|, |\phi(j q_n+m_n-1)|)\leq Cc_{n,j} \max(r_{j}^+, r_{j}^-).
\end{align}
Plugging \eqref{eq:res_6''} into \eqref{eq:res_5''}, we have
\begin{align}\label{eq:res_7''}
|\phi(k)|\leq e^{7\varepsilon q_n} \frac{e^{\beta_n q_n}}{\max(|\ell |, 1)} e^{-q_n L}\max(c_{n,\ell-1} r_{\ell-1}^-, 
c_{n,\ell-1} r_{\ell-1}^+,  c_{n,\ell} c_{n,\ell+1}r_{\ell+1}^-, c_{n,\ell} c_{n,\ell+1} r_{\ell+1}^+).
\end{align}

\subsection*{Case 3. If $x_1\in [(\ell-1)q_n-\lf q_n/2\rf, (\ell-1)q_n+m_n]$}
Now $x_2\in [(\ell+1)q_n-\lf q_n/2\rf-2, (\ell+1)q_n+m_n-2]$.

Lemma \ref{lem:upperbddtildeP} yields
\begin{align}\label{eq:res_1'}
|\tilde{P}_{k-x_1}(\theta_{x_1})|\leq e^{(\tilde{L}+\varepsilon)|k-x_1|}, \text{ and }
|\tilde{P}_{x_2-k}(\theta_{k+1})|\leq e^{(\tilde{L}+\varepsilon)|x_2-k|}.
\end{align}
Corollary \ref{cor:prod_cos} yields
\begin{align}\label{eq:res_3'}
\prod_{j=x_1}^k |\cos(\pi \theta_j)|\leq e^{-(\ln 2-\varepsilon)|x_1-k|} c_{n,\ell-1}, \text{ and }
\prod_{j=k}^{x_2} |\cos(\pi \theta_j)|\leq e^{-(\ln 2-\varepsilon)|x_2-k|} c_{n,\ell}.
\end{align}

Plugging estimates \eqref{eq:res_1'} and \eqref{eq:res_3'} into \eqref{eq:res_formula_1},
we have
\begin{align}\label{eq:res_5'}
|\phi(k)|\leq C(\varepsilon) e^{5\varepsilon q_n} \frac{e^{\beta_n q_n}}{\max(|\ell|, 1)} 
\left(e^{-L|x_1-k|} c_{n,\ell-1} |\phi(x_1-1)|+e^{-L|x_2-k|} c_{n,\ell} |\phi(x_2+1)|\right)
\end{align}
Lemma \ref{lem:22_non-to-res} implies
\begin{align}\label{eq:res_6'}
\begin{cases}
|\phi(x_1-1)|\leq  e^{30\varepsilon q_n} \max\{e^{-((\ell-1)q_n-x_1)L} r_{\ell-1}^-, e^{-(x_1-(\ell-2)q_n)L} r_{\ell-2}^+\}\\
|\phi(x_2+1)|\leq e^{30\varepsilon q_n} \max\{e^{-((\ell+1)q_n-x_2)L} r_{\ell+1}^-, e^{-(x_2-\ell q_n)L} r_{\ell}^+\}
\end{cases}
\end{align}

Plugging \eqref{eq:res_6'} into \eqref{eq:res_5'}, we have
\begin{align}\label{eq:res_7'}
|\phi(k)|\leq e^{42\varepsilon q_n} \frac{e^{\beta_n q_n}}{\max(|\ell|, 1)}e^{- q_n L}  
\max(
c_{n,\ell-1} r_{\ell-1}^-,
c_{n,\ell} r_{\ell}^+,
c_{n,\ell} r_{\ell+1}^-),
\end{align}
where we controlled $r_{\ell-2}^+$ by $r_{\ell-1}^-$ using Corollary \ref{cor:A_upper} in the following way 
$$r_{\ell-2}^+\leq e^{6\varepsilon q_n} e^{q_n L} r_{\ell-1}^-.$$

Putting the three cases \eqref{eq:res_7}, \eqref{eq:res_6''} and \eqref{eq:res_7'} together, taking $|\phi(k)|=r_{\ell}^-$, we obtain
\begin{align}\label{eq:res_10''}
r_{\ell}^-\leq e^{43\varepsilon q_n}\frac{e^{\beta_n q_n}}{\max(|\ell|, 1)} e^{-q_n L} 
\max(
&c_{n,\ell-1} r_{\ell-1}^-,
c_{n,\ell-1} r_{\ell-1}^+,  
(\max_{k\in R_{\ell}} g_{|x_2-k|, \ell})\cdot \max(r_{\ell-1}^+, r_{\ell}^-), \\
&\ \ c_{n,\ell} r_{\ell}^+,
c_{n,\ell} r_{\ell+1}^-,
c_{n,\ell} c_{n,\ell+1} r_{\ell+1}^+ \notag
).
\end{align}
Next, we further bound this, dividing into two cases.
Case (i). If $\beta_n\geq \delta_n+200\varepsilon$, using \eqref{def:g} to bound 
\begin{align*}
\max_{k\in R_{\ell}} g_{|x_2-k|,\ell} 
\leq \max(e^{\delta_n q_n}, |\ell|) e^{-(\beta_n-6\varepsilon)q_n},
\end{align*}
and using \eqref{eq:cnell_final} to bound
\begin{align*}
c_{n,\ell-1}\leq C \max(|\ell|, e^{\delta_n q_n}) e^{-\beta_n q_n},
\end{align*}
we arrive at
\begin{align}\label{eq:res_11''}
r_{\ell}^- \leq e^{49\varepsilon q_n} \frac{e^{\beta_n q_n}}{\max(|\ell|,1)} e^{-q_n L} \max(
&c_{n,\ell-1} r_{\ell-1}^-,
\max(e^{\delta_n q_n}, |\ell|) e^{-\beta_n q_n} \max(r_{\ell-1}^+, r_{\ell}^-), \notag\\ 
&\ \ c_{n,\ell} r_{\ell}^+, c_{n,\ell} r_{\ell+1}^-,
c_{n,\ell} c_{n,\ell+1} r_{\ell+1}^+).
\end{align}
Note that the coefficient of $r_{\ell}^-$ on the right-hand-side of \eqref{eq:res_11''} can be bounded by, using \eqref{eq:L>deltan+100},
\begin{align*}
e^{49\varepsilon q_n} \frac{e^{\beta_n q_n}}{\max(|\ell|,1)} e^{-q_n L}\max(e^{\delta_n q_n}, |\ell|) e^{-\beta_n q_n}\leq e^{-(L-\delta_n-50)q_n}<1.
\end{align*}
Hence \eqref{eq:res_11''} implies
\begin{align}\label{eq:res_12''}
&r_{\ell}^-
\leq e^{49\varepsilon q_n} \frac{e^{\beta_n q_n}}{\max(|\ell|,1)} e^{-q_n L} \max(
c_{n,\ell-1} r_{\ell-1}^-,
\max(e^{\delta_n q_n}, |\ell|) e^{-\beta_n q_n} r_{\ell-1}^+, c_{n,\ell} r_{\ell}^+, c_{n,\ell} r_{\ell+1}^-,
c_{n,\ell} c_{n,\ell+1} r_{\ell+1}^+).
\end{align}

Case (ii). If $\beta_n<\delta_n+200\varepsilon$, using \eqref{def:g} to bound 
\begin{align*}
\max_{k\in R_{\ell}} g_{|x_2-k|,\ell} 
\leq e^{2\varepsilon |x_2-k|}\leq e^{3\varepsilon q_n},
\end{align*}
and trivially bounding $c_{n,\ell-1}\leq 1$,
we obtain from \eqref{eq:res_10''} that
\begin{align}\label{eq:res_13''}
r_{\ell}^- \leq e^{46\varepsilon q_n} \frac{e^{\beta_n q_n}}{\max(|\ell|,1)} e^{-q_n L} 
\max(c_{n,\ell-1} r_{\ell-1}^-,
r_{\ell-1}^+,  
r_{\ell}^-, 
c_{n,\ell} r_{\ell}^+,
c_{n,\ell} r_{\ell+1}^-,
c_{n,\ell} c_{n,\ell+1} r_{\ell+1}^+).
\end{align}
Note that the coefficient of $r_{\ell}^-$ on the right-hand-side of \eqref{eq:res_13''} can be bounded by, using \eqref{eq:L>deltan+100},
\begin{align*}
e^{46\varepsilon q_n} \frac{e^{\beta_n q_n}}{\max(|\ell|,1)} e^{-q_n L}\leq e^{-(L-\delta_n-246\varepsilon)q_n}<1.
\end{align*}
Hence \eqref{eq:res_13''} implies
\begin{align}\label{eq:res_14''}
r_{\ell}^- \leq e^{46\varepsilon q_n} \frac{e^{\beta_n q_n}}{\max(|\ell|,1)} e^{-q_n L} \max(c_{n,\ell-1} r_{\ell-1}^-,
r_{\ell-1}^+,  
c_{n,\ell} r_{\ell}^+,
c_{n,\ell} r_{\ell+1}^-,
c_{n,\ell} c_{n,\ell+1} r_{\ell+1}^+).
\end{align}
Lemma \ref{lem:22_res} follows from combining \eqref{eq:res_12''} and \eqref{eq:res_14''}.
\qed

\section{Localization: proofs of Cases 2 and 3 of Lemma \ref{lem:main}}\label{Sec:loc}
The proofs of Cases 2 and 3 of Lemma \ref{lem:main} are completely analogous. Hence we will only prove Case 3. Assume $\beta_n\geq \delta_n+200\varepsilon$.

Let us first prove 
\begin{lemma}\label{lem:localization}
For any $\ell_0$ such that $1\leq |\ell_0|\leq q_{n+1}/(12q_n)$, we have
\begin{align*}
r_{\ell_0}\leq e^{(\delta_n-L+52\varepsilon)|\ell_0| q_n}.
\end{align*}
\end{lemma}
\begin{proof}
Without loss of generality, we will prove it for $\ell_0>0$. 
In view of Lemmas \ref{lem:21_ra<final} and Lemma \ref{lem:22_res_final}, for any $0 < |\ell_0|\leq q_{n+1}/(12q_n)$, we have
\begin{align}\label{A}
r_{\ell_0}\leq e^{(\delta_n-L+50\varepsilon)q_n} \max_{\ell_1=\ell_0\pm 1} r_{\ell_1}.
\end{align}
Let an even number $y_n\in\{[q_{n+1}/(6q_n)]-1, [q_{n+1}/(6q_n)]\}$. Let $t_0:=y_n/2$.
One can iterate \eqref{A} until one reaches $\ell_t$ (and stops the iteration once reaches such a $\ell_t$) where either:

{\it Case 1}. $\ell_t=0$ and $t<t_0$

{\it Case 2}. $\ell_t=y_n$ and $t<t_0$

{\it Case 3}. $t=t_0$

Hence one obtains
\begin{align*}
r_{\ell_0}\leq \max_{(\ell_0, \ell_1,...,\ell_t)\in \mathcal{G}} e^{(\delta_n-L+50\varepsilon)tq_n} r_{\ell_t},
\end{align*}
where $\mathcal{G}=\{(\ell_0, ..., \ell_t ): |\ell_i-\ell_{i-1}|=1, \text{ and } \ell_t \text{ satisfies one of the three cases above}\}$.


If $\ell_t$ satisfies Case 1, we have
\begin{align*}
e^{(\delta_n-L+50\varepsilon)tq_n} r_{\ell_t}\leq e^{(\delta_n-L+50\varepsilon)\ell_0 q_n} r_0
\leq &C_0 q_n e^{(\delta_n-L+50\varepsilon)\ell_0 q_n}\\
\leq &e^{(\delta_n-L+52\varepsilon)\ell_0 q_n}.
\end{align*}
where we used \eqref{Shnol} to control $r_0$. We will also use \eqref{Shnol} in the following two cases.

If $\ell_t$ satisfies Case 2, we have
\begin{align*}
e^{(\delta_n-L+50\varepsilon)tq_n} r_{\ell_t}\leq e^{(\delta_n-L+50\varepsilon)(y_n-\ell_0)q_n} r_{y_n}
\leq &C_0 y_n q_n e^{(\delta_n-L+50\varepsilon)(y_n-\ell_0)q_n}\\
\leq &e^{(\delta_n-L+51\varepsilon) y_n q_n} e^{-(\delta_n-L+50\varepsilon) \ell_0 q_n}\\
\leq &e^{(\delta_n-L+51\varepsilon) 2\ell_0q_n} e^{-(\delta_n-L+50\varepsilon) \ell_0 q_n}\\
=&e^{(\delta_n-L+52\varepsilon)\ell_0 q_n}.
\end{align*}

If $\ell_t$ satisfies Case 3, we have
\begin{align*}
e^{(\delta_n-L+50\varepsilon)tq_n} r_{\ell_t}\leq e^{(\delta_n-L+50\varepsilon)t_0 q_n} \max_{1\leq j\leq y_n-1} r_j
\leq &C_0 y_n q_n e^{(\delta_n-L+50\varepsilon)t_0 q_n}\\
\leq &e^{\varepsilon y_nq_n} e^{(\delta_n-L+50\varepsilon)\frac{y_n}{2} q_n}\\
=&e^{(\delta_n-L+52\varepsilon)\frac{y_n}{2} q_n}\\
\leq &e^{(\delta_n-L+52\varepsilon)\ell_0 q_n}.
\end{align*}
Combining the three cases we have proved Lemma \ref{lem:localization}.\qed
\end{proof}

We finally present the proof of Case 3 of Lemma \ref{lem:main}.
\begin{proof}
Let $y\in (\ell q_n+b_n, (\ell+1)q_n-b_n)$ for some $|\ell|\leq \frac{q_{n+1}}{12q_n}$.
We will only prove it for the cases when $\ell=0$ and $\ell\geq 1$. 
We distinguish three cases: 

If $\ell\neq 0, -1$, we have by Lemma \ref{lem:localization} that 
\begin{align*}
r_{\ell}\leq e^{(\delta_n-L+52\varepsilon)|\ell| q_n}, \text{ and } r_{\ell+1}\leq e^{(\delta_n-L+52\varepsilon)|\ell+1| q_n}.
\end{align*}
By Lemmas \ref{lem:C2_n-r} (if $m_n$ is non-resonant) and \ref{lem:22_non-to-res} (if $m_n$ is resonant), we have
\begin{align*}
|\phi(y)|
\leq &e^{30\varepsilon q_n} \max(e^{-(y-\ell q_n)L}r_{\ell}, e^{-((\ell+1)q_n-y)L}r_{\ell+1})\\
\leq &e^{(\delta_n-L+82\varepsilon)|y|}.
\end{align*}

If $\ell=0$, we have by Lemma \ref{lem:localization} that 
\begin{align*}
r_{1}\leq e^{(\delta_n-L+52\varepsilon)q_n}.
\end{align*}
and by \eqref{Shnol} that
\begin{align*}
r_0\leq C_0\tau_n q_n.
\end{align*}
By Lemmas \ref{lem:C2_n-r} and \ref{lem:22_non-to-res}, we have that for $q_n-b_n>y\geq q_n/12$,
\begin{align*}
|\phi(y)|
\leq &e^{30\varepsilon q_n} \max(e^{-y L}r_0, e^{-(q_n-y)L}r_1)\\
\leq &\max(C_0\tau_n q_n e^{30\varepsilon q_n} e^{-y L}, e^{(\delta_n-L+82\varepsilon)y})\\
\leq &\max(e^{-(L-361\varepsilon)y}, e^{(\delta_n-L+82\varepsilon)y})\\
\leq &e^{(\delta_n-L+361\varepsilon)y}.
\end{align*}

If $\ell=-1$, the proof is similar to that of $\ell=0$. We can prove that for $-q_n+b_n<y\leq -q_n/12$, 
\begin{align*}
|\phi(y)|\leq \max(e^{-(L-361\varepsilon)|y|}, e^{(\delta_n-L+82\varepsilon)|y|})\leq e^{(\delta_n-L+361\varepsilon)|y|}.
\end{align*}
Combining the three cases above, we have proved Case 3 of Lemma \ref{lem:main}.
\qed
\end{proof}

\section{Appendix}

\subsubsection*{Proof of Corollary \ref{cor:cos_prod_lower}}
Let $\tilde{I}:=[m_n+jq_n+1, m_n+(j+1)q_n-1]$. 
Let $$I^c:=\tilde{I}\setminus I=[m_n+j q_n+1, \ell_1-1]\cup [\ell_2+1, m_n+(j+1)q_n-1]=:I^c_1\cup I^c_2.$$
We have by Lemma \ref{lana} and equation \eqref{eq:delta_n_min_cos} that
\begin{align*}
\prod_{\ell\in\tilde{I}}|\cos(\pi\theta_{\ell})| \cdot |\cos(\pi\theta_{m_n+jq_n})| 
\geq &\frac{1}{q_n^C} e^{-q_n \ln 2}  \min_{\ell\in \tilde{I}\cup\{m_n+j q_n\}} |\cos(\pi\theta_{\ell})|\notag\\
\geq &\frac{1}{21} \frac{1}{q_n^C} e^{-q_n \ln 2}  |\cos(\pi\theta_{m_n+jq_n})| 
\end{align*}
Hence
\begin{align}\label{eq:cor_prod_1}
\prod_{\ell\in\tilde{I}}|\cos(\pi\theta_{\ell})| \geq \frac{1}{q_n^{C}} e^{-q_n \ln 2}.
\end{align}
Also by Corollary \ref{cor:prod_cos}, for $k\in \{1,2\}$  we have
\begin{align}\label{eq:cor_prod_2}
\prod_{\ell\in I^c_k}|\cos(\pi\theta_{\ell})|\leq C(\varepsilon) e^{|I_k^c|(-\ln 2+\varepsilon)}
\end{align}
Hence combining \eqref{eq:cor_prod_1} and \eqref{eq:cor_prod_2}, we have
\begin{align}\label{eq:cor_prod_3}
\prod_{\ell\in I}|\cos(\pi\theta_{\ell})|=\frac{\prod_{\ell\in \tilde{I}}|\cos(\pi\theta_{\ell})|}{\prod_{\ell\in I^c}|\cos(\pi\theta_{\ell})|}
\geq e^{-\varepsilon(2q_n-|I|)} e^{(-\ln 2)|I|}. 
\end{align}
\qed

\subsubsection*{Proof of Corollary \ref{cor:A_upper}}
By Lemma \ref{lem:upperbddtildeP}, \eqref{tildePinF} we have uniformly in $\theta$,
\begin{align*}
\|F_{|I|}(\theta)\|\leq C(\varepsilon) e^{|I|(\tilde{L}+\varepsilon)}.
\end{align*}
By Corollary \ref{cor:cos_prod_lower} and \eqref{defnonsingular}, we then have by \eqref{LEtildeLE}
\begin{align*}
\|A_{|I|}(\theta_{\ell_1})\|
=\frac{1}{\prod_{\ell\in I} |\cos(\pi\theta_{\ell})|} \|F_{|I|}(\theta_{\ell_1})\|
\leq  C(\varepsilon) e^{\varepsilon (2q_n-|I|)} e^{|I| \ln 2} e^{|I| (\tilde{L}+\varepsilon)}\leq e^{3\varepsilon q_n} e^{L |I|}.
\end{align*}
\qed

\subsubsection*{Proof of Corollary \ref{cor:A_upper_mn}}
We have by Corollary \ref{cor:A_upper} that
\begin{align*}
\|A_{|I|}(\theta_{\ell_1})\|
\leq &\|A_{jq_n+m_n-\ell_1}(\theta_{\ell_1})\|\cdot \|A(\theta_{jq_n+m_n})\| \cdot \| A_{\ell_2-jq_n-m_n}(\theta_{jq_n+m_n+1})\|\\
\leq &C e^{6\varepsilon q_n} e^{L(|I|-1)} |\tan(\pi(\theta_{jq_n+m_n}))|\\
\leq &\frac{1}{c_{n,j}} e^{7\varepsilon q_n} e^{L |I|}.
\end{align*}
\qed

\subsection{Uniformity lemmas}
\subsubsection{Proof of Lemma \ref{lem:nonres_uni}}\

Without loss of generality we assume $d(y, q_n\Z)=y-\ell q_n$ and $y=\ell q_n+2sq_{n-n_0}+x$ where $0\leq x<2q_{n-n_0}$.
We are going to show that for any $\theta\in \T$ and $j\in I_1$ (for $j\in I_2$ the proof is similar) that
\begin{align}\label{nonresdeftopbottom}
  \ln\left\lbrace\prod_{l\neq j} \frac{|\sin\pi(\theta-\theta_l)|}{|\sin\pi(\theta_j-\theta_l)|}\right\rbrace
=\sum_{l\neq j} \ln|\sin\pi(\theta-\theta_l)|-\sum_{l\neq j} \ln|\sin\pi(\theta_j-\theta_l)|< (2sq_{n-n_0}-1)\epsilon.
\end{align}
For $m=1, 2,..., s$, let 
$$T_m=[-[\frac{1}{2}sq_{n-n_0}]-sq_{n-n_0}+(m-1)q_{n-n_0}, -[\frac{1}{2}sq_{n-n_0}]-sq_{n-n_0}+mq_{n-n_0}-1].$$
For $m=s+1,..., 2s$, let 
$$T_m=[y-[\frac{1}{2}sq_{n-n_0}]-sq_{n-n_0}+(m-s-1)q_{n-n_0}, y-[\frac{1}{2}sq_{n-n_0}]+(m-s)q_{n-n_0}-1].$$
Each $T_m$ consists of $q_{n-n_0}$ consequential numbers.
Denote 
$$|\sin\pi(\theta_j-\theta_{l_m})|=\min_{l\in T_m}|\sin\pi (\theta_j-\theta_l)|.$$ 
Assume $j \in T_{m_0}$, then clearly $l_{m_0}=j$.

First, about $\sum_{l\neq j} \ln|\sin\pi(\theta-\theta_l)|$, applying Lemma \ref{lana} on each $T_m$, we have
\begin{align}\label{nonrestopestimate}
\sum_{l\neq i} \ln|\sin\pi(\theta-\theta_l)|
 \leq &2s(C\ln q_{n-n_0}-(q_{n-n_0}-1)\ln 2) \notag\\
 \leq &(2sq_{n-n_0}-1)(-\ln 2+\frac{C\ln{q_{n-n_0}}}{q_{n-n_0}})   
\end{align}

Secondly, about $\sum_{l\neq j} \ln|\sin\pi(\theta_j-\theta_l)|$. Clearly,
\begin{align}\label{defsum1sum2}
\sum_{l\neq j} \ln|\sin\pi(\theta_j-\theta_l)|=\sum_{l\in I_1, l\neq j}\ln{|\sin\pi (j-l)\alpha|}+\sum_{l\in I_2}\ln{|\sin\pi (j-l)\alpha|}\triangleq \sum_{1}+\sum_{2}.
\end{align}

For $\sum_{1}$, by Lemma \ref{lana} we have
\begin{align}
\sum_{1} \geq&s(-C\ln q_{n-n_0}-(q_{n-n_0}-1)\ln2)+\sum_{m=1, m\neq m_0}^{s} \ln|\sin \pi(j-l_m)\alpha| \notag \\
\geq &s(-C\ln q_{n-n_0}-(q_{n-n_0}-1)\ln2)+\sum_{m=1, m\neq m_0}^{s} \ln{\|(j-l_m)\alpha\|_{\T}} \label{sum1first}
\end{align}
Since $|j-l_m|<sq_{n-n_0}<q_{n-n_0+1}$, we have 
\begin{align}
\sum_{m=1, m\neq m_0}^{s} \ln{\|(j-l_m)\alpha\|_{\T}}
\geq &2\sum^{[\frac{s+1}{2}]}_{k=1}\ln{(k\|q_{n-n_0}\alpha\|_{\T})} \notag\\
\geq &2\int_{1}^{\frac{s+1}{2}}\ln{x} \mathrm{d}x+(s+1)\ln{\frac{1}{2q_{n-n_0+1}}} \notag\\
= &(s+1)\ln{(s+1)}-2(s+1)\ln{2}-(s-1)-(s+1)\ln{q_{n-n_0+1}} \notag\\
> &(s+1)\ln(\frac{s}{q_{n-n_0+1}})-Cs \notag\\
\geq &2s\ln(\frac{s}{q_{n-n_0+1}})-Cs. \label{sum1smallest}
\end{align}
Thus combining (\ref{sum1first}) with (\ref{sum1smallest}), we have
\begin{align}\label{sum1final}
\sum_1\geq sq_{n-n_0}\left(-\frac{C \ln q_{n-n_0}}{q_{n-n_0}}-\ln2+\frac{2 \ln(s/q_{n-n_0+1})}{q_{n-n_0}}\right).
\end{align}
For $\sum_{2}$, first consider $\tilde{T}_m=T_m-aq_n$, let $\tilde{l}_m$ be such that $|\sin\pi (j-\tilde{l}_m)\alpha|=\min_{l\in\tilde{T}_m}|\sin\pi (j-l)\alpha|$. Clearly, $\tilde{l}_m=l_m-aq_n$.
Again using Lemma \ref{lana} we have
\begin{align}
\sum_{2} 
\geq&s(-C\ln q_{n-n_0}-(q_{n-n_0}-1)\ln2)+\sum_{m=s+1}^{2s} \ln|\sin\pi(j-l_m)\alpha| \notag \\
\geq&s(-C\ln q_{n-n_0}-(q_{n-n_0}-1)\ln2)+\sum_{m=s+1}^{2s} \ln \|(j-l_m)\alpha\|_{\T}
\end{align}
Note that
\begin{align*}
\|(j-\tilde{\ell}_m)\alpha\|\geq \|q_{n-1}\alpha\|\geq a_{n+1}\|q_n\alpha\|\geq 2a\|q_n\alpha\|.
\end{align*}
Hence
\begin{align}
\sum_{2} \geq s(-C\ln q_{n-n_0}-(q_{n-n_0}-1)\ln2)+\sum_{m=s+1}^{2s} \ln (\|(j-\tilde{l}_m)\alpha\|_{\T}-a\|q_n\alpha\|_{\T}) \label{sum2first}
\end{align} 
Now let $min\triangleq \|(j-\tilde{l}_{m^*})\alpha\|_{\T}=\min_{m=s+1}^{2s}\|(j-\tilde{l}_m)\alpha\|_{\T}$. Note that $|\tilde{l}_{m_1}-\tilde{l}_{m_2}|<sq_{n-n_0}<q_{n-n_0+1}$ for any $m_1, m_2\in \{s+1,..., 2s\}$. We will divide $\sum_{m=s+1}^{2s} \ln (\|(j-\tilde{l}_m)\alpha\|_{\T}-a\|q_n\alpha\|_{\T})$ into two cases:
\begin{enumerate}
\item if $min> \frac{2}{3}\|q_{n-n_0}\alpha\|_{\T}$, then
\begin{align}
        &\sum_{m=s+1}^{2s} \ln (\|(j-\tilde{l}_m)\alpha\|_{\T}-a\|q_n\alpha\|_{\T}) \notag\\
\geq &\ln {(min-a\|q_n\alpha\|_{\T})}+ \sum_{k=1}^{s-1} \ln {(\|kq_{n-n_0}\alpha\|_{\T}+min-a\|q_n\alpha\|_{\T})}\notag \\
\geq &\sum_{k=0}^{s-1} \ln {(\frac{3k+2}{3}\|q_{n-n_0}\alpha\|_{\T}-a\|q_n\alpha\|_{\T})}\notag  \\
\geq &\sum_{k=0}^{s-1} \ln {(\frac{6k+1}{6}\|q_{n-n_0}\alpha\|_{\T})}\notag  \\
=      &\sum_{k=0}^{s-1} \ln{(6k+1)}-s\ln{q_{n-n_0+1}}-Cs \notag\\
\geq &s\ln{(s/q_{n-n_0+1})}-Cs. \label{sumI2smallestcase1}
\end{align}
\item if $min \leq \frac{2}{3}\|q_{n-n_0}\alpha\|_{\T}$, clearly we still have a lower bound $min\geq \|q_{n-1}\alpha\|_{\T}$, then
\begin{align}
        &\sum_{m=s+1}^{2s} \ln (\|(j-\tilde{l}_m)\alpha\|_{\T}-a\|q_n\alpha\|_{\T}) \notag\\
\geq &\ln {(min-a\|q_n\alpha\|_{\T})}+ \sum_{k=1}^{[\frac{s+1}{2}]} \ln {(\|kq_{n-n_0}\alpha\|_{\T}+min-a\|q_n\alpha\|_{\T})}\notag\\
+&\sum_{k=1}^{[\frac{s+1}{2}]}\ln {(\|kq_{n-n_0}\alpha\|_{\T}-min-a\|q_n\alpha\|_{\T})}.\notag\\
\geq&\ln {(\|q_{n-1}\alpha\|_{\T}-a\|q_n\alpha\|_{\T})}+\sum_{k=1}^{[\frac{s+1}{2}]} \ln {(k\|q_{n-n_0}\alpha\|_{\T}-a\|q_n\alpha\|_{\T})} \notag\\
+&\sum_{k=1}^{[\frac{s+1}{2}]}\ln {(\frac{2k-1}{2}\|q_{n-n_0}\alpha\|_{\T}-a\|q_n\alpha\|_{\T})} \notag \\
\geq & \ln{\frac{1}{4q_n}}+2\sum_{k=1}^{[\frac{s+1}{2}]}\ln{\frac{4k-3}{4}}-2[\frac{s+1}{2}]\ln{2q_{n-n_0+1}} \notag\\
\geq &  s\ln{(s/q_{n-n_0+1})}-\ln{q_n}-Cs. \label{sumI2smallestcase2}
\end{align}
Note that 
\begin{align}\label{nonresqnn0}
\tau_n q_n=b_n<\dist{(y, q_n\Z)}<2(s+1)q_{n-n_0},
\end{align} 
we will get the lower bound of $\sum_{m=s+1}^{2s} \ln (\|(j-\tilde{l}_m)\alpha\|_{\T}-a\|q_n\alpha\|_{\T})$ from case (2),
\begin{align}\label{suml2smallestcase2final}
\sum_{m=s+1}^{2s} \ln (\|(j-\tilde{l}_m)\alpha\|_{\T}-a\|q_n\alpha\|_{\T})
\geq &s\ln{(s/q_{n-n_0+1})}-\ln{(2(s+1)q_{n-n_0}/\tau_n)}-Cs \notag\\
 >&s\ln{(s/q_{n-n_0+1})}-\ln{q_{n-n_0}}-Cs-\ln(2/\tau_n).
\end{align}
\end{enumerate}
Thus combining (\ref{sum2first}), (\ref{sumI2smallestcase2}) with (\ref{suml2smallestcase2final}), we have
\begin{align}\label{sum2final}
\sum_{2}\geq sq_{n-n_0}(-C\frac{\ln{q_{n-n_0}}}{q_{n-n_0}}-\ln{2}+\frac{\ln(s/q_{n-n_0+1})}{q_{n-n_0}}).
\end{align}
Combining (\ref{defsum1sum2}), (\ref{sum1final}) with (\ref{sum2final}), we have
\begin{align}\label{nonresonbottomestimate}
\sum_{l\neq j} \ln|\sin\pi(\theta_j-\theta_l)|\geq 2sq_{n-n_0}\left(-C\frac{\ln{q_{n-n_0}}}{q_{n-n_0}}-\ln{2}+\frac{2\ln(s/q_{n-n_0+1})}{q_{n-n_0}}\right).
\end{align}
Eventually, by (\ref{nonresdeftopbottom}), (\ref{nonrestopestimate}) and (\ref{nonresonbottomestimate}),
\begin{align}\label{nonres2rdlast}
\ln\left\lbrace\prod_{l\neq j} \frac{|\sin\pi(\theta-\theta_l)|}{|\sin\pi(\theta_j-\theta_l)|}\right\rbrace< &(2sq_{n-n_0}-1)(C\frac{\ln{q_{n-n_0}}}{q_{n-n_0}}+\frac{2\ln(q_{n-n_0+1}/s)}{q_{n-n_0}}).
\end{align}
Taking into account that $b_n=\tau_n q_n<4sq_{n-n_0}$, and that $q_{n-n_0+1}\leq q_n$, \eqref{nonres2rdlast} yields
\begin{align}\label{nonres2rdlast'}
\ln\left\lbrace\prod_{l\neq j} \frac{|\sin\pi(\theta-\theta_l)|}{|\sin\pi(\theta_j-\theta_l)|}\right\rbrace
< &(2sq_{n-n_0}-1)(C\frac{\ln{q_{n-n_0}}}{q_{n-n_0}}+\frac{2\ln(q_{n}/s)}{q_{n-n_0}}) \notag\\
\leq &(2sq_{n-n_0}-1)(C\frac{\ln{q_{n-n_0}}}{q_{n-n_0}}+\frac{2\ln(q_{n-n_0}/\tau_n)}{q_{n-n_0}})\notag\\
<&(2sq_{n-n_0}-1)\varepsilon.
\end{align}
\qed

\subsubsection{Proof of Lemma \ref{lem:nonres_I2_large}}\

As a corollary of Lemma \ref{lem:nonres_uni} we have
\begin{cor}\label{cor:nonres_uni}
There exists $x_1\in \tilde{I}_0\cup \tilde{I}_y$ such that 
$$|\tilde{P}_{2sq_{n-n_0}-1}(\theta_{x_1})|\geq e^{(\tilde{L}-2\varepsilon)(2sq_{n-n_0}-1)}.$$
\end{cor}
\begin{proof}
Suppose otherwise, we have for any $x_1\in I_1\cup I_2$, 
$$|\tilde{P}_{2sq_{n-n_0}-1}(\theta_{x_1})|< e^{(\tilde{L}-2\varepsilon)(2sq_{n-n_0}-1)}.$$
By \eqref{tildePlagrange}
\begin{align*}
\tilde{P}_{2sq_{n-n_0}-1}(\theta)=\sum_{x_1\in I_1\cup I_2}\tilde{P}_{2sq_{n-n_0}-1}(\theta_{x_1})\prod_{\substack{j\in I_1\cup I_2\\ j\neq x_1}} \frac{\sin\pi(\theta-\theta_j)}{\sin\pi(\theta_{x_1}-\theta_j)}.
\end{align*}
Combining this with Lemma \ref{lem:nonres_uni} yields, uniformly in $\theta$,
\begin{align*}
|\tilde{P}_{2sq_{n-n_0}-1}(\theta)|\leq 2sq_{n-n_0} e^{(\tilde{L}-\varepsilon)(2sq_{n-n_0}-1)}<e^{(\tilde{L}-\frac{\varepsilon}{2})(2sq_{n-n_0}-1)}.
\end{align*}
Hence contradiction with \eqref{averagelower}. \qed
\end{proof}

\begin{lemma}\label{lem:nonres_I1_small}
For $x_1\in \tilde{I}_0$, we have  
$$|\tilde{P}_{2sq_{n-n_0}-1}(\theta_{x_1})|< e^{(\tilde{L}-2\epsilon)(2sq_{n-n_0}-1)}.$$
\end{lemma}
\begin{proof}
Suppose otherwise, we have that for some $x_1\in \tilde{I}_0$, 
\begin{align}\label{eq:P2sq_I1_1}
|\tilde{P}_{2sq_{n-n_0}-1}(\theta_{x_1})|\geq e^{(\tilde{L}-2\epsilon)(2sq_{n-n_0}-1)}.
\end{align}
Let $x_2:=x_1+2sq_{n-n_0}-2$ and $I:=[x_1, x_2]$. Clearly $|I|=2sq_{n-n_0}-1$.

By the Green's formula 
\begin{align}\label{eq:P2sq_I1_2}
|\phi(0)|
\leq &|G_I(x_1, 0)|\cdot |\phi(x_1-1)|+|G_I(x_2, 0)|\cdot |\phi(x_2+1)| \notag\\
=&\frac{|\tilde{P}_{x_2}(\theta_1)|}{|\tilde{P}_I(\theta_{x_1})|} \prod_{j=x_1}^0 |\cos(\pi(\theta_j))|\cdot |\phi(x_1-1)|+
\frac{|\tilde{P}_{-x_1}(\theta_{x_1})|}{|\tilde{P}_I(\theta_{x_1})|} \prod_{j=0}^{x_2} |\cos(\pi(\theta_j))|\cdot |\phi(x_2+1)|.
\end{align}
By \eqref{Shnol}, we have
\begin{align*}
|\phi(x_1-1)|\leq C_0|x_1-1|, \text{ and } |\phi(x_2+1)|\leq C_0|x_2+1|.
\end{align*}
In \eqref{eq:P2sq_I1_2}, use Lemma \ref{lem:upperbddtildeP} to estimate the $\tilde{P}$'s in the numerators, use \eqref{eq:P2sq_I1_1} to estimate the $\tilde{P}$ in denominator, and use Corollary \ref{cor:prod_cos} to estimate the cosine products, we arrive at
\begin{align*}
1\leq  |\phi(0)|\leq  C_0 C(\varepsilon) e^{3\varepsilon |I|} \left(e^{-(1-x_1)L} |1-x_1| +e^{-x_2L} |x_2+1| \right)
  \leq  C(\varepsilon) e^{3\varepsilon |I|} |I| e^{-\frac{|I|}{4}L}<C(\varepsilon) e^{-(\frac{L}{4}-4\varepsilon)|I|}.
\end{align*}
Contradiction, provided that $q_{n-n_0+1}\geq \frac{1}{2} \tau_n q_n$ is sufficiently large which is satisfied for large $n$. 
\end{proof}

A quick combination of Corollary \ref{cor:nonres_uni} and Lemma \ref{lem:nonres_I1_small} yields Lemma \ref{lem:nonres_I2_large}. \qed

\subsubsection{Proof of Lemma \ref{lem:1_res_uni}}
We are going to show that for any $\theta\in \T$ and $k\in I_0$ (for $k\in I_{\ell}$ the proof is similar) 
\begin{align}\label{resdeftopbottom}
  \ln\left\lbrace\prod_{j\neq k} \frac{|\sin\pi(\theta-\theta_j)|}{|\sin\pi(\theta_k-\theta_j)|}\right\rbrace
=\sum_{j\neq k} \ln|\sin\pi(\theta-\theta_j)|-\sum_{j\neq k} \ln|\sin\pi(\theta_k-\theta_j)|< (2q_n-1)(\frac{\ln{q_{n+1}/|\ell|}}{2q_n-1}+\epsilon).
\end{align}
First, about $\sum_{j\neq k} \ln|\sin\pi(\theta-\theta_j)|$, applying Lemma \ref{lana} on $I_0$ and $I_{\ell}$ respectively, we have
\begin{align}\label{restop}
\sum_{j\neq k} \ln|\sin\pi(\theta-\theta_j)|
 \leq 2(C\ln q_{n}-(q_{n}-1)\ln 2).
\end{align}
Secondly, about $\sum_{j\neq k} \ln|\sin\pi(\theta_k-\theta_j)|$. Clearly,
\begin{align}\label{defressum1sum2}
\sum_{j\neq k} \ln|\sin\pi(\theta_k-\theta_j)|=\sum_{j\in I_0, j\neq k}\ln{|\sin\pi (k-j)\alpha|}+\sum_{j\in I_{\ell}}\ln{|\sin\pi (k-j)\alpha|}=: \sum_{1}+\sum_{2}.
\end{align}
For $\sum_{1}$, by Lemma \ref{lana} we have
\begin{align}\label{ressum1}
\sum_{1}\geq -C\ln{q_n}-(q_n-1)\ln{2}.
\end{align}
For $\sum_2$, again by Lemma \ref{lana} we have
\begin{align}\label{ressum2smallestterm}
\sum_{2}\geq -C\ln{q_n}-(q_n-1)\ln{2}+\ln{|\sin{\pi(k-j_0)\alpha}|},
\end{align}
where $\ln{|\sin{\pi(k-j_0)\alpha}|}:= \min_{j\in I_2}\ln{|\sin{\pi(k-j)\alpha}|}$.
Clearly 
\begin{align*}
k-j\in [-(\ell+1)q_n+1, -(\ell-1)q_n-1].
\end{align*} 
Let $j_1\in I_{\ell}$ (it is possible that $j_1=j_0$) be such that $k-j_1=-\ell q_n$. We have for any $j\in I_{\ell}$,
\begin{align*}
\|(k-j)\alpha\|\geq \|(j-j_1)\alpha\|-\|(k-j_1)\alpha\|
\geq &\|q_{n-1}\alpha\|-|\ell|\cdot \|q_n\alpha\|\\
\geq &(a_{n+1}-|\ell|)\|q_n\alpha\|+\|q_{n+1}\alpha\|\\
\geq &(\frac{q_{n+1}}{3q_n}-1)\|q_n\alpha\|\\
\geq &\frac{1}{3}|\ell|\cdot \|q_n\alpha\|.
\end{align*}
Thus
\begin{align}\label{ressum2smallestestimate}
\ln{|\sin{\pi(k-j_0)\alpha}|}
\geq \ln{\|(k-j_0)\alpha\|_{\T}} +\ln 2
\geq \ln{|\ell|\cdot \|q_n\alpha\|_{\T}} +\ln\frac{2}{3}
\geq \ln\frac{|\ell|}{q_{n+1}}+\ln\frac{2}{3}. 
\end{align}
Thus combining (\ref{ressum2smallestterm}) with (\ref{ressum2smallestestimate}), we have
\begin{align}\label{ressum2}
\sum_{2}\geq -(C+2)\ln{q_n}-(q_n-1)\ln{2}+\ln{\frac{|\ell|}{q_{n+1}}}.
\end{align}
Therefore by (\ref{defressum1sum2}), (\ref{ressum1}) and (\ref{ressum2}),
\begin{align}\label{resbottom}
\sum_{j\neq k} \ln|\sin\pi(\theta_k-\theta_j)|\geq 2(-(C+1)\ln{q_n}-(q_n-1)\ln{2})+\ln{\frac{|\ell|}{q_{n+1}}}.
\end{align}
Eventually, by (\ref{resdeftopbottom}), (\ref{restop}) and (\ref{resbottom}) we get
\begin{align*}
 \ln\left\lbrace\prod_{j\neq k} \frac{|\sin\pi(\theta-\theta_j)|}{|\sin\pi(\theta_k-\theta_j)|}\right\rbrace
 \leq \ln{(q_{n+1}/|\ell|)}+\epsilon (4C+2)\ln q_n
 \leq (2q_n-1)(\frac{\ln{q_{n+1}/|\ell|}}{2q_n-1}+\epsilon).
\end{align*}
\qed

\subsection{Proof of Lemmas \ref{lem:C2_n-r} and \ref{lem:22_non-to-res}}
\subsubsection{Proof of Lemma \ref{lem:C2_n-r}}
For $y$ so that $\dist(y, q_n\Z)>b_n$, we have proved that for some $x_1\in I_2$, 
$$|\tilde{P}_{2sq_{n-n_0}-1}(\theta_{x_1})|\geq e^{(\tilde{L}(E)-2\varepsilon)(2sq_{n-n_0}-1)}.$$
Let $z_2=z_1+2sq_{n-n_0}-2$, $I(y)=[z_1, z_2]\cap \Z$ and $\partial I(y)=\{z_1, z_2\}$. 
In general, if $I=[a,b]$, let $\partial I:=\{a,b\}$ and $a':=a-1$, $b':=b+1$.

Let us first consider $y\in I^-$.
By Green's function expansion, we have
\begin{align*}
\phi(y)=\sum_{z\in \partial I(y)}G_{I(y)}(z,y) \phi(z').
\end{align*}
If $x_1-1> \ell q_n+b_n$ or $x_2+1<(\ell+1)q_n-b_n$, we could expand $\phi(x_1-1)$ or $\phi(x_2+1)$. 
We will continue this process until we arrive at a $z$ so that $z\leq \ell q_n+b_n$ or $z\geq (\ell+1)q_n-b_n$, or the iterating number reaches $t_0:=[\frac{24}{\tau_n}]+1$.
We obtain, after a series of expansions, the following
\begin{align*}
\phi(y)=\sum_{\substack{z_1,,,z_t,z_{t+1}\\z_{i+1}\in I(z_i')}}G_{I(y)}(y, z_1)G_{I(z_1')}(z_1', z_2)\cdots G_{I(z_t')}(z_t', z_{t+1})\phi({z_{t+1}'}),
\end{align*}
where $z_{t+1}'$ either satisfies 

{\it Case 1}\ : $\ell q_n\leq z_{t+1}'\leq \ell q_n+b_n$ and $t<t_0$ or 

{\it Case 2}\ : $(\ell+1)q_n\geq z_{t+1}'\geq (\ell+1)q_n-b_n$ and $t<t_0$ or 

{\it Case 3}\ : $t=t_0$. 

For simplicity, let us denote $y=z_0'$.

If $z_{t+1}'$ satisfies Case 1. For each $z_j'$, $0\leq j\leq t$, denote $\partial I(z_j')=\{z_{j+1}, y_{j+1}\}$. Assume WLOG $z_{j+1}<y_{j+1}$, we have
\begin{align*}
|G_{I(z_j')}(z_j', z_{j+1})|
=\frac{|P_{|y_{j+1}-z_j'|}(\theta+(z_j'+1)\alpha)|}{|P_{|I(z_j')|}(\theta+z_{j+1}\alpha)|}
=&\frac{|\tilde{P}_{|y_{j+1}-z_j'|}(\theta+(z_j'+1)\alpha)|}{|\tilde{P}_{|I(z_j')|}(\theta+z_{j+1}\alpha)|} 
\prod_{\ell=z_{j+1}}^{z_j'}|\cos(\pi(\theta+\ell\alpha))|.
\end{align*}
We have by Lemma \ref{lem:nonres_I2_large} that
\begin{align*}
|\tilde{P}_{|I(z_j')|}(\theta+z_{j+1}\alpha)|\geq e^{|I(z_j')|(\tilde{L}-2\varepsilon)},
\end{align*}
by Lemma \ref{lem:upperbddtildeP} that
\begin{align*}
|\tilde{P}_{|y_{j+1}-z_j'|}(\theta+(z_j'+1)\alpha)|\leq C(\varepsilon) e^{|y_{j+1}-z_j'|(\tilde{L}+\varepsilon)},
\end{align*}
and by \ref{cor:prod_cos},
\begin{align*}
\prod_{\ell=z_{j+1}}^{z_j'}|\cos(\pi(\theta+\ell\alpha))|\leq e^{|z_j'-z_{j+1}|(-\ln 2+\varepsilon)}.
\end{align*}
Putting them all together, we have
\begin{align}\label{eq:non-res-Green-general}
|G_{I(z_j')}(z_j', z_{j+1})|\leq C(\varepsilon) e^{-|z_j'-z_{j+1}+1|(L-12\varepsilon)}.
\end{align}
Bounding $|\phi(z_{t+1}')|\leq r_a$, we have
\begin{align}\label{eq:non-res-Green-general_1}
|G_{I(y)}(y, z_1)G_{I(z_1')}(z_1', z_2)\cdots G_{I(z_t')}(z_t', z_{t+1})\phi({z_{t+1}'})|
\leq &(C(\varepsilon))^{t_0+1} e^{-(y-\ell q_n-b_n)(L-12\varepsilon)} r_{\ell} \notag\\
\leq &(C(\varepsilon))^{t_0+1} e^{\varepsilon q_n}e^{-(y-\ell q_n)(L-12\varepsilon)} r_{\ell}.
\end{align}

If $z_{t+1}'$ satisfies Case 2, there must be a $z_j'$ such that $aq_n+m_n\in I(z_j')$.
For this particular pair, let $\partial I(z_j')=\{z_{j+1}, y_{j+1}\}$ and $y_{j+1}<z_{j+1}$, we estimate similarly to \eqref{eq:non-res-Green-general}, only modifying the estimate of the cosine product, 
\begin{align*}
|G_{I(z_j')}(z_j', z_{j+1})|
=&\frac{|P_{|y_{j+1}-z_j'|}(\theta+y_{j+1}\alpha)|}{|P_{|I(z_j')|}(\theta+y_{j+1}\alpha)|}\\
=&\frac{|P_{|y_{j+1}-z_j'|}(\theta+y_{j+1}\alpha)|}{|P_{|I(z_j')|}(\theta+y_{j+1}\alpha)|} \prod_{\ell=z_j'}^{z_{j+1}}|\cos(\pi(\theta+\ell\alpha))|\\
\leq &C(\varepsilon) e^{|y_{j+1}-z_j'|(L+\varepsilon)} e^{-|I(z_j')|(L-2\varepsilon)} e^{|z_j'-z_{j+1}|(-\ln 2+\varepsilon)} c_{n,\ell}\\
\leq &C(\varepsilon) e^{-|z_j'-z_{j+1}+1|(L-12\varepsilon)} c_{n,\ell}.
\end{align*}
Hence
\begin{align}\label{eq:non-res-Green-special}
|G_{I(y)}(y, z_1)G_{I(z_1')}(z_1', z_2)\cdots G_{I(z_t')}(z_t', z_{t+1})\phi({z_{t+1}'})|
\leq (C(\varepsilon))^{t_0+1} e^{\varepsilon q_n} e^{-((\ell+1)q_n-y)(L-12\varepsilon)} c_{n,\ell} r_{\ell+1}.
\end{align}

If $z_{t+1}'$ satisfies Case 3. Here 
\begin{align*}
|z_j'-z_{j+1}|\geq \frac{1}{4} I(z_j')\geq \frac{1}{8} \min((z_j'-\ell q_n, (\ell+1)q_n-z_j'))\geq \frac{b_n}{8}=\frac{\tau_n}{8} q_n.
\end{align*} 
In general, we use the Green's function estimate \eqref{eq:non-res-Green-general}, which yields
\begin{align}\label{eq:non-res_101}
|G_{I(z_j')}(z_j', z_{j+1})|\leq C(\varepsilon) e^{-\frac{1}{8}\tau_n q_n(L-12\varepsilon)}.
\end{align}
If $z_{t+1}'\in I^-$, we control $|\phi(z_{t+1}')|\leq |\phi(x_0^-)|$.

If $z_{t+1}'\in I^+$, we control $|\phi(z_{t+1}')|\leq |\phi(x_0^+)|$. 
Furthermore, we know from \eqref{eq:non-res-Green-special}, that one of the Green's function satisfies additional decay, 
\begin{align}\label{eq:non-res_102}
|G_{I(z_j')}(z_j', z_{j+1})|\leq C(\varepsilon) e^{-\frac{1}{8}\tau_n q_n(L-12\varepsilon)} c_{n,\ell}.
\end{align}
We control $|\phi(z_{t+1}')$ as follows
\begin{align}\label{eq:non-res_103}
|\phi(z_{t+1}')|\leq 
\begin{cases}
|\phi(x_0^-)|, \text{ if } z_{t+1}'\in I^-\\
|\phi(\ell q_n+m_n)|, \text{ if } z_{t+1}'=\ell q_n+m_n\\
|\phi(x_0^+), \text{ if } z_{t+1}'\in I^+
\end{cases}
\end{align}
Hence combining \eqref{eq:non-res_101}, \eqref{eq:non-res_102} and \eqref{eq:non-res_103}, we have
\begin{align}\label{eq:non-res_104_t0}
&|G_{I(y)}(y, z_1)G_{I(z_1')}(z_1', z_2)\cdots G_{I(z_t')}(z_t', z_{t+1})\phi({z_{t+1}'})| \notag\\
&\leq (C(\varepsilon))^{t_0} e^{-\frac{1}{8}\tau_n t_0 (L-12\varepsilon)} \max
\{|\phi(x_0^-)|, |\phi(\ell q_n+m_n)|, c_{n,\ell} |\phi(x_0^+)|\} \notag\\
&\leq e^{-3q_n(L-12\varepsilon)} \max
\{|\phi(x_0^-)|, |\phi(\ell q_n+m_n)|, c_{n,\ell} |\phi(x_0^+)|\}.
\end{align}
Taking into account all the three cases \eqref{eq:non-res-Green-general_1}, \eqref{eq:non-res-Green-special} and \eqref{eq:non-res_104_t0}, we have proved that for $y\in I^-$,
\begin{align}\label{eq:non-res-I-}
|\phi(y)|\leq (C(\varepsilon))^{t_0} \max(& e^{\varepsilon q_n} e^{-(y-\ell q_n)(L-12\varepsilon)} r_{\ell}, e^{\varepsilon q_n}  e^{-((\ell+1)q_n-y)(L-12\varepsilon)} c_{n,\ell} r_{\ell+1},\\
&e^{-3q_n (L-12\varepsilon)} 
\max(|\phi(x_0^-)|, |\phi(\ell q_n+m_n)|, c_{n,\ell} |\phi(x_0^+)|) ). \notag
\end{align}
Similarly, one can show that for $y\in I_+$,
\begin{align}\label{eq:non-res-I+}
|\phi(y)|\leq (C(\varepsilon))^{t_0} \max(&e^{\varepsilon q_n}  e^{-(y-\ell q_n)(L-12\varepsilon)} c_{n,\ell} r_{\ell}, e^{\varepsilon q_n}  e^{-((\ell+1)q_n-y)(L-12\varepsilon)} r_{\ell+1},\\
&e^{-3q_n (L-12\varepsilon)} 
\max(c_{n,\ell}|\phi(x_0^-)|, |\phi(\ell q_n+m_n)|, |\phi(x_0^+)|) ). \notag
\end{align}
and 
\begin{align}\label{eq:non-res-I-mid}
|\phi(\ell q_n+m_n)|\leq (C(\varepsilon))^{t_0} c_{n,\ell} \max(&e^{\varepsilon q_n}  e^{-(y-\ell q_n)(L-12\varepsilon)} r_{\ell}, e^{\varepsilon q_n}  e^{-((\ell+1)q_n-y)(L-12\varepsilon)} r_{\ell+1},\\
&e^{-3q_n (L-12\varepsilon)} 
\max(\phi(x_0^-)|, |\phi(\ell q_n+m_n)|, |\phi(x_0^+)|) ). \notag
\end{align}
Combining \eqref{eq:non-res-I-} applied to $y=x_0^-$, \eqref{eq:non-res-I-} applied to $y=x_0^+$, with \eqref{eq:non-res-I-mid}, we obtain that
\begin{align*}
\max(|\phi(x_0^-)|, |\phi(x_0^+)|, |\phi(aq_n+m_n)|)\leq (C(\varepsilon))^{t_0} \max(r_{\ell}, r_{\ell+1}).
\end{align*}
Plugging this back into \eqref{eq:non-res-I-}, \eqref{eq:non-res-I+}, \eqref{eq:non-res-I-mid}, we obtain the claimed result. 
\qed

\subsubsection*{Proof of Lemma \ref{lem:22_non-to-res}}
The proof of this lemma is almost identical to that of Lemma \ref{lem:C2_n-r}. We only give a brief proof.
By Green's function expansion, we have
\begin{align*}
\phi(y)=\sum_{z\in \partial I(y)}G_{I(y)}(z,y) \phi(z').
\end{align*}
If $x_1-1> \ell q_n+b_n$ or $x_2+1<(\ell+1)q_n-b_n$, we continue to expand $\phi(x_1-1)$ or $\phi(x_2+1)$. 
We repeat this process until we arrive at a $z$ so that $z\leq \ell q_n+b_n$ or $z\geq (\ell+1)q_n-b_n$, or the iterating number reaches $t_0:=[24/\tau_n]+1$.
We obtain, after a series of expansions, the following
\begin{align*}
\phi(y)=\sum_{s; z_{i+1}\in I(z_i')}G_{I(y)}(y, z_1)G_{I(z_1')}(z_1', z_2)\cdots G_{I(z_t')}(z_t', z_{t+1})\phi({z_{t+1}'}),
\end{align*}
where $z_{t+1}'$ either satisfies

{\it Case 1}\ :$z_{t+1}'\in R_{\ell}^+\cup \{\ell q_n+m_n\}$ and $t<t_0$ or 

{\it Case 2}\ :$z_{t+1}'\in R_{\ell}^-$ and $t<t_0$ or

{\it Case 3}\ : $z_{t+1}'\in R_{\ell+1}^-$ and $t<t_0$ or 

{\it Case 4}\ : $t=t_0$. 

If $z_{t+1}'$ satisfies Case 1. 
One can follow the proof of Case 1 of Lemma \ref{lem:C2_n-r}.
Bounding $|\phi(z_{t+1}')|\leq \max(r_{\ell}^+, |\phi(\ell q_n+m_n)|)$, we have, similar to \eqref{eq:non-res-Green-general_1}
\begin{align}\label{eq:non-res-Green-general_111}
|G_{I(y)}(y, z_1)G_{I(z_1')}(z_1', z_2)\cdots G_{I(z_t')}(z_t', z_{t+1})\phi({z_{t+1}'})|\leq (C(\varepsilon))^{t_0+1} e^{\varepsilon q_n}  e^{-(y-\ell q_n)(L-12\varepsilon)}\max(r_{\ell}^+, |\phi(\ell q_n+m_n)|).
\end{align}
Note that by the eigenvalue equation, 
$$|\phi(\ell q_n+m_n)|\leq Cc_{n,\ell}\max(|\phi(\ell q_n+m_n-1)|, |\phi(\ell q_n+m_n+1)|)\leq Cc_{n,\ell}e^{5\varepsilon q_n}\max(r_{\ell}^-, r_{\ell}^+).$$
Hence \eqref{eq:non-res-Green-general_111} yields
\begin{align}\label{eq:non-res-Green-general_111'}
|G_{I(y)}(y, z_1)G_{I(z_1')}(z_1', z_2)\cdots G_{I(z_t')}(z_t', z_{t+1})\phi({z_{t+1}'})|\leq (C(\varepsilon))^{t_0+1} e^{\varepsilon q_n}  e^{-(y-\ell q_n)(L-12\varepsilon)}\max(r_{\ell}^+, Cc_{n,\ell}e^{5\varepsilon q_n}r_{\ell}^-).
\end{align}

If $z_{t+1}'$ satisfies Case 2, there must be a $z_j'$ such that $aq_n+m_n\in I(z_j')$.
Similar to the proof of Case 2 of Lemma \ref{lem:C2_n-r}, we have
\begin{align}\label{eq:non-res-Green-special_222}
|G_{I(y)}(y, z_1)G_{I(z_1')}(z_1', z_2)\cdots G_{I(z_t')}(z_t', z_{t+1})\phi({z_{t+1}'})|\leq (C(\varepsilon))^{t_0+1} e^{\varepsilon q_n} e^{-(y-\ell q_n)(L-12\varepsilon)} c_{n,\ell} r_{\ell}^-.
\end{align}

If $z_{t+1}'$ satisfies Case 3. Estimating similar to Case 1, we have
\begin{align}\label{eq:non-res-Green-general_333}
|G_{I(y)}(y, z_1)G_{I(z_1')}(z_1', z_2)\cdots G_{I(z_t')}(z_t', z_{t+1})\phi({z_{t+1}'})|\leq (C(\varepsilon))^{t_0+1} e^{\varepsilon q_n}  e^{-((\ell+1) q_n-y)(L-12\varepsilon)}.
\end{align}

If $z_{t+1}'$ satisfies Case 4. Estimating similar to Case 3 of the proof of Lemma \ref{lem:C2_n-r}, we have
\begin{align}\label{eq:non-res_104_t0_444}
|G_{I(y)}(y, z_1)G_{I(z_1')}(z_1', z_2)\cdots G_{I(z_t')}(z_t', z_{t+1})\phi({z_{t+1}'})| \leq e^{-3q_n(L-12\varepsilon)} |\phi(x_0)|,
\end{align}
where $|\phi(x_0)|:=\sup_{\ell q_n+b_n<y<(\ell+1)q_n-b_n} |\phi(y)|$.
Taking into account that (by Corollary \ref{cor:A_upper}),
\begin{align*}
|\phi(x_0)|\leq e^{3\varepsilon q_n} e^{q_nL} \max(r_{\ell}^+, r_{\ell+1}^-).
\end{align*}
\eqref{eq:non-res_104_t0_444} yields
\begin{align}\label{eq:non-res_104_t0_555}
|G_{I(y)}(y, z_1)G_{I(z_1')}(z_1', z_2)\cdots G_{I(z_t')}(z_t', z_{t+1})\phi({z_{t+1}'})|\leq e^{(-2L+40\varepsilon)q_n}\max(r_{\ell}^+, r_{\ell+1}^-).
\end{align}

Eventually combining the four cases \eqref{eq:non-res-Green-general_111}, \eqref{eq:non-res-Green-special_222}, \eqref{eq:non-res-Green-general_333} and \eqref{eq:non-res_104_t0_555}, we have
\begin{align}\label{eq:non-res-Green-666}
|\phi(y)|\leq (C(\varepsilon))^{t_0} e^{18\varepsilon q_n} \max(e^{-(y-\ell q_n)L}r_{\ell}^+, c_{n,\ell}e^{-(y-\ell q_n)L}r_{\ell}^-, e^{-((\ell+1)q_n-y)L}r_{\ell+1}^-).
\end{align}
By Corollary \ref{cor:A_upper_mn}, we have
\begin{align*}
r_{\ell}^-\leq e^{9\varepsilon q_n} \frac{1}{c_{n,\ell}} r_{\ell}^+.
\end{align*}
Hence \eqref{eq:non-res-Green-666} yields
\begin{align}\label{eq:non-res-Green-777}
|\phi(y)|\leq e^{30\varepsilon q_n} \max(e^{-(y-\ell q_n)L}r_{\ell}^+, e^{-((\ell+1)q_n-y)L}r_{\ell+1}^-).
\end{align}
This proves the claimed result.
\qed

\subsection{Proof of Case 1 of Lemma \ref{lem:main}}\
We divide into two cases depending if $k<q_n/2$.

Case 1. If $\frac{q_n}{12}< k< \frac{q_n}{2}$. \

The proof of this lemma is similar to that of Lemma \ref{lem:C2_n-r}. We only give a brief proof.
By Green's function expansion, we have
\begin{align*}
\phi(k)=\sum_{z\in \partial I(k)}G_{I(k)}(z,k) \phi(z').
\end{align*}
If $x_1-1>b_n$ or $x_2+1<q_n-b_n$, we continue to expand $\phi(x_1-1)$ or $\phi(x_2+1)$. 
We repeat this process until we arrive at a $z$ so that $z\leq b_n$ or $z\geq q_n-b_n$, or the iterating number reaches $t_0:=[24/\tau_n]+1$.
We obtain, after a series of expansions, the following
\begin{align*}
\phi(k)=\sum_{s; z_{i+1}\in I(z_i')}G_{I(k)}(k, z_1)G_{I(z_1')}(z_1', z_2)\cdots G_{I(z_t')}(z_t', z_{t+1})\phi({z_{t+1}'}),
\end{align*}
where $z_{t+1}'$ either satisfies

{\it Case (i)}\ : $z_{t+1}'\leq b_n$ and $t<t_0$ or 

{\it Case (ii)}\ : $z_{t+1}'\geq q_n-b_n$ and $t<t_0$ or

{\it Case (iii)}\ : $t=t_0$. 

If $z_{t+1}'$ satisfies Case (i). 
One can follow the proof of Case 1 of Lemma \ref{lem:C2_n-r}.
Bounding $|\phi(z_{t+1}')|\leq r_0$, we have, similar to \eqref{eq:non-res-Green-general_1}
\begin{align}\label{eq:non-res-Green-general_111_dio}
|G_{I(k)}(k, z_1)G_{I(z_1')}(z_1', z_2)\cdots G_{I(z_t')}(z_t', z_{t+1})\phi({z_{t+1}'})|\leq (C(\varepsilon))^{t_0+1} e^{\varepsilon q_n}  e^{-k(L-12\varepsilon)}r_0.
\end{align}

If $z_{t+1}'$ satisfies Case (ii), similar to Case (i) above, bounding $|\phi(z_{t+1}')|\leq r_1$, we have, similar to \eqref{eq:non-res-Green-general_1}
\begin{align}\label{eq:non-res-Green-special_222_dio}
|G_{I(k)}(k, z_1)G_{I(z_1')}(z_1', z_2)\cdots G_{I(z_t')}(z_t', z_{t+1})\phi({z_{t+1}'})|\leq (C(\varepsilon))^{t_0+1} e^{\varepsilon q_n} e^{-(q_n-k)(L-12\varepsilon)} r_1.
\end{align}

If $z_{t+1}'$ satisfies Case (iii). Estimating similar to Case 3 of the proof of Lemma \ref{lem:C2_n-r}, we have
\begin{align}\label{eq:non-res_104_t0_444_dio}
|G_{I(k)}(k, z_1)G_{I(z_1')}(z_1', z_2)\cdots G_{I(z_t')}(z_t', z_{t+1})\phi({z_{t+1}'})| \leq e^{-3q_n(L-12\varepsilon)} |\phi(x_0)|,
\end{align}
where $|\phi(x_0)|:=\sup_{b_n<y<q_n-b_n} |\phi(y)|$.
Taking \eqref{Shnol} into account, we have
\begin{align}\label{eq:dio_333}
\max(r_0, r_1, |\phi(x_0)|)\leq C_0 q_n.
\end{align}
Hence combining \eqref{eq:dio_333} with \eqref{eq:non-res-Green-general_111_dio}, \eqref{eq:non-res-Green-special_222_dio} and \eqref{eq:non-res_104_t0_444}, we obtain
\begin{align}\label{eq:non-res-Green-666_dio}
|\phi(k)|\leq e^{2\varepsilon q_n} e^{-k (L-12\varepsilon)}\leq e^{-k(L-36\varepsilon)}.
\end{align}
where we used $q_n<12k$.

Case 2. If $\frac{q_n}{2}< k< \frac{q_{n+1}}{12}$. \

For $y$ such that $\frac{q_n}{2}< y< \frac{q_{n+1}}{6}$,
let $s$ be the smallest positive integer such that 
\begin{align*}
(2s-1-\frac{1}{2})q_n\leq y<(2s+1-\frac{1}{2})q_n,
\end{align*}
and 
\begin{align*}
I_1:= [-2sq_n+[sq_n/2]+1, -sq_n+[sq_n/2] ], \text{ and } I_2:= [y-[sq_n/2]-sq_n+1, y-[sq_n/2]].
\end{align*}
It is easy to see that $I_1\cap I_2=\emptyset$.
Next we show
\begin{lemma}\label{lem:Dio_uni}
For $n$ large enough, $\{\theta_{\ell}\}_{\ell\in I_1\cup I_2}$ is $153\varepsilon$-uniform.
\end{lemma}
\begin{proof}
We divide the $2sq_n$ points into $2s$ intervals: $T_1, \cdots, T_{2s}$, each containing $q_n$ points.
Fixing any $j$. For $1\leq w\leq 2s$, let 
$$|\sin\pi(\theta_j-\theta_{{\ell}_w})|:=\min_{\ell\in T_w} |\sin\pi(\theta_j-\theta_{\ell})|.$$ 
Without loss of generality, assume $j \in T_{s_0}$ for some $s_0$ such that $1\leq s_0 \leq s$. 

By Lemma \ref{lana}, we have
\begin{align}\label{Dio:3}
\sum_{\ell \neq j} \ln|\sin\pi(\theta-\theta_{\ell})| \leq 2s (C\ln q_n-(q_n-1)\ln 2) \leq 2sq_n(-\ln 2+\varepsilon),
\end{align}
and
\begin{align}\label{Dio:4}
 & \sum_{\ell \neq j} \ln|\sin\pi(\theta_j-\theta_{\ell})| \notag \\
\geq &2s(-C\ln q_n-(q_n-1)\ln2)+\sum_{w=1, w\neq s_0}^s \ln|\sin\pi(\theta_j-\theta_{\ell_w})|
          +\sum_{w=s+1}^{2s} \ln|\sin\pi(\theta_i-\theta_{\ell_w})| \notag\\
\geq &2sq_n(-\ln 2-\varepsilon)+\sum_{w=1, w\neq s_0}^s \ln \|(j-\ell_w)\alpha\|
          +\sum_{w=s+1}^{2s} \ln \|(j-\ell_w)\alpha\|
\end{align}
We have
\begin{align}\label{eq:klarge}
|j-\ell_w|< y+sq_n<3y<q_{n+1},
\end{align}
hence for each $w\neq s_0$, we have by \eqref{eq:cont2} and \eqref{def:betan} that
\begin{align}\label{Dio:5}
\|(j-\ell_w)\alpha\|\geq \|q_n\alpha\|\geq\frac{1}{2q_{n+1}}=\frac{1}{2}e^{-\beta_n q_n}\geq \frac{1}{2}e^{-300\varepsilon q_n}.
\end{align}
Combining \eqref{Dio:3}, \eqref{Dio:4} and \eqref{Dio:5}, we have
\begin{align}
\sum_{\ell \neq j} \ln|\sin\pi(\theta-\theta_{\ell})|-\sum_{\ell \neq j} \ln|\sin\pi(\theta_j-\theta_{\ell})|<305\varepsilon sq_n.
\end{align}
\qed
\end{proof}

Lemma \ref{lem:Dio_uni} implies the following, similar to Lemma \ref{lem:nonres_I2_large},
\begin{lemma}
For $n$ large enough, there exists $x_1\in I_2$ such that that 
\begin{align}\label{Dio:6}
|\tilde{P}_{2sq_n-1}(\theta_{x_1})|\geq e^{(\tilde{L}-155\varepsilon)(2sq_n-1)}.
\end{align}
\end{lemma}
Let $x_2=x_1+2sq_n-2$ and $I(k):=[x_1, x_2]$.
Plugging \eqref{Dio:6} into the Green's formula \eqref{Green_tildeP}, and using estimates from Lemma \ref{lem:upperbddtildeP} and Corollary \ref{cor:prod_cos}, we have that for $\frac{q_n}{2}<k<\frac{q_{n+1}}{12}$,
\begin{align}\label{Dio2}
|\phi(k)|
\leq &\sum_{z\in \partial I(k)} e^{-(L-625\varepsilon)|k-z|} |\phi(z')|.
\end{align}
Iterating this process for $\phi(z')$ until we arrive at a $z'$ such that $z'\leq \max(\gamma k, q_n/2)$ or $z'\geq 2k$ or the iteration number $t$ reaches $t_0:=[5/\gamma]+1$, where $\gamma$ is a small positive constant such that
$$(L-625\varepsilon)(1-\gamma)=L-626\varepsilon.$$
We obtain, after a series of expansions, the following
\begin{align*}
|\phi(k)|\leq \sum_{s; z_{i+1}\in I(z_i')} e^{-(L-625\varepsilon)(|k-z_1|+|z_1'-z_2|+...+|z_t'-z_{t+1}|)} |\phi({z_{t+1}'})|,
\end{align*}
where $z_{t+1}'$ either satisfies

{\it Case (i)}\ : $z_{t+1}'\leq \max(\gamma k, \frac{q_n}{2})$ and $t<t_0$ or 

{\it Case (ii)}\ : $z_{t+1}'\geq 2k$ and $t<t_0$ or

{\it Case (iii)}\ : $t=t_0$. \

If $z_{t+1}'\leq \frac{q_n}{2}$, we bound $|\phi(z_{t+1}')|$ by \eqref{eq:non-res-Green-666_dio}, which is
\begin{align}\label{Dio33}
|\phi(z_{t+1}')|\leq e^{-(L-36\varepsilon)z_{t+1}'},
\end{align}
and hence
\begin{align}\label{Dio33'}
e^{-(L-625\varepsilon)(|k-z_1|+|z_1'-z_2|+...+|z_t'-z_{t+1}|)} |\phi({z_{t+1}'})|\leq &e^{-(L-625\varepsilon)(k-z_{t+1})} e^{-(L-36\varepsilon)z_{t+1}'}
\leq e^{-(L-625\varepsilon)k}.
\end{align}

If $z_{t+1}'\leq \gamma k$, bounding $|\phi(z_{t+1}')|\leq C_0 \gamma k$ by \eqref{Shnol}, 
we obtain
\begin{align}\label{Dio33''}
e^{-(L-625\varepsilon)(|k-z_1|+|z_1'-z_2|+...+|z_t'-z_{t+1}|)} |\phi({z_{t+1}'})|\leq &C_0 \gamma k e^{-(L-625\varepsilon)(k-z_{t+1})}  \notag\\
\leq &C_0 \gamma k e^{-(L-625\varepsilon)(1-\gamma)k} \notag\\
\leq &e^{-(L-627\varepsilon)k}.
\end{align}

If $z_{t+1}'$ satisfies Case (ii), bounding $|\phi(z_{t+1}')|\leq C_0 z_{t+1}'$ by \eqref{Shnol}, we obtain
\begin{align}\label{Dio44'}
e^{-(L-625\varepsilon)(|k-z_1|+|z_1'-z_2|+...+|z_t'-z_{t+1}|)} |\phi({z_{t+1}'})|\leq &e^{-(L-625\varepsilon)|k-z_{t+1}|} C_0 z_{t+1}'\notag\\
\leq &e^{-(L-626\varepsilon)|k-z_{t+1}|}\leq e^{-(L-626\varepsilon)k}.
\end{align}

If $z_{t+1}'$ satisfies Case (iii), we bound $|\phi(z_{t+1}')|\leq C_0z_{t+1}'\leq 2C_0 k$ by \eqref{Shnol}.
Further we bound each $|z_j'-z_{j+1}|$, denoting for simplicity $k=z_0'$, in the following way.
For $z_j'\geq \gamma k$ satisfying 
\begin{align*}
\max(\gamma k, (2s-1-\frac{1}{2})q_n)\leq z_j'<(2s+1-\frac{1}{2})q_n,
\end{align*}
we have
\begin{align*}
|z_j'-z_{j+1}|\geq \frac{1}{2}sq_n\geq \frac{s}{4s+1} \gamma k\geq \frac{1}{5}\gamma k.
\end{align*}
Hence we have by \eqref{def:betan} that
\begin{align}\label{Dio55}
e^{-(L-625\varepsilon)(|k-z_1|+|z_1'-z_2|+...+|z_t'-z_{t+1}|)} |\phi({z_{t+1}'})|
\leq &2C_0 k e^{-(L-625\varepsilon) \frac{t_0}{5}\gamma k}\notag\\
\leq &e^{-(L-626\varepsilon)k}.
\end{align}
Summarizing \eqref{Dio33'}, \eqref{Dio33''}, \eqref{Dio44'} and \eqref{Dio55}, we have
\begin{align}\label{Dio66}
|\phi(k)|\leq 2^{t_0}e^{-(L-627\varepsilon)k}\leq e^{-(L-630\varepsilon) k}.
\end{align}
\qed

\section*{Acknowledgement}
R. H. is partially supported by NSF-DMS-2053285. S.J. was a 2020-21 Simons fellow. Her work
was also partially supported by NSF DMS-2052899, DMS-2155211, and Simons 681675. F. Y. is partially supported by an AMS-Simons Travel Grant. 
R. H. and F. Y. thank the hospitality of University of California,
Irvine during summer 2017, when the key work of this paper was done and the work on \cite{HJY} was started.

\bibliographystyle{amsplain}

\begin{thebibliography}{10}
\bibitem{aw} M. Aizenman and S. Warzel. Resonant delocalization for random Schr\"odinger operators on tree graphs. Journal of the European Mathematical Society, 15(4):1167-–1222, 2013.
\bibitem {AJ1} A. Avila,  S. Jitomirskaya. The ten Martini problem. Annals of Mathematics, 170 (2009): 303-342.

\bibitem{solving}Avila, A., and Jitomirskaya, S. (2006). Solving the ten martini problem. In Mathematical physics of quantum mechanics (pp. 5-16). Springer, Berlin, Heidelberg.

\bibitem{AJM} Avila, A., Jitomirskaya, S., and Marx, C. A. (2017). Spectral theory of extended Harper's model and a question by Erd\"os and Szekeres. Inventiones mathematicae, 210(1), 283-339.

\bibitem{ayz}Avila, A., You, J., and Zhou, Q. (2017). Sharp phase transitions for the almost Mathieu operator. Duke Mathematical Journal, 166(14), 2697-2718.

\bibitem{as}Avron, J., Simon, B.: Singular continuous spectrum for a class of almost
periodic Jacobi matrices. Bull. AMS 6, 81-85 (1982).

\bibitem{graphene}Becker, S., Han, R., and Jitomirskaya, S. (2019). Cantor spectrum of graphene in magnetic fields. Inventiones mathematicae, 218(3), 979-1041.

\bibitem{berez} Y. M. Berezansky. Expansions in eigenfunctions of selfadjoint operators. American Mathematical Society, (1968), Providence, RI.

\bibitem{berry} M. Berry. Incommensurability in an exactly-soluble quantal and classical model for a kicked rotator.
 Physica D, 10 (1984): 369–378.
 

\bibitem{fs} A. Fedotov and F. Sandomirskiy. An exact renormalization formula for the Maryland model. Communications in Mathematical Physics 334.2 (2015): 1083-1099.

\bibitem{fp1984} A. L. Figotin, L.A.Pastur. An exactly solvable model of a multidimensional incommensurate structure. Communications in Mathematical Physics 95.4 (1984): 401-425.

\bibitem{fishman2010} S. Fishman. Anderson localization and quantum chaos maps. Scholarpedia, (2010) 5(8):9816.

\bibitem{Furman}A. Furman, On the multiplicative ergodic theorem for uniquely ergodic systems Annales de l'Institut Henri
Poincare (B) Probability and Statistics. No longer published by Elsevier, 33(6): 797-815 (1997).

\bibitem{GKDS2014}S. Ganeshan, K. Kechedzhi, and S. Das Sarma. Critical integer quantum hall topology and the integrable maryland model as a topological quantum critical point. Physical Review B 90.4 (2014): 041405.

\bibitem{g}A. Gordon, \textit{The point spectrum of the one-dimensional Schr\"odinger operator}, Uspehi Mat. Nauk 31, 257-258 (1976).

\bibitem{fgp}D. Grempel, S. Fishman, and R. Prange. Localization in an incommensurate potential: An exactly solvable model. Physical Review Letters, 49.11(1982): 833.


\bibitem{dryten}Han, R. (2018). Dry Ten Martini problem for the non-self-dual extended Harper's model. Transactions of the American Mathematical Society, 370(1), 197-217.
\bibitem{nopoint}Han, R. (2018). Absence of point spectrum for the self-dual extended Harper's model. International Mathematics Research Notices, 2018(9), 2801-2809.

\bibitem{rui}R. Han. Shnol's theorem and the spectrum of long range operators. Proceedings of the AMS, 147(7) (2019), 2887-2897.

\bibitem{AA}R. Han. Cantor spectrum of AA-stacked graphene model in magnetic fields. In preparation.

\bibitem{HJ}Han, R., and Jitomirskaya, S. (2017). Full measure reducibility and localization for quasiperiodic Jacobi operators: A topological criterion. Advances in Mathematics, 319, 224-250.

\bibitem{HJY}R. Han, S. Jitomirskaya and F. Yang. Universal hierarchical structure of eigenfunctions in the Maryland model. (in preparation)

\bibitem{HYZ}Han, R., Yang, F., and Zhang, S. (2020). Spectral
  dimension for $\beta$-almost periodic singular Jacobi operators and
  the extended Harper’s model. Journal d'Analyse Mathématique, 142(2),
  605-666.

  \bibitem{j0}S. Jitomirskaya, Anderson Localization for the Almost Mathieu Equation: A Nonperturbative Proof. Comm.
Math. Phys. 165, 49-58 (1994)

\bibitem{jcongr1}S. Jitomirskaya, \textit{Almost Everything About the Almost Mathieu Operator}, II. Proceedings of XI International Congress of Mathematical Physics,Int. Press, 373-382, (1995).

\bibitem{j}S. Jitomirskaya, Metal-insulator transition for the almost
  Mathieu operator. Annals of Mathematics 150.3 (1999): 1159-1175.
  
  \bibitem{icm}S. Jitomirskaya, One-dimensional quasiperiodic
    operators: global theory, duality, and sharp analysis of small
    denominators. Proceedings of ICM 2022
\bibitem{jk} S. Jitomirskaya, S. Kocic. Spectral Theory of Schr\"odinger Operators over Circle Diffeomorphisms. International Mathematics Research Notices.

\bibitem {jks}S. Jitomirskaya, D. A. Koslover and M. S. Schulteis. Localization for a family of one-dimensional quasiperiodic operators of magnetic origin. Annales Henri Poincare 6.1 (2005): 103-124.

\bibitem{maryland}S. Jitomirskaya, W. Liu. Arithmetic spectral transitions for the Maryland model. CPAM, (2017), (70)no.6, 1025-1051. 

\bibitem{jl1}S. Jitomirskaya, W. Liu. Universal hierarchical structure  of quasiperiodic eigenfunctions. Annals of Mathematics (2) 187 (2018), no. 3, 721--776

\bibitem{jl2}S. Jitomirskaya, W. Liu. Universal reflective-hierarchical structure of quasiperiodic eigenfunctions and sharp spectral transitions in phase, (2018) arXiv preprint arXiv:1802.00781.


  \bibitem{pcmi}S. Jitomirskaya, W. Liu, S. Zhang, Arithmetic spectral
 transitions. IAS/Park City Mathematics Series: Harmonic Analysis  and Applications {\bf 27}, 35-72. 

\bibitem{Jm}S. Jitomirskaya, C. A. Marx, \textit{Analytic quasi-perodic cocycles with singularities and the Lyapunov Exponent of Extended Harper's Model}, Commun. Math. Phys. 316, 237 -- 267 (2012).

\bibitem{js}Jitomirskaya, S., and Simon, B., "Operators with singular continuous spectrum: III. Almost periodic Schr\"odinger operators." Communications in Mathematical Physics 165.1 (1994): 201-205.

\bibitem{JYsingular}Jitomirskaya, S., and Yang, F. (2017). Singular continuous spectrum for singular potentials. Communications in mathematical physics, 351(3), 1127-1135.

\bibitem{JYMaryland}Jitomirskaya, S. and Yang, F., 2021. Pure point spectrum for the Maryland model: a constructive proof. Ergodic Theory and Dynamical Systems, 41(1), pp.283-294.

\bibitem{ilya1}Kachkovskiy, I. (2019). Localization for quasiperiodic operators with unbounded monotone potentials. Journal of Functional Analysis, 277(10), 3467-3490.

\bibitem{ilya2}Kachkovskiy, I., Krymski, S., Parnovski, L., and Shterenberg, R. (2021). Perturbative diagonalization for Maryland-type quasiperiodic operators with flat pieces. Journal of Mathematical Physics, 62(6), 063509.

\bibitem{ilya3}Kachkovskiy, I., Parnovski, L., and Shterenberg, R. (2020). Convergence of perturbation series for unbounded monotone quasiperiodic operators. arXiv preprint arXiv:2006.00346.

\bibitem{KF}Klopp, F. and Fedotov, A.A., 2020. On the Hierarchical Behavior of Solutions of the Maryland Equation in the Semiclassical Approximation. Mathematical Notes, 108(5), pp.906-910.

\bibitem{liu2}W. Liu, Almost Mathieu operators with completely resonant phases. Ergodic Theory Dynam. Systems 40 (2020), no. 7, 1875–1893.

\bibitem{liu}W. Liu, Small denominators and large numerators of quasiperiodic Schr\"odinger operators. Preprint 2022.

\bibitem{liuyuan2}W. Liu, X. Yuan. Anderson localization for the completely resonant phases. J. Funct. Anal. 268 (2015), no. 3, 732–747.

\bibitem{liuyuan}W. Liu, X. Yuan. Anderson Localization for the almost Mathieu operator in exponential regime. Journal of Spectral Theory, 5.1 (2015):89-112.

\bibitem{JMreview}Marx, C. A., and Jitomirskaya, S. (2017). Dynamics and spectral theory of quasi-periodic Schr\"odinger-type operators. Ergodic Theory and Dynamical Systems, 37(8), 2353-2393.

\bibitem{wave}  R. Prange, D. Grempel, and S. Fishman. Wave functions at a mobility edge: an example of a singular continuous spectrum. Physical Review B 28.12 (1983): 7370.

\bibitem{s}Sarnak, P. (1982). Spectral behavior of quasi periodic potentials. Communications in Mathematical Physics, 84(3), 377-401.

\bibitem{simm}B. Simon. Almost periodic Schr\"odinger operators. IV. The Maryland model. Annals of Physics 159.1 (1985): 157-183.

\bibitem{thou}Thouless D.J., 1983. Bandwidths for a quasiperiodic tight-binding model. Phys. Rev. B 28, pp.4272-4276

\bibitem{gps}Wang, Y., Xia, X., You, J., Zheng, Z., and Zhou, Q. (2021). Exact mobility edges for 1D quasiperiodic models. arXiv preprint arXiv:2110.00962.

\bibitem{Xin}Zhao, X. (2021). H\"older continuity of absolutely continuous spectral measure for the extended HARPER’S model. Nonlinearity, 34(5), 3356-3372.

\bibitem{Yang}Yang, F. (2018). Spectral transition line for the extended Harper's model in the positive Lyapunov exponent regime. Journal of Functional Analysis, 275(3), 712-734.

\end{thebibliography}

\

\address{Rui Han, rhan@lsu.edu\\
Department of Mathematics,
Louisiana State University, Baton Rouge, USA}
\

\

\address{Svetlana Jitomirskaya, szhitomi@math.uci.edu\\ 
Department of Mathematics,
University of California, Irvine, California, USA}
\

\

\address{Fan Yang, ffyangmath@gmail.com\\
Department of Mathematics,
Louisiana State University, Baton Rouge, USA}


\end{document}